\documentclass[11pt]{article}
\usepackage{color}

\usepackage{amssymb}   %basic
\usepackage{amsthm}    %theoremstyle
\usepackage{amsmath}   %\mathcal
\usepackage{stmaryrd}  %\interleave
\usepackage{titletoc}  %\ttl
\usepackage{mathrsfs}  %\mathscr
\usepackage{graphicx}
\usepackage{multirow}
\usepackage{floatrow}
\newfloatcommand{capbtabbox}{table}[][\FBwidth]

\usepackage{blindtext}

\newlength{\defbaselineskip}
\newcommand{\setlinespacing}[1]%
           {\setlength{\baselineskip}{#1 \defbaselineskip}}

\theoremstyle{plain}
\newtheorem{thm}{Theorem}[section]

\newtheorem{lem}[thm]{Lemma}

\theoremstyle{definition}
\newtheorem{defn}{Definition}[section]
\newtheorem{ass}{Assumption}[section]
\newtheorem{rmk}{Remark}[section]

\newcommand{\eps}{\varepsilon}
\newcommand{\vrh}{\varrho}

\newcommand{\cR}{\mathcal{R}}

\newcommand{\cL}{\mathcal{L}}

\newcommand{\cB}{\mathcal{B}}

\newcommand{\cS}{\mathcal{S}}

\newcommand{\cF}{\mathcal{F}}
\newcommand{\cU}{\mathcal{U}}

\newcommand{\cY}{\mathcal{Y}}
\newcommand{\cZ}{\mathcal{Z}}

\newcommand{\bE}{\mathbb{E}}
\newcommand{\bP}{\mathbb{P}}

\newcommand{\bR}{\mathbb{R}}

\newcommand{\rd}{\mathrm{d}}
\newcommand{\sF}{\mathscr{F}}
\newcommand{\sP}{\mathscr{P}}

\newcommand{\la}{\langle}
\newcommand{\ra}{\rangle}

\newcommand{\ptl}{\partial}

\makeatletter\@addtoreset{equation}{section} \makeatother
 \allowdisplaybreaks
\begin{document}

\title{Numerical Approximations of Coupled Forward-Backward SPDEs
}

\author{  Hasib Uddin Molla\footnotemark[1] \and Jinniao Qiu\footnotemark[1] }
%\date{}
\footnotetext[1]{Department of Mathematics \& Statistics, University of Calgary, 2500 University Drive NW, Calgary, AB T2N 1N4, Canada. \textit{Email}: \texttt{jinniao.qiu@ucalgary.ca} (J. Qiu), \texttt{mdhasibuddin.molla@ucalgary.ca} (H. U. Molla). J. Qiu was partially supported by the National Science and Engineering Research Council of Canada and by the start-up funds from the University of Calgary. }

\maketitle

%----------------------
\begin{abstract}
We propose and study a scheme combining the finite element method and machine learning techniques for the numerical approximations of coupled nonlinear forward-backward stochastic partial differential equations (FBSPDEs) with homogeneous Dirichlet boundary conditions. Precisely, we generalize the pioneering work of Dunst and Prohl [\textit{SIAM J. Sci. Comp., 38(2017), 2725--2755}] by considering general \textit{nonlinear} and \textit{nonlocal} FBSPDEs with more inclusive coupling; self-contained proofs are provided and different numerical techniques for the resulting finite dimensional equations are adopted.     For such FBSPDEs, we first prove the existence and uniqueness of the strong solution as well as of the weak solution. Then the finite element method in the spatial domain leads to approximations of FBSPDEs by finite-dimensional forward-backward stochastic differential equations (FBSDEs) which are numerically computed by using some deep learning-based schemes.  The convergence analysis  is addressed for the spatial discretization of FBSPDEs, and the numerical examples, including both decoupled and coupled cases, indicate that our methods are quite efficient. 
\end{abstract}



\section{Introduction}
Let $(\Omega,\sF,\{\sF_t\}_{t\geq0},\bP)$ be a complete filtered probability space on which is defined a $k$-dimensional Wiener process $W=\{W_t:t\in[0,\infty)\}$ such that $\{\sF_t\}_{t\geq0}$ is the natural filtration generated by $W$ and augmented by all the
$\bP$-null sets in $\sF$. We denote by $\sP$ the $\sigma$-algebra of the predictable sets on $\Omega\times[0,T]$ associated with $\{\sF_t\}_{t\geq0}$. 

In this paper, we consider the following coupled forward and backward stochastic partial differential equations with homogeneous Dirichlet boundary conditions:
\begin{equation}\label{FSPDE}
  \left\{\begin{array}{l}
  \begin{split}
%  &\begin{split}
  \rd\rho(t,x)&=\bigg(\Delta \rho(t,x)+F\big(t,x,\rho(t),\nabla\rho(t),u(t),\nabla u(t),\psi(t)\big)\bigg)\,\rd t\\
  &\quad
  -\sum_{i=1}^{k}f^i\big(t,x,\rho(t),u(t)\big) \,\rd W_{t}^{i}, \quad (t,x)\in[0,T]\times D;\\
%  \end{split}\\
    \rho(0,x)&=\rho_0(x), \quad x\in D;\\
    \rho(t,x)\bigg|_{\partial D}&=0,\quad t\in[0,T],
    \end{split}
  \end{array}\right.
\end{equation}
and 
\begin{equation}\label{BSPDE}
\left\{\begin{array}{l}
\begin{split}
%&\begin{split}
-\rd u(t,x)&=\bigg(\Delta u(t,x)+G\big(t,x,\rho(t),\nabla\rho(t),u(t),\nabla u(t),\psi(t)\big)\bigg)\,\rd t \\
&\quad
-\sum_{i=1}^{k}\psi^i (t,x)\,\rd W_{t}^{i}, \quad
(t,x)\in[0,T]\times D;\\
%\end{split}\\
u(T,x)&=g(x,\rho(T)), \quad x\in D;\\
u(t,x)\bigg|_{\partial D} &=0,\quad t\in[0,T].
\end{split}
\end{array}\right.
\end{equation}
Here and throughout this paper, $D\subset\bR^d$ is a bounded domain with $C^2$ boundary $\ptl D$ and
$T\in(0,\infty)$ a finite deterministic time.

% The forward stochastic differential equation (FSDE) has applications in wide range of problems from various fields. Discretization and approximation of the solution of FSDE may be straightforward and explicit and implicit Euler schemes \cite{EulerMaryamaFSDE} are quite capable. Wide range of applications and other numerical approximation schemes of FSDEs can be found in \cite{KloedenPlaten_NumercApplSDE} for instance.

The forward-backward stochastic partial differential equation (FBSPDE) comprising of two equations like \eqref{FSPDE} and \eqref{BSPDE} arises naturally in many applications of probability theory and stochastic processes, for instance in the nonlinear filtering and stochastic control theory for processes with incomplete information, as the (usually coupled) system of the Duncan-Mortensen-Zakai filtration equation (or controlled SPDE) and its adjoint equation (for instance, see \cite{Bernsoussan,dunst2016forward,fuhrman2013stochastic,Hu_Ma_Yong02}); along this line, the study of FBSPDEs can date back to about forty years ago (see \cite{Bernsoussan}). On the other side, in the mean-field game theory certain classes of FBSPDEs \eqref{FSPDE}-\eqref{BSPDE} are raised as the mean-field game system with common noise; the former is a forward stochastic Kolmogorov equation describing the evolution of the conditional distributions of the states of the players given the common noise, while the  latter is the stochastic Hamilton-Jacobi-Bellman  equation characterizing the value function of the optimization problem when the flow of conditional distributions is given; more details are referred to \cite{bensoussan2017interpretation,CardaliaguetetelMFG2015,carmona2014master} for instance.

For the decoupled case when $F$ is independent of $(u,\nabla u,\psi)$ or $(G,g)$ is independent of $\rho$, the equations \eqref{FSPDE} and \eqref{BSPDE} may be solved separately and they have been extensively studied in the literature; see \cite{bayer2020pricing,conus2019weak,da2014stochastic,du2011strong,Kryov1996,LI2020124518,MaYong97BSPDE,PengBSPDEnonlin,qiu2018viscosity,wang2020l2} among many others. However, there are few results on the wellposedness of coupled FBSPDEs, let alone numerical approximations. Indeed, a class of fully coupled FBSPDEs on the whole space was studied in \cite{HongYin2014FBSPDE} where the FBSPDEs are viewed as natural extensions of  (finite dimensional) forward-backward stochastic differential equations (FBSDEs) and the existence and uniqueness of weak solution (in the PDE/SPDE theory) is proved in the spirit of approaches for FBSDEs, while in \cite{CardaliaguetetelMFG2015} the wellposedness in H\"older spaces is addressed for a class of FBSPDEs with linear coefficient $f$ and periodic boundary conditions under certain strong assumptions; meanwhile, numerical methods for a special class of coupled linear FBSPDEs with
\begin{align*}
F\big(t,x,\rho(t),\nabla\rho(t),u(t),\nabla u(t),\psi(t)\big) &=  u(t), \quad 
f^i(t,x,\rho(t),u(t))=-\nu^i (t)\rho(t),\quad i=1,\dots, k,\\
G\big(t,x,\rho(t),\nabla\rho(t),u(t),\nabla u(t),\psi(t)\big) &= \sum_{i=1}^k \nu^i(t)\psi^i(t) + h(t),\quad
G(x,\rho(T))= \Psi(x),
\end{align*}
 may be found in the pioneering work of Dunst and Prohl   \cite{dunst2016forward} where the convergence analysis is established with finite element method for spatial discretization and the least square Monte Carlo simulation mixed with Picard type iterations or stochastic gradient method for the approximations of the resulting (finite-dimensional) FBSDEs. 
 
In this work, we consider coupled FBSPDEs like  \eqref{FSPDE}-\eqref{BSPDE} with homogeneous Dirichlet boundary conditions where coefficients may be \textit{nonlinear} and \textit{nonlocal}. The existence and uniqueness of strong solution of coupled FBSPDEs is derived under Lipschitz conditions. For numerical simulations,  the coupled FBSPDE is discretized in spatial domain with finite element method, which results in finite dimensional coupled FBSDEs in temporal domain.  We  address the wellpossedness of such FBSDEs as well as the convergence rate for the spatial discretization. Finally, the resulting FBSDEs are numerically computed with  some existing deep learning-based schemes and we present two numerical examples which include both decoupled and coupled cases showing the efficiency of our methods. The approaches mix the existing probability theory and stochastic analysis, (S)PDE theory, and the numerical analysis in both  deterministic and stochastic settings.

To overcome the so-called curse of dimensionality, several deep learning-based algorithms have been proposed and studied for numerical computations of partial differential equations (PDEs); see \cite{weinan2017deep,han2018deep,Han&Long2019coupleddeep,hure2019deep,JiPeng2020CoupledBSDE} among many others. As these deep learning schemes are based on the equivalence representation relationship between \textit{deterministic} PDEs and associated \textit{Markovian} FBSDEs, such numerical methods for PDEs and FBSDEs are one and the same. This paper extends the applications of these numerical methods (or their modifications) to FBSPDEs like \eqref{FSPDE}-\eqref{BSPDE} that may be coupled, nonlinear, and/or nonlocal; nevertheless, the high-dimensionality is not due to the spatial domain of  \eqref{FSPDE}-\eqref{BSPDE} but from the resulting finite-dimensional FBSDEs after the spatial discretization of FBSPDE \eqref{FSPDE}-\eqref{BSPDE} with finite element methods, and the deep learning schemes are used to numerically compute solutions of these approximating (finite-dimensional) FBSDEs. On the other hand, many of such FBSPDEs arise from the non-Markovian type stochastic controls/games, with the associated representation systems (FBSDEs) being non-Markovian and even of McKean–Vlasov type (see \cite{bensoussan2017interpretation,CardaliaguetetelMFG2015,carmona2014master,PengBSPDEnonlin,qiu2018hormander,tang2005semi}), and this incurs the inapplicability of the existing deep learning-based methods that are only working under Markovian framework. Because of this, we adopt in this work the strategy: first discretize the FBSPDE \eqref{FSPDE}-\eqref{BSPDE} in spatial domain and then numerically compute the resulting finite-dimensional FBSDEs with the existing deep learning methods.

 The rest of this paper is organized as follows. In section 2, we give the notation and assumptions as well as a brief introduction on the finite element methods and the deep neural networks. Section 3 is devoted to the wellposedness of coupled FBSPDEs for both weak and strong solutions. Then the rate of convergence for semi-discrete approximations is proved in Section 4 where the wellposedness of the finite dimensional approximating FBSDEs is also addresed. In Section 5, we introduce and discuss three different deep learning-based methods for the numerical approximations of Markovian FBSDEs. Finally, two numerical examples are presented in Section 6 and the proof of Lemma \ref{lem1} is given in the appendix.
%%%%%%%%%%%%%%%%%%%%%%%%%%%%%%%%%%%%%%%%%%%%%%%%%%%%%%%%%%%%%%%%%%%%%%%%%%%%%%%%%%%%%%%%%%%%%%%%%%%%%%%%%%%%%%%%%%%%%%%%%%%%%%%%%%%%%%%%%%%%%%%%%%%%%%%%%%%%%%%%%%%%%%%%%%%%%%%%%%%%%%%%%%%%%%%%%%%%%%%%%%%%%%%%%%%%%%%%%%%%%%%%%%%%%%%%%%%%%%%%%%%%%%%%%%%%%%%%%%%%%%%%%%%%%%%%%%%%%%%%%%%%%%%%%%%%%%%%%%%%%%%%%%%%%%%%%%%%%%%%%%%%%%%%%%%%%%%%%%%%%%%%%%%%%%%%%%%%%%%%%%%%%%%%%%%%%%%%%%%%%%%%%%%%%%%%%%%%%%%%%%%%%%%%%%%%%%%%%%%%%%%%%%%%%%%%%%%%%%%%%%%%%%%%%%%%%%%%%%%%%%%%%%%%%%%%%%%%%%%%%%%%%%%%%%%%%%%%%%%%%%%%%%%%%%%%%%%%%%%%%%%%%%%%%%%%%%%%%%%%%%%%%%%%%%%%%%%%%%%%

\section{Preliminary}

\subsection{Notations and assumptions}

Denote by $|\cdot|$ the norm in  Euclidean spaces. For each $l\in \mathbb{N}^+$ and domain $D\subset \bR^d$, denote by $C_c^{\infty}(D;\bR^l)$ the space of infinitely differentiable functions $f:D\rightarrow \bR^l$ with compact supports in $D$. We write $C_c^{\infty}:=C_c^{\infty}(\bR^l)=C_c^{\infty}(D;\bR^l)$ when there is no confusion on the dimension. The Lebesgue measure in $\bR^d$ will be denoted by $\rd x$. Also when there is no confusion on the dimension we write $L^2:=L^2(\bR^l)=L^2(D;\bR^l)$ for the usual Lebesgue integrable space with scalar product and norm defined
$$
\langle \phi,\,\psi\rangle=\sum_{j=1}^l\int_{D}\phi^j(x)\psi^j(x)dx,\quad \|\phi\|=\langle\phi,\,\phi\rangle^{1/2},\,\,\forall
\phi,\psi\in L^2(D;\bR^l).
$$
In addition, for each $(n,p)\in\bR\times \left[1,\infty\right]$ we define  the $n$-th order Sobolev space $(H^{n,p}(D;\bR^l),\|\cdot\|_{n,p}) $ as usual; for simplicity, we may write $H^{n,p}(D;\bR^l)$ as $H^{n,p}$ when there is no ambiguity about the dimension and domain. Denote by $H^{1,2}_0$ the space of $f\in H^{1,2}$ with vanishing traces on $\partial D$, i.e., $H^{1,2}_0=\{f\in H^{1,2}: f\cdot1_{\partial D}=0\}$. Write $H^{2,2}_0=H^{2,2}\cap H^{1,2}_0$.

Let $V$ be a Banach space equipped with norm $\|\cdot\|_V$. For $p\in[1,\infty)$, $\cS ^p (V)$ is the set of all the $V$-valued,
$(\sF_t)$-adapted and continuous processes $\{X_{t}\}_{t\in [0,T]}$ such
that
\begin{equation*}
\|X\|_{\cS^p(V)}:=\bigg\|\max_{t\in [0,T]}\|X(t)\|_V\bigg\|_{L^p(\Omega,\cF,\bP)}<\infty.
\end{equation*}
Denote by $\mathcal{L}^p(V)$ the space of all  the $V$-valued,
$(\sF_t)$-adapted processes $\{X_{t}\}_{t\in [0,T]}$ such
that
\begin{equation*}
\|X\|_{\cL^p(V)}:=\bigg(\bE\left[\int_{0}^{T}\|X(t)\|^p_V \rd t\right]\bigg)^{\frac{1}{p}}<\infty.
\end{equation*}
Obviously, $\left(\cS^p(V),\,\|\cdot\|_{\cS^p(V)}\right)$ and $\left(\mathcal{L}^p(V),\|\cdot\|_{\mathcal{L}^p(V)}\right)$
are Banach spaces. By convention, we treat elements of spaces like $\cS^2(H^{n,2})$ and $\cL^2(H^{n,2})$ as functions rather than distributions or classes of equivalent functions, and if a function of such class admits a version with better properties, we always denote this version by itself. For example, if $u\in \cL^2(H^{n,2})$ and $u$ admits a version lying in $\mathcal{S}^2(H^{n,2})$, we always adopt the modification $u\in \cL^2(H^{n,2})\cap \cS^2(H^{n,2})$.\\[3pt]

For the FBSPDE \eqref{FSPDE}-\eqref{BSPDE},  following are the assumptions we use throughout this paper.
\begin{ass} \label{ass1}
	\hfill 
\begin{enumerate}
	\item[(a)] For each 
	$(\rho,\bar{\rho},u, \bar{u},\psi)\in L^2(\bR)\times L^2(\bR^d)\times L^2(\bR)\times L^2(\bR^d)\times L^2(\bR^k)$, the function
	\begin{equation*}
	F(\omega,t,x,\rho,\bar{\rho},u, \bar{u},\psi):~\Omega\times[0,T]\times D\rightarrow\bR
	\end{equation*}
	is $\sP\otimes\cB(D)$-measurable. There exist positive constants $L^F_1,L^F_2$
	such that for all \\$(\rho_1,\bar{\rho_1},u_1, \bar{u_1},\psi_1)$, $(\rho_2,\bar{\rho_2},u_2, \bar{u_2},\psi_2)\in L^2(\bR)\times L^2(\bR^d)\times L^2(\bR)\times L^2(\bR^d)\times L^2(\bR^k)$
	and $(\omega,t)\in \Omega\times[0,T]$,
	\begin{align*}
	&\|F(\omega,t,\rho_1,\bar{\rho_1},u_1, \bar{u_1},\psi_1)-F(\omega,t,\rho_2,\bar{\rho_2},u_2, \bar{u_2},\psi_2)\|\\
	&\leq L^F_1\bigg(\|\rho_1-\rho_2\|+\|\bar{\rho_1}-\bar{\rho_2}\|\bigg)+L^F_2\bigg(\|u_2-u_1\|+\|\bar{u_2}- \bar{u_1}\|+\|\psi_2-\psi_1\|\bigg).
	\end{align*}\rm
	\item[(b)] For each $(\rho,u)\in L^2(\bR)\times L^2(\bR)$, the function
	\begin{equation*}
	f(\omega,t,x,\rho,u):~\Omega\times[0,T]\times D \rightarrow\bR^k
	\end{equation*}
	is $\sP\otimes\cB(D)$-measurable. There exists positive constants $L_1^f,L_2^f$ and $\tilde{L}^f$
	such that for all $\rho_1,\rho_2,u_1,u_2\in L^2(\bR)$ and $\rho,u\in H^{1,2}_0$
	and $(\omega,t)\in \Omega\times[0,T]$,
	\begin{equation*}
	\begin{split}
	\|f(\omega,t,\rho_1,u_1)&-f(\omega,t,\rho_2,u_2)\|\leq L_1^f\|\rho_1-\rho_2\|+L_2^f\|u_1-u_2\|,\\
	\|f(\omega,t,\rho,u)\|&_{1,2}\le \tilde{L}^f\big(1+\|\rho\|_{1,2}+\|u\|_{1,2}\big).
	\end{split}
	\end{equation*}\rm
	\item[(c)] 	$\rho_0\in  L^2(\Omega;\sF_0, H^{1,2}_0)$ and $F^0_t\in \cL^2(L^2),f^0_t\in\cL^2(H^{1,2}_0)$ where
	$$
	F_t^0=F(\omega,t,0,0,0,0,0),f_t^0=f(\omega,t,0,0); \quad (\omega,t)\in \Omega\times [0,T].
	$$
\end{enumerate}
      
\end{ass}

%\begin{ass}   \label{ass2}
%    There exists an increasing function $\kappa:
%    [0,\infty)\rrow[0,\infty)$ such that $k(s)\downarrow 0$ as $s\downarrow0$
%    and
%    \begin{equation}
%     \sum_{i,j=1}^d|a^{ij}(t,x)-a^{ij}(t,y)|
%     +\sum_{i=1}^d\sum_{k=1}^m|\sigma^{ik}(t,x)-\sigma^{ik}(t,y)|\leq
%            \kappa(|x-y|)
%    \end{equation}
%    holds almost surely for all $(t,x,y)\in [0,T]\times \bR^d\times \bR^d.$
%\end{ass}

%\begin{ass}     \label{ass2}
%     The functions $a^{ij}(t,x)$ and $\sigma^{ik}(t,x)$ are real-valued
%     $\sP\times\cB (\bT^{d})$-measurable, such that they are continuously %differentiable with respect to $x$ and for any $i,j=1,\cdots,d,k=1,\cdots,m$,
%\begin{equation}
%  \begin{split}
%    \esssup_{(\omega,t,x)\in\Omega\times [0,T] \times \bT^d}
%     \left\{    \left|  a^{ij}(t,x)  \right|  +\left|  \sigma^{jk}(t,x)  \right| %+\left|  Da^{ij}(t,x) \right|  +  \left|  D\sigma^{ij}(t,x)  \right|  \right\}
%     \leq \Lambda,
%  \end{split}
%\end{equation}
%$$\forall i,j,k ~\textrm{and}~ \forall \omega\in\Omega,t\geq 0.$$
%\end{ass}

\begin{ass}\label{ass2}
	\hfill
	\begin{enumerate}
	\item[(a)] For each 
	$(\rho,\bar{\rho},u,\bar{u},\psi)\in L^2(\bR)\times L^2(\bR^d)\times L^2(\bR)\times L^2(\bR^d)\times L^2(\bR^k)$, the function
\begin{equation*}
G(\omega,t,x,\rho,\bar{\rho},u,\bar{u},\psi):~\Omega\times[0,T]\times D\rightarrow\bR
\end{equation*}
is $\sP\otimes\cB(D)$-measurable. There exist positive constants $L^G_1,L^G_2$
such that for all \\$(\rho_1,\bar{\rho_1},u_1, \bar{u_1},\psi_1)$, $(\rho_2,\bar{\rho_2},u_2, \bar{u_2},\psi_2)\in L^2(\bR)\times L^2(\bR^d)\times L^2(\bR)\times L^2(\bR^d)\times L^2(\bR^k)$
and $(\omega,t)\in \Omega\times[0,T]$,
\begin{equation*}
\begin{split}
&\|G(\omega,t,\rho_1,\bar{\rho_1},u_1,\bar{u_1},\psi_1)-G(\omega,t,\rho_2,\bar{\rho_2},u_2,\bar{u_2},\psi_2)\|\\
&\leq L^G_1\bigg(\|\rho_1-\rho_2\|+\|\bar{\rho_1}-\bar{\rho_2}\|\bigg)+L^G_2\bigg(\|u_2-u_1\|+\|\bar{u_2}-\bar{u_1}\|+\|\psi_2-\psi_1\|\bigg).
\end{split}
\end{equation*}\rm
		\item[(b)] For each $\rho\in L^2(\bR)$, the function
		\begin{equation*}
		g(\omega,x,\rho):~\Omega\times D \rightarrow\bR
		\end{equation*}
		is $\cF_T\otimes\cB(D)$-measurable. There exist positive constants $L^g$ and $\tilde{L}^g$
		such that for all $\rho_1,\rho_2\in L^2$, $\rho_3\in L^2(\Omega;\sF_T,H^{1,2}_0)$
		and $\omega\in \Omega$,
		\begin{equation*}
		\begin{split}
		\|g(\omega,\rho_1)&-g(\omega,\rho_2)\|\leq L^g\|\rho_1-\rho_2\|,
		\end{split}
		\end{equation*}\rm
		and
		\begin{equation*}
		\|g(\omega,\rho_3)\|_{1,2}\le \tilde{L}^g\bigg(1+\|\rho_3\|_{1,2}\bigg).
		\end{equation*}
		\item[(c)] $g^0\in  L^2(\Omega;\sF_T, H^{1,2}_0)$ and $G_t^0\in \cL^2(L^2)$ where
		$$
		g^0=g(\omega,0),G_t^0=G(\omega,t,0,0,0,0,0); \quad (\omega,t)\in \Omega\times [0,T].
		$$
	\end{enumerate}
\end{ass}

\begin{rmk}\label{rmk-nonlocal}
The above assumptions regarding Lipschitz continuity and linear growth are more or less standard. However, it is worth noting that the dependence of coefficients $F$, $G$, and $g$ on the unknown random fields may be nonlocal; for instance, the assumption on $g$ covers some classes of functions of the following form:
\begin{align*}
g(x, \rho(T))=\int_{D} h(x,y) \rho(T,y)\, dy,\quad \text{for } x\in D, \text{ given function }h:\, D\times D\rightarrow \bR. 
\end{align*}
Such nonlocal dependence is substantially demanding in the mean-field game systems with common noise (see \cite{bensoussan2017interpretation,CardaliaguetetelMFG2015,carmona2014master} for instance).  In the existing literature, this nonlocal dependence is not taken into account in \cite{dunst2016forward,HongYin2014FBSPDE}. In the theory of FBSPDEs for solutions in H\"older spaces in \cite{CardaliaguetetelMFG2015}, the nonlocal dependence is demanded and allowed; nevertheless, the FBSPDEs therein are equipped with linear coefficients $f^j(t,x,\rho(t),u(t))=\nu^j \sum_{i=1}^d \nabla_i \rho(t,x)$ for some constants $\nu^j$, $j=1,\dots,k$, and the associated stochastic integral can be and is actually disappeared in \cite{CardaliaguetetelMFG2015} by using the It\^o-Kunita-Wentzell formula.
\end{rmk}

%\section{Finite Element Approximations and Deep Neural Networks}
\subsection{Finite Element Approximations}
Let $\mathcal{T}_h=\{K\}$ be a triangulation of the convex polyhedral domain $D\subset \mathbb{R}^d$ into regular simplicial elements $K$ with $h:=\max{\{diam(K):K\in\mathcal{T}_h\}}$. The intersection of two different elements is either empty, or a vertex, or an entire edge of both elements. $\mathcal{T}_h$ is locally quasi-uniform, i.e., each element contains a ball of radius $c_1h$ and is contained in a circle of radius $c_2h$, where the constants $c_1>0$ and $c_2<\infty$ do not depend on $K$ or $h$. 
For each element $K\in\mathcal{T}_h$, let $\mathcal{P}^q(K)$ be the set of all polynomials of degree less than or equal to $q$. 

Now we define the finite dimensional space $V_h^0\subset H_0^{1,2}$ consisting of the continuous piecewise linear functions on $\mathcal{T}_h$ by
\begin{center} 
$V_h^0=\{\phi:\phi|_K\in\mathcal{P}^1(K)\:\forall\:K\in\mathcal{T}_h, \phi$ is continuous on $D$ and $\phi=0$ on $\ptl D\}$. 
\end{center} 
Let $\{N_1,\cdots,N_L\}$ be an enumeration of internal nodes of $\mathcal{T}_h$ and the space $V_h^0$ is spanned by the set of nodal basis functions $\{\phi_h^1,\cdots,\phi_h^L\}$. 
By $\Pi_h:L^2\rightarrow V_h^0$, we denote the $L^2$-projection of a given function $\xi$ onto finite dimensional space, i.e., $\langle \Pi_h\xi-\xi,\phi_h\rangle=0$ for all $\phi_h \in V_h^0$. The Ritz projection $\cR_h:H^{1,2}_0\rightarrow V_h^0$ is defined by $\langle \nabla[\cR_h\xi-\xi],\nabla\phi_h\rangle=0$ for all $\phi_h \in V_h^0$. Discrete Laplace operator $\Pi_h\Delta\equiv \Delta_h:V_h^0\rightarrow V_h^0$ is given by $-\langle \Delta_h\xi_h,\phi_h\rangle=\langle\nabla\xi_h,\nabla\phi_h\rangle, \; \forall \:\phi_h,\xi_h\in V_h^0$.\\
Following are some standard results about the stability of $L^2$-projection onto finite element spaces and the associated approximation error estimates; refer to \cite{egger2015ErrEstFE} for instance.

%These results and their proofs can be found in \cite{SchatzWahlbin1978ErrEstFE}, \cite{RannacherScott1982ErrorEstFE}, \cite{eriksson1996computational}, \cite{Bramble2001stabilityL2}, \cite{BrenerScott2008FETheory}, \cite{Bank&Yserentant2014StabilityL2}, \cite{egger2015ErrEstFE} and \cite{Hengguang2017Stability}  among many others. 

\begin{thm}\label{FEThm}
	In the following assertions, the constant $C_e$ only depends on the domain and the regularity constants of the mesh but does not depend on $h$:
	\begin{enumerate}
		\item [(a)] The $L^2$-projection is stable on $L^2$, i.e., for all $\xi\in L^2$ we have
		$\|\Pi_h\xi\|\le\|\xi\|$. For locally quasi-uniform mesh, the $L^2$-projection is also stable on $H^{1,2}_0$, i.e., for all $\xi\in H^{1,2}_0$ we have 
		$\|\Pi_h\xi\|_{1,2}\le C_e\|\xi\|_{1,2}$. %\cite{egger2015ErrEstFE} [Lemma 4].
		\item [(b)] The $L^2$-projection of a function into finite element space $V^0_h$ is the best approximation in $V_h^0$, i.e., $\|\xi-\Pi_h\xi\|\le\|\xi-\phi_h\|$ for all $\phi_h\in V_h^0$.
		\item [(c)] The error estimates for $L^2$-projection are given as follow: %\cite{egger2015ErrEstFE} [Lemma 5]
		\begin{equation*}
		\begin{split}
		\|\xi-\Pi_h\xi\|&\le C_eh\|\xi\|_{1,2},\quad \forall\, \xi\in H^{1,2}_0; \\
		\|\xi-\Pi_h\xi\|&\le C_eh^2\|\xi\|_{2,2},\quad \forall\, \xi\in H^{2,2}_0;\\
		\|\xi-\Pi_h\xi\|_{1,2}&\le C_eh\|\xi\|_{2,2},\quad \forall\, \xi\in H^{2,2}_0.
		\end{split}
		\end{equation*}
	\end{enumerate}
\end{thm}

\subsection{Deep Neural Networks}
Deep learning provides a very powerful framework for high dimensional function approximation.  In what follows, we shall introduce the architecture of deep neural networks and associated universal approximation results.
%With enough layers and units (neurons) within a layer, a deep neural network is able to represent functions with increasing complexity. 

%\subsubsection{Architecture of Deep Neural Networks}
%Neuron is the basic building block for any neural networks. Say, a neuron takes $x\in\bR^{d_{inp}}$ as input and produce a single output $y\in\bR$. Then operation that takes place inside that neuron is
%\begin{equation*}
%y=\alpha(\omega\cdot x+\beta).
%\end{equation*}
%Here weight $\omega\in\bR^{d_{inp}}$ and bias $\beta\in\bR$ are called parameters and $\alpha:\bR\rightarrow\bR$ is a non-linear function known as activation function. Activation function is usually used to turn an unbounded input into a bounded one. Rectified linear units and their several generalizations, logistic sigmoid and hyperbolic tangent are most commonly used activation functions \cite{Goodfellow-et-al-2016}.
%\par In a neural network, neurons are organized into several layers. First layer is the input layer and last one is the output layer. Any other layers between the input and the output layer are hidden layer. Number of hidden layers determines the depth of the model. Usually a neural network can have any number of layers and any number of neurons in those layers. In a fully-connected feed-forward neural network, every neuron of a certain layer takes as input the output from all neurons of the previous layer and then pass its output to all neurons in the next layer.
%\par 
Consider a deep neural network with input dimension $d_i$, output dimension $d_o$, number of layers $N+1\in\mathbb{N}\diagdown\{1,2\}$, and number of neurons $m_n,n=0,\cdots,N$ on each layer. Here, $m_0=d_i$, $m_N=d_o$ and for simplicity we choose an identical number of neurons for all hidden layers, that is, $m_n=m,n=1,\cdots,N-1$. Then a feed-forward neural network may be thought of as a function from $\bR^{d_i}$ to $\bR^{d_o}$ defined by compositions of simple functions as
\begin{equation}\label{nnfunc}
x\in\bR^{d_i}\mapsto A_N\;o\;\alpha\;o\;A_{N-1}\;o\cdots o\;\alpha\;o\;A_1(x)\in\bR^{d_o},
\end{equation}
where $f_1of_2(x)=f_1(f_2(x))$.
Here, $A_1:\bR^{d_i}\mapsto\bR^m,A_N:\bR^{m}\mapsto\bR^{d_o}$, and $A_n:\bR^{m}\mapsto\bR^m,n=2,\cdots,N-1$ are affine transformations that take place inside a whole layer and defined by
\begin{equation*}
A_n(x)=\mathcal{W}_nx+\beta_n. 
\end{equation*}
Here matrix $\mathcal{W}_n$ and vector $\beta_n$ are called weight and bias respectively for the $n$th layer of the network. For the last layer we choose the identity function as activation function, and the activation function $\alpha$ here applied component-wise on the outputs of $A_n$.

We denote by $\theta=(\mathcal{W}_n,\beta_n)_{n=1}^N$ the parameters of neural network. Given $d_i,d_o,N$ and $m$, the total number of parameters in a network is $N_{\theta}=\sum_{n=0}^{N-1}(m_n+1)m_{n+1}=(d_0+1)m+(m+1)m(N-2)+(m+1)d_1$ and thus $\theta\in\bR^{N_{\theta}}$. Let $\Theta$ be the set of all possible values of $\theta$ and if there are no constraint on parameters then $\Theta=\bR^{N_{\theta}}$. By $\mathcal{X}^{\mathcal{N}}(\cdot;\theta)$ we denote the neural network function defined in (\ref{nnfunc}) and the set of all such neural networks $\mathcal{X}^{\mathcal{N}}(\cdot;\theta),\theta\in\Theta$ within a fixed structure determined by $d_i,d_o,N,m$ and $\alpha$ is denoted by $\mathcal{NN}_{d_i,d_o,N,m}^{\alpha}(\Theta)$. 
%\subsubsection{Training of Deep Neural Networks}
%Say we want to use the neural network function $\mathcal{X}^{\mathcal{N}}(x;\theta)$ to approximate some function $f(x)$. During training, neural network $\mathcal{X}^{\mathcal{N}}(\cdot;\theta)$ learns the value of the parameter $\theta$ that results in best function approximation. Usually neural networks are trained by iterative, gradient based optimizer. 
%\subsubsection{Universal Approximation Property}

Deep neural networks are very efficient for approximations of functions even in high-dimensional spaces. The following fundamental result is  from \cite{hornik1989multilayer,hornik1990universal}:
\begin{thm}[Universal Approximation Therorem] If $\alpha$ is continuous and non-constant, it holds that:
\begin{enumerate}
	\item[(i)] The set $\cup_{m\in\mathbb{N}}\mathcal{NN}_{d_i,d_o,N,m}^{\alpha}(\Theta)(\bR^{N_{\theta}})$ is dense in $L^2(\nu)$ for any finite measure $\nu$ on $\bR^{d_i}$.
	\item[(ii)] If we further have $\alpha\in C^k$, then $\cup_{m\in\mathbb{N}}\mathcal{NN}_{d_0,d_1,2,m}^{\alpha}(\bR^{2_{\theta}})$ approximate any function and its derivatives up to order $k$, arbitrary well on any compact set of $\bR^{d_i}$.
\end{enumerate}
\end{thm} 
%Following result is also due to \cite{hornik1990universal}~ for the derivatives in case of smooth activation function and a single hidden layer, that is, $N=2$:
%\begin{thm}[Universal Approximation Therorem (II)]
%	Assume that $\alpha$ is a non-constant $C^k$ function. Then $\cup_{m\in\mathbb{N}}\mathcal{NN}_{d_0,d_1,2,m}^{\alpha}(\bR^{N_{\theta}})$ approximate any function and its derivatives up to order $k$, arbitrary well on any compact set of $\bR^{d_i}$.
%\end{thm}
%%%%%%%%%%%%%%%%%%%%%%%%%%%%%%%%%%%%%%%%%%%%%%%%%%%%%%%%%%%%%%%%%%%%%%%%%%%%%%%%%%%%%%

%%%%%%%%%%%%%%%%%%%%%%%%%%%%%%%%%%%%%%%%%%%%%%%%%%%%%%%%%%%%%%%%%%%%%%%%%%%%%%%%%%%%%%%%%%%%%%%%%%%%%%%%%%%%%%%%%%%%%%%%%%%%%%%%%%%%%%%%%%%%%%%%%%%%%%%%%%%%%%%%%%%%%%%%%%%%%%%%%%%%%%%%%%%%%%%%%%%%%%%%%%%%%%%%%%%%%%%%%%%%%%%%%%%%%%%%%%%%%%%%%%%%%%%%%%%%%%%%%%%%%%%%
%\section{Preliminaries}
\section{Wellposedness of FBSPDEs}
The wellposedness of FBSPDE \eqref{FSPDE}-\eqref{BSPDE} will be addressed in this section, before which we first introduce the definitions of weak and strong solutions.
\begin{defn}\label{defnFsoln}
	Given $(u,\psi)\in \bigg(\cL^2(H^{1,2}_0)\cap\cS^2(L^2)\bigg)\times\cL^2(L^2)$, the random function $ \rho\in \cL^2(H^{1,2}_0)\cap\cS^2(L^2)$ is said to be a \textit{weak} solution to FSPDE \eqref{FSPDE} if
	for each $\phi \in C^{\infty}_{c}$, the equality
	\begin{equation}
	\begin{split}
	\langle \rho(t),\phi\rangle=\la\rho_0,\phi\ra&+\int_{0}^{t}\bigg(-\langle\nabla \rho(t),\nabla\phi\rangle+\langle F(t,x,\rho,\nabla\rho,u,\nabla u,\psi),\phi\rangle\bigg)\,
	\mathrm{d}t\\
	&-\int_{0}^{t}\sum_{i=1}^{k}\langle f^i( t,\rho(t),u(t)),\phi\rangle \mathrm{d}W^i(t)
	\end{split}
	\end{equation}
	holds for all $t\in[0,T]$ with probability 1. If we further have 
	$\rho\in \cL^2(H^{2,2}_0)\cap\cS^2(H^{1,2}_0)$ for given $ (u,\psi)\in \bigg(\cL^2(H^{2,2}_0)\cap\cS^2(H^{1,2}_0)\bigg)\times\cL^2(H^{1,2}_0)$, then the solution $\rho$ is called a \textit{strong} solution.
\end{defn}
\begin{defn}\label{defnBsoln}
	Given $ \rho\in \cL^2(H^{1,2}_0)\cap\cS^2(L^2)$, the pair $ (u,\psi)\in \bigg(\cL^2(H^{1,2}_0)\cap\cS^2(L^2)\bigg)\times\cL^2(L^2)$ is called a \textit{weak} solution to BSPDE \eqref{BSPDE} if for each $\phi \in C^{\infty}_{c}$, the equality
	\begin{equation}
	\begin{split}
	\langle u(t),\phi\rangle=\la g,\phi\ra&+\int_{t}^{T}\bigg(-\langle\nabla u(t),\nabla\phi\rangle+\langle G(t,x,\rho,\nabla\rho,u,\nabla u,\psi),\phi\rangle\bigg)\,\mathrm{d}t\\
	&-\int_{t}^{T}\sum_{i=1}^{k}\langle\psi^i(t) ,\phi\rangle \mathrm{d}W^i(t)
	\end{split}
	\end{equation}
	holds for all $t\in[0,T]$ with probability 1. If we further have $ (u,\psi)\in (\cL^2(H^{2,2}_0)\cap\cS^2(H^{1,2}_0))\times\cL^2(H^{1,2}_0)$  for given $\rho\in \cL^2(H^{2,2}_0)\cap\cS^2(H^{1,2}_0)$, the solution pair $(u,\psi)$ is called a \textit{strong} solution.
\end{defn}
\begin{defn}\label{defnFBsystem}
	The tuple $ (\rho, u,\psi)\in \bigg(\cL^2(H^{1,2}_0)\cap\cS^2(L^2)\bigg)\times\bigg(\cL^2(H^{1,2}_0)\cap\cS^2(L^2)\bigg)\times\cL^2(L^2)$ is called a \textit{weak} solution to FBSPDE \eqref{FSPDE}-\eqref{BSPDE} if
	for each $\phi \in C^{\infty}_{c},$ the equalities
	\begin{equation}
	\begin{split}
	\langle \rho(t),\phi\rangle=\la\rho_0,\phi\ra&+\int_{0}^{t}\bigg(-\langle\nabla \rho(t),\nabla\phi\rangle+\langle F(t,x,\rho,\nabla\rho,u,\nabla u,\psi),\phi\rangle\bigg)\,
	\mathrm{d}t\\
	&-\int_{0}^{t}\sum_{i=1}^{k}\langle f^i( t,\rho(t),u(t)),\phi\rangle \mathrm{d}W^i(t),
	\end{split}
	\end{equation}
	and
	\begin{equation}
	\begin{split}
	\langle u(t),\phi\rangle=\la g,\phi\ra&+\int_{t}^{T}\bigg(-\langle\nabla u(t),\nabla\phi\rangle+\langle G(t,x,\rho,\nabla\rho,u,\nabla u,\psi),\phi\rangle\bigg)\,\mathrm{d}t\\
	&-\int_{t}^{T}\sum_{i=1}^{k}\langle\psi^i(t) ,\phi\rangle \mathrm{d}W^i(t),
	\end{split}
	\end{equation}
	hold for all $t\in[0,T]$ with probability 1. If we further have $ (\rho,u,\psi)\in \bigg(\cL^2(H^{2,2}_0)\cap\cS^2(H^{1,2}_0)\bigg)\times\bigg(\cL^2(H^{2,2}_0)\cap\cS^2(H^{1,2}_0)\bigg)\times\cL^2(H^{1,2}_0)$, the solution tuple $(\rho,u,\psi)$ is called a \textit{strong} solution.
\end{defn}

\begin{thm}\label{Thm1}
	Let assumptions \ref{ass1}-\ref{ass2} hold. Then there exists $\bar{C}=\bar{C}(L_1^f,L_1^F,L^G_2)$ such that if 
	\begin{align}
\bar{C}e^{\bar{C}T}\cdot\max{\{|{L_1^G}|^2T^2+|{L^g}|^2,|{L_1^G}|^2T\}}\cdot T\cdot\max \{|L_2^F|^2,|L_2^F|^2T+|L_2^f|^2\}<1, \label{condtn-wellp}
	\end{align}
	 FBSPDE \eqref{FSPDE}-\eqref{BSPDE} admits a unique weak solution $(\rho,u,\psi)$, with 
	\begin{equation*}
	\begin{split}
	&\bE\bigg[\sup_{t\in [0,T]}\|\rho(t)\|^2+\sup_{t\in [0,T]}\|u(t)\|^2\bigg]+\bE\int_{0}^{T}\left(\|\nabla\rho(s)\|^2+\|\nabla u(s)\|^2+\|\psi(s)\|^2\right)\rd s\\
	&\le C\bE\left[ \|\rho_0\|^2+\|g^0\|^2+\int_{0}^{T}\|F^0_s\|^2\rd s+
	\int_{0}^{T}\|f^{0}_s\|^2\rd s+ \int_{0}^{T}\|G^0_s\|^2\rd s\right],
	\end{split}
	\end{equation*}
	%	\begin{equation*}
	%	\begin{split}
	%		\|\rho\|^2_{\cS^2(H^{1,2})}&+\|u\|^2_{\cS^2(H^{1,2})} 
	%	+\|\rho\|^2_{\cL^2(H^{2,2})}+\|u\|^2_{\cL^2(H^{2,2})}
	%	+\sum_{i=1}^{k}\|\psi^i\|^2_{\cL^2(H^{1,2})} \\
	%	&\leq C\left(\|\rho_0\|^2_{L^2(\Omega; H^{1,2})}+
	%\|g^0\|^2_{L^2(\Omega;H^{1,2})}+	\|F^0\|^2_{\cL^2(L^2)}+	\|G^0\|^2_{\cL^2(L^2)} +	\|f^{i,0}\|^2_{\cL^2(H^{1,2})}
	%	\right)
	%	\end{split}
	%	\end{equation*}
	where the constant $C=C(L_1^f,L_2^f,L^g,L^F_1,L^F_2,L_1^G,L_2^G,T)$. Moreover, this weak solution is a strong one satisfying
	\begin{equation*}
	\begin{split}
	&\bE\bigg[\sup_{t\in [0,T]}\|\nabla\rho(t)\|^2+\sup_{t\in [0,T]}\|\nabla u(t)\|^2\bigg]+\bE\int_{0}^{T}\left(\|\Delta\rho(s)\|^2+\|\Delta u(s)\|^2+ \|\psi(s)\|_{1,2}^2\right)\rd s\\
	&\le C\bE\left[ 1+\|\rho_0\|^2_{1,2}+\| g^0\|^2+\int_{0}^{T}\|F^0_s\|^2\rd s+\int_{0}^{T}\| f^{0}_s\|^2\rd s+ \int_{0}^{T}\|G^0_s\|^2\rd s\right],
	\end{split}
	\end{equation*}
	with constant $C=C(L_1^f,L_2^f,\tilde{L}^f,L^g,\tilde{L}^g,L^F_1,L^F_2,L^G_1,L^G_2,T)$.
\end{thm}

\begin{rmk}\label{rmk-wellposedness-FBSPDE}
Recalling that in Assumptions \ref{ass1}-\ref{ass2}, the Lipschitz constant $(L_2^F,L_2^f)$ (resp. the pair $(L_1^G,L^g)$) characterizes the dependence of the forward equation \eqref{FSPDE} (resp. backward equation \eqref{BSPDE}) on the solution of backward equation \eqref{BSPDE} (resp. forward equation \eqref{FSPDE}), we may see that when either $(L_2^F,L_2^f)$ or $(L_1^G,L^g)$ takes sufficiently small value, the extent of coupling can be thought of to be weak and condition \eqref{condtn-wellp} guarantees the wellposedness of FBSPDE  \eqref{FSPDE}-\eqref{BSPDE}. Such an assertion/observation does not exist in \cite{HongYin2014FBSPDE}, because therein, the Lipschitz constants on solutions of forward and backward equations are not separated out as in Assumptions \ref{ass1}-\ref{ass2} and the function spaces for the solutions  $\rho$ and $u$ and the associated computations are different from ours. To discuss the numerical approximations, we focus Theorem \ref{Thm1} on the wellposedness of both weak and strong solution, while in \cite{HongYin2014FBSPDE}, only the weak solution is concerned but for a more general class of FBSPDEs without any numerical discussions. Some numerical methods and wellposeness of solutions in Sobolev spaces for coupled FBSPDEs may be found in \cite{dunst2016forward}, but the FBSPDEs therein are restricted to linear ones with coefficients of the following form:
\begin{align*}
F\big(t,x,\rho(t),\nabla\rho(t),u(t),\nabla u(t),\psi(t)\big) &=  u(t), \quad 
f^i(t,x,\rho(t),u(t))=-\nu^i (t)\rho(t),\quad i=1,\dots, k,\\
G\big(t,x,\rho(t),\nabla\rho(t),u(t),\nabla u(t),\psi(t)\big) &= \sum_{i=1}^k \nu^i(t)\psi^i(t) + h(t),\quad
G(x,\rho(T))= \Psi(x).
\end{align*}
We would note that neither of the papers \cite{HongYin2014FBSPDE,dunst2016forward} incorporate the nonlocal dependence as stated in Remark \ref{rmk-nonlocal}.

In addition, the wellposedness  in H\"older spaces is addressed in \cite{CardaliaguetetelMFG2015} for a class of FBSPDEs with linear coefficient $f$ and periodic boundary conditions under certain strong assumptions; no numerical approximation is discussed and the readers may refer to Remark \ref{rmk-nonlocal} for more comparisons. 
\end{rmk}
The proof below will be divided into two steps; the first time reader may skip the detailed proof to enjoy the numerical analysis in the next sections.

\begin{proof}[Proof of Theorem \ref{Thm1}.]
	For each $u'_1,u'_2\in \cL^2(H^{1,2}_0)\cap\cS^2(L^2)$ and ${\psi'_1},{\psi'_2}\in\cL^2(L^2)$, the standard SPDE theory (see \cite{PR07} for instance) indicates that there are unique solutions, denoted by $\rho_1$ and $\rho_2$ respectively,  to the following SPDEs
	\begin{equation}\label{FWcontrc1}
	\left\{\begin{array}{l}
	\begin{split}
	%&\begin{split}
	\rd\rho_1(t,x)&=\bigg(\Delta \rho_1(t,x)+F(t,x,\rho_1(t),\nabla\rho_1(t),u'_1(t),\nabla u'_1(t),{\psi'_1}(t))\bigg)\,\rd t\\
	&\quad
	-\sum_{i=1}^{k} f^i(t,x,\rho_1(t),u'_1(t)) \,\rd W_{t}^{i},\quad (t,x)\in[0,T]\times D;\\
	%\end{split}\\
	\rho_1(0,x)&=\rho_0(x), \quad x\in D;\\
	\rho_1(t,x)\bigg|_{\partial D}&=0,\quad t\in[0,T],
	\end{split}
	\end{array}\right.
	\end{equation}
	and
	\begin{equation}
	\left\{\begin{array}{l}
	\begin{split}
	%&\begin{split}
	\rd\rho_2(t,x) &
		=\bigg(\Delta \rho_2(t,x)+F(t,x,\rho_2(t),\nabla\rho_2(t),u'_2(t),\nabla u'_2(t),{\psi'_2}(t))\bigg)\,\rd t\\
		&\quad
		-\sum_{i=1}^{k} f^i(t,x,\rho_2(t),u'_2(t)) \,\rd W_{t}^{i}, 
	\quad (t,x)\in[0,T]\times D;\\
	%\end{split}\\
	\rho_2(0,x) &
		=\rho_0(x), \quad x\in D;\\
	\rho_2(t,x)\bigg|_{\partial D}&
		=0,\quad t\in[0,T].
	\end{split}
	\end{array}\right.
	\end{equation}
	Meanwhile, by the theory of BSPDEs (see \cite{marquez2004some} for instance),  there are unique weak solutions, denoted by $(u_1,\psi_1)$ and $(u_2,\psi_2)$ respectively,  to the following backward SPDEs:
	\begin{equation}\label{BWcontrc1}
	\left\{\begin{array}{l}
	\begin{split}
	%&\begin{split}
	-\rd u_1(t,x)&
		=\bigg(\Delta u_1(t,x)+G(t,x,\rho_1(t),\nabla\rho_1(t),u_1(t),\nabla u_1(t),\psi_1(t))\bigg)\,\rd t\\
		&\quad
			 -\sum_{i=1}^{k}\psi_1^i (t,x)\,\rd W_{t}^{i}, \quad
	(t,x)\in[0,T]\times D;\\
	%\end{split}\\
	u_1(T,x)&
		=g(x,\rho_1(T)), \quad x\in D;\\
	u_1(t,x)\bigg|_{\partial D}&
		=0,\quad t\in[0,T],
	\end{split}
	\end{array}\right.
	\end{equation}
	and
	\begin{equation}
	\left\{\begin{array}{l}
	\begin{split}
	%&\begin{split}
	-\rd u_2(t,x)&
		=\bigg(\Delta u_2(t,x)+G(t,x,\rho_2(t),\nabla\rho_2(t),u_2(t),\nabla u_2(t),\psi_2(t))\bigg)\,\rd t\\
		&\quad
			 -\sum_{i=1}^{k}\psi_2^i (t,x)\,\rd W_{t}^{i},\quad
	(t,x)\in[0,T]\times D;\\
	%\end{split}\\
	u_2(T,x)&
		=g(x,\rho_2(T)), \quad x\in D;\\
	u_2(t,x)\bigg|_{\partial D}&
		=0,\quad t\in[0,T].
	\end{split}
	\end{array}\right.
	\end{equation}
	We shall use the contraction mapping methods to prove the existence and uniqueness of weak and strong solution to FBSPDE  \eqref{FSPDE}-\eqref{BSPDE}  and the involved computations will be divided into two parts.
	
	\textbf{Step 1.} The first part is devoted to some computations and estimates associated to the forward equation \eqref{FSPDE}. Applying It\^o formula for square norm (see \cite[Theorem 3.1]{Krylov_Rozovskii81} for instance) to \eqref{FWcontrc1} gives
	\begin{align*}
	&\|\rho_1(t)\|^2
	\\
	&=\|\rho_0\|^2-2\int_{0}^{t}\langle\nabla\rho_1(s),\nabla\rho_1(s)\rangle\rd s-2\sum_{i=1}^{k}\int_{0}^{t}\langle f^i(s,\rho_1(s),u'_1(s)),\rho_1(s)\rangle\rd W^i_s\\
	&\quad
	+\sum_{i=1}^{k}\int_{0}^{t}\|f^i(s,\rho_1(s),u'_1(s))\|^2\rd s
	+2\int_{0}^{t}\langle F(s,\rho_1(s),\nabla\rho_1(s),u'_1(s),\nabla u'_1(s),{\psi'_1}(s)),\rho_1(s)\rangle\rd s. 
	\end{align*}
	In view of the Lipschitz continuity in Assumption \ref{ass1}, we have
	\begin{align*}
	%\begin{split}
	&2\int_{0}^{t}\langle F(s,\rho_1(s),\nabla\rho_1(s),u'_1(s),\nabla u'_1(s),{\psi'_1}(s)),\rho_1(s)\rangle\rd s\\
	&\le 2\int_{0}^{t}\|F(s,\rho_1(s),\nabla\rho_1(s),u'_1(s),\nabla u'_1(s),{\psi'_1}(s))\|\cdot\|\rho_1(s)\|\rd s\\
	&\le 2\int_{0}^{t}\bigg( \|F^0_s\|+L^F_1\left(\|\rho_1(s)\|+\|\nabla\rho_1(s)\|\right)+L^F_2\left(\|u'_1(s)\|+\|\nabla u'_1(s)\|+\|{\psi'_1}(s)\|\right)\bigg)\cdot\|\rho_1(s)\|\rd s\\
	&\le 2\int_{0}^{t}\|F^0_s\|\cdot\|\rho_1(s)\|\rd s+2L^F_1\int_{0}^{t}\|\rho_1(s)\|^2\rd s+2L^F_1\int_{0}^{t}\|\nabla\rho_1(s)\|\cdot\|\rho_1(s)\|\rd s\\
	&\quad\quad+2L^F_2\int_{0}^{t}\|u'_1(s)\|\cdot\|\rho_1(s)\|\rd s+2L^F_2\int_{0}^{t}\|\nabla u'_1(s)\|\cdot\|\rho_1(s)\|\rd s+2L^F_2\int_{0}^{t}\|{\psi'_1}(s)\|\cdot\|\rho_1(s)\|\rd s.
	%\end{split}
	\end{align*}
	Notice that 
	\begin{align*}
	%\begin{split}
	2L^F_2\int_{0}^{t}\|u'_1(s)\|\cdot\|\rho_1(s)\|\rd s&\le 2\sup_{s\in [0,t]}\|\rho_1(s)\|\;L^F_2\int_{0}^{t}\|u'_1(s)\|\rd s \\
	&\le\eps\sup_{s\in [0,t]}\|\rho_1(s)\|^2+\frac{1}{\eps}|{L^F_2}|^2\left(\int_{0}^{t}\|u'_1(s)\|\rd s\right)^2\\
	&\le \eps\sup_{s\in [0,t]}\|\rho_1(s)\|^2+\frac{1}{\eps}|{L^F_2}|^2\;t\int_{0}^{t}\|u'_1(s)\|^2\rd s.
	%\end{split}
	\end{align*}
	Using similar computations as above and taking supremum over $t\in[0,\tau]$ for $\tau\in [0,T]$, we may arrive at
	\begin{align*}
	%\begin{split}
	&\sup_{t\in [0,\tau]}\|\rho_1(t)\|^2+2\int_{0}^{\tau}\|\nabla\rho_1(s)\|^2\rd s\\
	&\le
	\|\rho_0\|^2
	+ \int_{0}^{\tau}\|f(s,\rho_1(s),u'_1(s))\|^2\rd s
	+2L^F_1\int_{0}^{\tau}\|\rho_1(s)\|^2\rd s\\
	&\quad 
		+2\int_{0}^{\tau}\|F^0_s\|\cdot\|\rho_1(s)\|\rd s	+2L^F_1\int_{0}^{\tau}\|\nabla\rho_1(s)\|\cdot\|\rho_1(s)\|\rd s+\frac{1}{\eps_1}|L^F_2|^2\;\tau\int_{0}^{\tau}\|u'_1(s)\|^2\rd s\\
	&\quad
	+\frac{1}{\eps_2}|L^F_2|^2\;\tau\int_{0}^{\tau}\|\nabla u'_1(s)\|^2\rd s+\frac{1}{\eps_3}|L^F_2|^2\;\tau\int_{0}^{\tau}\sum_{i=1}^{k}\|{\psi'_1}^i(s)\|^2\rd s+(\eps_1+\eps_2+\eps_3)\sup_{t\in [0,\tau]}\|\rho_1(t)\|^2
	\\
	&\quad
	+2\sup_{t\in [0,\tau]}\bigg|\int_{0}^{t}\sum_{i=1}^{k}\langle f^i(s,\rho_1(s),u'_1(s)),\rho_1(s)\rangle\rd W^i_s\bigg|.
	%\end{split}
	\end{align*}
	A straightforward application of Young's inequality further gives
	\begin{align*}
	&(1-\eps_1-\eps_2-\eps_3)\sup_{t\in [0,\tau]}\|\rho_1(t)\|^2+(2-\eps_4)\int_{0}^{\tau}\|\nabla\rho_1(s)\|^2\rd s\\
	&\le
	\|\rho_0\|^2+\eps_5 \int_{0}^{\tau}\|F^0_s\|^2\rd s+3 \int_{0}^{\tau}\|f^{0}_s\|^2\rd s+\frac{1}{\eps_1}|L_2^F|^2\;\tau\int_{0}^{\tau}\|u'_1(s)\|^2\rd s\\
	&
	+\frac{1}{\eps_2}|L_2^F|^2\;\tau\int_{0}^{\tau}\|\nabla u'_1(s)\|^2\rd s+\frac{1}{\eps_3}|L^F_2|^2\;\tau\int_{0}^{\tau}\sum_{i=1}^{k}\|{\psi'_1}^i(s)\|^2\rd s+3|L_2^f|^2\int_{0}^{\tau}\|u'_1(s)\|^2\rd s\\
	&
	+2\sup_{t\in [0,\tau]}\bigg|\int_{0}^{t}\!\!\sum_{i=1}^{k}\langle f^i(s,\rho_1(s),u'_1(s)),\rho_1(s)\rangle\rd W^i_s\bigg|+\left(2L^F_1+\frac{|L_1^F|^2}{\eps_4}+\frac{1}{\eps_5}+3|L_1^f|^2\right)\int_{0}^{\tau}\!\!\|\rho_1(s)\|^2\rd s.
	\end{align*}
	Letting $\eps_1=\eps_2=\eps_3=1/6,\eps_4=7/4,\eps_5=1$, it follows that
	\begin{align}
	&\sup_{t\in [0,\tau]}\|\rho_1(t)\|^2+\frac{1}{2}\int_{0}^{\tau}\|\nabla\rho_1(s)\|^2\rd s
	\nonumber\\
	&\le
	2\left(\|\rho_0\|^2+\int_{0}^{\tau}\|F^0_s\|^2\rd s+3\sum_{i=1}^{k}\int_{0}^{\tau}\|f^{i,0}_s\|^2\rd s\right)
	+6|L_2^f|^2\int_{0}^{\tau}\|u'_1(s)\|^2\rd s
	\nonumber\\
	&\quad
	+12|{L_2^F}|^2\;\tau\left(\int_{0}^{\tau}\|u'_1(s)\|^2\rd s+\int_{0}^{\tau}\|\nabla u'_1(s)\|^2\rd s+\int_{0}^{\tau} \|{\psi'_1}(s)\|^2\rd s\right)
	\nonumber\\
	&\quad
		+C_1\int_{0}^{\tau}\|\rho_1(s)\|^2\rd s
	+4\sup_{t\in [0,\tau]}\bigg|\int_{0}^{t}\sum_{i=1}^{k}\langle f^i(s,\rho_1(s),u'_1(s)),\rho_1(s)\rangle\rd W^i_s\bigg|,
	\label{ito-formla-est}
	\end{align}
	where constant $C_1=C_1(L_1^f,L^F_1)$.
	
For the terms involving stochastic integrals, we  use BDG inequality and obtain
	\begin{align*}
	&4\bE\bigg[\sup_{t\in [0,\tau]}\bigg|\int_{0}^{t}\sum_{i=1}^{k}\langle f^i(s,\rho_1(s),u'_1(s)),\rho_1(s)\rangle\rd W^i_s\bigg|\bigg]\\
	&\le 
	C\bE\bigg[\left(\int_{0}^{\tau}\sum_{i=1}^{k}\bigg|\langle f^i(s,\rho_1(s),u'_1(s)),\rho_1(s)\rangle\bigg|^2\rd s\right)^\frac{1}{2}\bigg]\\
	&\le 
	\frac{\tilde{C}}{3}\bE\bigg[\left(\int_{0}^{\tau} \bigg(\|f_s^{0}\|+L_1^f\|\rho_1(s)\|+L_2^f\|u'_1(s)\|\bigg)^2
		\cdot\|\rho_1(s)\|^2\rd s\right)^\frac{1}{2}\bigg]\\
	&\le 
		\frac{\tilde{C}}{3}\bE\bigg[\left(\sup_{t\in [0,\tau]}\|\rho_1(t)\|^2\int_{0}^{\tau} \bigg(\|f_s^{0}\|+L_1^f\|\rho_1(s)\|+L_2^f\|u'_1(s)\|\bigg)^2\rd s\right)^\frac{1}{2}\bigg]\\
	&\le 
	\bE\bigg[\eps_6\sup_{t\in [0,\tau]}\|\rho_1(t)\|^2
	+\frac{\tilde{C}^2}{9\eps_6}\int_{0}^{\tau} \bigg(\|f_s^{0}\|+L_1^f\|\rho_1(s)\|+L_2^f\|u'_1(s)\|\bigg)^2\rd s\bigg]\\
	&\le 
	\bE\bigg[\eps_6\sup_{t\in [0,\tau]}\|\rho_1(t)\|^2
	+\frac{\tilde{C}^2}{\eps_6}\int_{0}^{\tau} 
	\bigg(\|f_s^{0}\|^2+|L_1^f|^2\|\rho_1(s)\|^2+|L_2^f|^2\|u'_1(s)\|^2\bigg)\rd s\bigg]\\
	&\le 
	\eps_6\bE\bigg[\sup_{t\in [0,\tau]}\|\rho_1(t)\|^2\bigg]+\frac{\tilde{C}^2|L_1^f|^2}{\eps_6}\bE\int_{0}^{\tau}\|\rho_1(s)\|^2\rd s+\frac{\tilde{C}^2|L_2^f|^2}{\eps_6}\bE\int_{0}^{\tau}\|u'_1(s)\|^2\rd s
	\\
	&\quad
		+\frac{\tilde{C}^2}{\eps_6}\bE\int_{0}^{\tau} \|f_s^{0}\|^2\rd s.
	\end{align*}
	Taking expectations on both sides of \eqref{ito-formla-est} and using above deductions with $\eps_6=1/2$, we have
	\begin{align*}
	&\bE\bigg[\sup_{t\in [0,\tau]}\|\rho_1(t)\|^2\bigg]+\bE\int_{0}^{\tau}\|\nabla\rho_1(s)\|^2\rd s\\
	&\le 4\bE\left[\|\rho_0\|^2+\int_{0}^{\tau}\|F^0_s\|^2\rd s+(\tilde{C}^2+3) \int_{0}^{\tau}\|f^{0}_s\|^2\rd s\right]+C_1\bE\int_{0}^{\tau}\|\rho_1(s)\|^2\rd s\\
	&\quad\quad+24|{L_2^F}|^2\;\tau \bE\left[\int_{0}^{\tau}\|u'_1(s)\|^2\rd s+\int_{0}^{\tau}\|\nabla u'_1(s)\|^2\rd s+\int_{0}^{\tau} \|{\psi'_1}(s)\|^2\rd s\right]\\
	&\quad\quad +2(\tilde{C}^2+3)|L_2^f|^2\tau \bE \bigg[\sup_{t\in [0,\tau]}\|u'_1(t)\|^2\bigg],
	\end{align*}
	which by Gronwall's inequality implies
	\begin{align}
	&\bE\bigg[\sup_{t\in [0,T]}\|\rho_1(t)\|^2\bigg]+\bE\int_{0}^{T}\|\nabla\rho_1(s)\|^2\rd s
	\nonumber\\
	&\le \bar{C}_1\bE\left[\|\rho_0\|^2+\int_{0}^{T}\|F^0_s\|^2\rd s+\int_{0}^{T}\|f^{0}_s\|^2\rd s\right]+\bar{C}_1|L_2^f|^2T \bE \bigg[\sup_{t\in [0,T]}\|u'_1(t)\|^2\bigg]
	\nonumber\\
	&\quad
	+\bar{C}_1|{L_2^F}|^2\;T\bE\left[\int_{0}^{T}\bigg(\|u'_1(s)\|^2+\|\nabla u'_1(s)\|^2+ \|{\psi'_1}(s)\|^2\bigg)\rd s\right], \nonumber
	\end{align}
	with the constant $\bar {C}_1=24 \left(\tilde C^2+3\right)e^{C_1T}$.
		
	Put $\hat{\rho}=\rho_1-\rho_2$, $\hat{u}'=u'_1-u'_2$,  and $\hat{\psi'}={\psi'_1}-{\psi'_2}$. Through analogous applications of Ito formula to $\hat{\rho}$ together with similar computations,  we may arrive at
	\begin{align}
	&\bE\bigg[\sup_{t\in [0,T]}\|\hat{\rho}(t)\|^2\bigg]+\bE\int_{0}^{T}\|\nabla\hat{\rho}(s)\|^2\rd s
	\nonumber\\
	&\le \bar{C}_1|{L_2^F}|^2\;T\bE\int_{0}^{T}\left(\|\hat{u}'(s)\|^2+\|\nabla \hat{u}'(s)\|^2+ \|{\hat{\psi'}}(s)\|^2\right)\rd s+\bar{C}_1|L_2^f|^2T \bE \bigg[\sup_{t\in [0,T]}\|\hat{u}'(t)\|^2\bigg]
	\nonumber\\
	&\le \bar{C}_1|{L_2^F}|^2\;T\bE\int_{0}^{T}\left(\|\nabla \hat{u}'(s)\|^2
	+\|{\hat{\psi'}}(s)\|^2\right)\rd s+\bar{C}_1T\bigg(|L_2^F|^2T+|L_2^f|^2\bigg) \bE \bigg[\sup_{t\in [0,T]}\|\hat{u}'(t)\|^2\bigg]
	\nonumber\\
	&\le \bar{C}_1T\cdot\max \{|L_2^F|^2,|L_2^F|^2T+|L_2^f|^2\}\bE \bigg[\sup_{t\in [0,T]}\|\hat{u}'(t)\|^2
%	\nonumber\\
%	&\quad
	+\int_{0}^{T}\bigg(\|\nabla \hat{u}'(s)\|^2+ \|{\hat{\psi'}}(s)\|^2\bigg)\rd s\bigg].
	\label{Rf1}
	\end{align}

	 \textbf{Step 2.} Then we conduct the computations for the backward equation \eqref{BSPDE} and derive the wellposedness of FBSPDE \eqref{FSPDE}-\eqref{BSPDE}. Applying It\^o formula to \eqref{BWcontrc1} yields that
	\begin{align*}
	&\|u_1(t)\|^2
	\\
	&=\|g(\rho_1(T)\|^2-2\int_{t}^{T}\langle\nabla u_1(s),\nabla u_1(s)\rangle\rd s-2\sum_{i=1}^{k}\int_{t}^{T}\langle\psi_1^i(s),u_1(s)\rangle\rd W^i_s-\int_{t}^{T}\|\psi_1(s)\|^2\rd s\\
	&
	+2\int_{t}^{T}\langle G(s,\rho_1(s),\nabla\rho_1(s),u_1(s),\nabla u_1(s),{\psi_1}(s)),u_1(s)\rangle\rd s,
	\end{align*}
	which by Assumption \ref{ass2}-(c) implies
	\begin{align}
	&\|u_1(t)\|^2+2\int_{t}^{T}\|\nabla u_1(s)\|^2\rd s
	+ \int_{t}^{T}\|\psi_1(s)\|^2\rd s
	\nonumber\\
	&\le 2\|g^0\|^2+2|{L^g}|^2\|\rho_1(T)\|^2+2\int_{t}^{T}\langle G(s,\rho_1(s),\nabla\rho_1(s),u_1(s),\nabla u_1(s),{\psi_1}(s)),u_1(s)\rangle\rd s 
	\nonumber\\
	&\quad\quad-2\sum_{i=1}^{k}\int_{t}^{T}\langle\psi_1^i(s),u_1(s)\rangle\rd W^i_s.
	\label{ItoBw}
	\end{align}
	Further by Assumption \ref{ass2}-(a) and (c), we have
	\begin{align}
	&2\int_{t}^{T}\langle G(s,\rho_1(s),\nabla\rho_1(s),u_1(s),\nabla u_1(s),{\psi_1}(s)),u_1(s)\rangle\rd s
	\notag\\
	&\le 
	2\int_{t}^{T}\!\!\bigg[ \|G^0_s\|+L^G_1\big(\|\rho_1(s)\|+\|\nabla\rho_1(s)\|\big)+L^G_2\big(\|u_1(s)\|+\|\nabla u_1(s)\|+\|{\psi_1}(s)\|\big)\bigg]\|u_1(s)\|\, \rd s
	\nonumber\\
	&\le 
	2\int_{t}^{T}\!\!\bigg(\|G^0_s\|\cdot\|u_1(s)\|+L^G_2\big(\|u_1(s)\|\cdot\|u_1(s)\|+\|\nabla u_1(s)\|\cdot\|u_1(s)\|+\|{\psi_1}(s)\|\cdot\|u_1(s)\|\big)\bigg)\rd s
	\nonumber\\
	&\quad
	+2L^G_1\int_{t}^{T}\bigg( \|\rho_1(s)\|\cdot\|u_1(s)\|+\|\nabla\rho_1(s)\|\cdot\|u_1(s)\|\bigg)\rd s
	\nonumber\\
	& \le 
	\int_{t}^{T}\bigg(\eps_1\|G^0_s\|^2+\eps_2\|\nabla u_1(s)\|^2+\eps_3 \|{\psi_1}(s))\|^2\bigg)\rd s
	\nonumber\\
	&\quad
	+\left(\frac{1}{\eps_1}+\frac{|{L_2^G}|^2}{\eps_2}+\frac{|{L_2^G}|^2}{\eps_3}+2L^G_2\right)\int_{t}^{T}\|u_1(s)\|^2\, \rd s
	\notag\\
	&\quad
	+2L^G_1\sup_{s\in [t,T]}\|u_1(s)\|\int_{t}^{T}\bigg(\|\rho_1(s)\|+\|\nabla\rho_1(s)\|\bigg)\rd s
	\nonumber\\
	& \le
	 \int_{t}^{T}\bigg(\eps_1\|G^0_s\|^2+\eps_2\|\nabla u_1(s)\|^2+\eps_3 \|{\psi_1}(s))\|^2\bigg)\rd s
	 \notag\\
	 &\quad
	 +\left(\frac{1}{\eps_1}+\frac{|{L_2^G}|^2}{\eps_2}+\frac{|{L_2^G}|^2}{\eps_3}+2L^G_2\right)\int_{t}^{T}\|u_1(s)\|^2\rd s
	 +\eps_4\sup_{s\in [t,T]}\|u_1(s)\|^2
	 \notag
	 \\
	&\quad
	+|{L_1^G}|^2(T-t)\frac{2}{\eps_4}\int_{t}^{T}\left(\|\rho_1(s)\|^2+\|\nabla\rho_1(s)\|^2\right)\rd s.
	\label{est-G-BSPDE}
	\end{align}
	Taking $\eps_1=1,\eps_2=3/2$, and $\eps_3=1/2$ and combining the above computations with 
	\eqref{ItoBw} yield that
	\begin{align*}
	&\|u_1(t)\|^2+\frac{1}{2}\int_{t}^{T}\|\nabla u_1(s)\|^2\rd s+\frac{1}{2}
	\int_{t}^{T}\|\psi_1(s)\|^2\rd s
		\\
	&\le 2\|g^0\|^2+2|L^g|^2\|\rho_1(T)\|^2+\int_{t}^{T}\|G^0_s\|^2\rd s
	+C_2\int_{t}^{T}\|u_1(s)\|^2\rd s+\eps_4\sup_{s\in [t,T]}\|u_1(s)\|^2\\
	&\quad\quad+(T-t)|L^G_1|^2\frac{2}{\eps_4}\int_{t}^{T}\left(\|\rho_1(s)\|^2+ \|\nabla\rho_1(s)\|^2\right)\rd s -2\sum_{i=1}^{k}\int_{t}^{T}\langle\psi_1^i(s),u_1(s)\rangle\rd W^i_s,
	\end{align*}
	where the constant $C_2=C_2(L^G_2)$; in particular,  taking expectations gives
	 \begin{align}
	&\bE\int_{t}^{T}\|\psi_1(s)\|^2\rd s
	\notag\\
	&\le
	 4\bE\bigg[\|g^0\|^2+|{L^g}|^2\|\rho_1(T)\|^2\bigg]+2\bE\int_{t}^{T}\|G^0_s\|^2\rd s+2\eps_4\bE\bigg[\sup_{s\in [t,T]}\|u_1(s)\|^2\bigg]
	\notag\\
	&
	\quad
	+4(T-t)|{L_1^G}|^2\frac{1}{\eps_4}\bE\int_{t}^{T}\bigg(\|\rho_1(s)\|^2+ \|\nabla\rho_1(s)\|^2\bigg)\rd s+C_2\bE\int_{t}^{T}\|u_1(s)\|^2\rd s.
	\label{RB1}
	\end{align}
	
	On the other hand, taking supremum over $t\in[\tau,T]$ for $\tau\in[0,T]$ in \eqref{ItoBw}, we have
	\begin{align*}
	&\sup_{t\in [\tau,T]}\|u_1(t)\|^2+2\int_{\tau}^{T}\|\nabla u_1(s)\|^2\rd s
	+\int_{\tau}^{T}\|\psi_1(s)\|^2\rd s\\
	&\le2\|g^0\|^2+2|{L^g}|^2\|\rho_1(T)\|^2
	+2\sup_{t\in [\tau,T]}\bigg|\int_{t}^{T}\sum_{i=1}^{k}\langle\psi_1^i(s),u_1(s)\rangle\rd W^i_s\bigg|
	 \\
	&\quad
	+2\int_{\tau}^{T}\big\langle G(s,\rho_1(s),\nabla\rho_1(s),u_1(s),\nabla u_1(s),{\psi_1}(s)),u_1(s)\big\rangle\rd s.
	\end{align*}
	Rewriting \eqref{est-G-BSPDE} as 
	\begin{align*}
	&2\int_{\tau}^{T}\big\langle G(s,\rho_1(s),\nabla\rho_1(s),u_1(s),\nabla u_1(s),{\psi_1}(s)),u_1(s)\big\rangle\rd s\\
	& \le \int_{\tau}^{T}\bigg(\eps_1\|G^0_s\|^2+\eps_2\|\nabla u_1(s)\|^2+\eps_3\|{\psi_1}(s))\|^2\bigg)\rd s+\eps_5\sup_{t\in [\tau,T]}\|u_1(t)\|^2\\
	&\quad\quad+\left(\frac{1}{\eps_1}+\frac{|{L_2^G}|^2}{\eps_2}+\frac{|{L_2^G}|^2}{\eps_3}+2L^G_2\right)\int_{\tau}^{T}\|u_1(s)\|^2\rd s\\
	&\quad\quad+2|{L_1^G}|^2(T-\tau)\frac{1}{\eps_5}\int_{\tau}^{T}\left(\|\rho_1(s)\|^2+\|\nabla\rho_1(s)\|^2\right)\rd s,
	\end{align*}
	and choosing $\eps_1=1,\eps_2=7/4,\eps_3=3/4$, and $\eps_5=1/2$, we have
	\begin{align*}
	&\sup_{t\in [\tau,T]}\|u_1(t)\|^2+\frac{1}{2}\int_{\tau}^{T}\|\nabla u_1(s)\|^2\rd s+\frac{1}{2}\int_{\tau}^{T}\|\psi_1(s)\|^2\rd s\\
	&\le 4\bigg(\|g^0\|^2+|{L^g}|^2\|\rho_1(T)\|^2\bigg)+8|{L_1^G}|^2(T-\tau)\int_{\tau}^{T}\bigg(\|\rho_1(s)\|^2+\|\nabla\rho_1(s)\|^2\bigg)\rd s\\
	&\quad\quad+4\sup_{t\in [\tau,T]}\bigg|\int_{t}^{T}\sum_{i=1}^{k}\langle\psi_1^i(s),u_1(s)\rangle\rd W^i_s\bigg|+C_2\int_{\tau}^{T}\|u_1(s)\|^2\rd s+ 2\int_{\tau}^{T}\|G^0_s\|^2\rd s,
	\end{align*}
	with $C_2=C_2(L^G_2)$, which by taking expectations on both sides implies
	\begin{align}
	&\bE\bigg[\sup_{t\in [\tau,T]}\|u_1(t)\|^2\bigg]+\frac{1}{2}\bE\int_{\tau}^{T}\|\nabla u_1(s)\|^2\rd s+\frac{1}{2}\bE\int_{\tau}^{T}\|\psi_1(s)\|^2\rd s
	\nonumber\\
	&\le 
	4\bE\bigg[\|g^0\|^2+|{L^g}|^2\|\rho_1(T)\|^2\bigg]+ 2\bE\bigg[\int_{\tau}^{T}\|G^0_s\|^2\rd s\bigg]+4\bE\bigg[\sup_{t\in [\tau,T]}\bigg|\int_{t}^{T}\sum_{i=1}^{k}\langle\psi_1^i(s),u_1(s)\rangle\rd W^i_s\bigg|\bigg]
	\nonumber\\
	&\quad
	+8|{L_1^G}|^2(T-\tau)\bE\bigg[\int_{\tau}^{T}\bigg(\|\rho_1(s)\|^2+\|\nabla\rho_1(s)\|^2\bigg)\rd s\bigg]+C_2\bE\int_{\tau}^{T}\|u_1(s)\|^2\rd s.
	\label{est-q-1}
	\end{align}
	Again, we use BDG inequality to deal with the stochastic integrals and obtain
	\begin{align*}
	&4\bE\bigg[\sup_{t\in [\tau,T]}\bigg|\int_{t}^{T}\sum_{i=1}^{k}\langle\psi_1^i(s),u_1(s)\rangle\rd W^i_s\bigg|\bigg]\\
	&= 4\bE\bigg[\sup_{t\in [\tau,T]}\bigg|\int_{\tau}^{T}\sum_{i=1}^{k}\langle\psi_1^i(s),u_1(s)\rangle\rd W^i_s-\int_{\tau}^{t}\sum_{i=1}^{k}\langle\psi_1^i(s),u_1(s)\rangle\rd W^i_s\bigg|\bigg]\\
	&\le 4\bE\bigg[\bigg|\int_{\tau}^{T}\sum_{i=1}^{k}\langle\psi_1^i(s),u_1(s)\rangle\rd W^i_s\bigg|+\sup_{t\in [\tau,T]}\bigg|\int_{\tau}^{t}\sum_{i=1}^{k}\langle\psi_1^i(s),u_1(s)\rangle\rd W^i_s\bigg|\bigg]\\
	&\le 4\bE\bigg[\sup_{t\in [\tau,T]}\bigg|\int_{\tau}^{t}\sum_{i=1}^{k}\langle\psi_1^i(s),u_1(s)\rangle\rd W^i_s\bigg|+\sup_{t\in [\tau,T]}\bigg|\int_{\tau}^{t}\sum_{i=1}^{k}\langle\psi_1^i(s),u_1(s)\rangle\rd W^i_s\bigg|\bigg]
	\\
	&\le 
	\tilde{C} \, \bE\bigg[\int_{\tau}^{T}\sum_{i=1}^{k}\bigg|\langle\psi_1^i(s),u_1(s)\rangle\bigg|^2\rd s\bigg]^\frac{1}{2}\\
	&\le
	 \tilde{C} \,  \bE\bigg[\sup_{t\in [\tau,T]}\|u_1(t)\|^2\int_{\tau}^{T}\sum_{i=1}^{k}\|\psi_1^i(s)\|^2\rd s\bigg]^\frac{1}{2}\\
	&\le  \frac{1}{4} \bE\bigg[\sup_{t\in [\tau,T]}\|u_1(t)\|^2\bigg] +{\tilde{C}^2}\bE\int_{\tau}^{T} \|\psi_1(s)\|^2\rd s,
	\end{align*}
	which together with \eqref{RB1} and \eqref{est-q-1} implies
	\begin{align*}
	&\frac{3}{4}\bE\bigg[\sup_{t\in [\tau,T]}\|u_1(t)\|^2\bigg]+\frac{1}{2}\bE\int_{\tau}^{T}\|\nabla u_1(s)\|^2\rd s+\frac{1}{2}\bE \int_{\tau}^{T}\|\psi_1(s)\|^2\rd s
	\\
	&\le 
	\tilde C_1 
	\bE\bigg[\|g^0\|^2+|{L^g}|^2\|\rho_1(T)\|^2\bigg]+ \tilde{C}_1\bE\bigg[\int_{\tau}^{T}\|G^0_s\|^2\rd s\bigg]+ \tilde{C}_2\bE\int_{\tau}^{T}\|u_1(s)\|^2\rd s
	\\
	&\quad
	+\tilde{C}_1|{L_1^G}|^2(T-\tau)\bE\bigg[\int_{\tau}^{T}\bigg(\|\rho_1(s)\|^2+\|\nabla\rho_1(s)\|^2\bigg)\rd s\bigg]+2\tilde{C}^2\eps_4\bE\bigg[\sup_{t\in [\tau,T]}\|u_1(t)\|^2\bigg],
	\end{align*}
	with $\tilde C_1=\tilde C^2+8$ and $\tilde{C}_2= \tilde{C}^2 C_2 + C_2$. 
	Taking $\eps_4=\frac{1}{8\tilde{C}^2}$ gives
	\begin{align*}
	&\bE\bigg[\sup_{t\in [\tau,T]}\|u_1(t)\|^2\bigg]+\bE\int_{\tau}^{T}\|\nabla u_1(s)\|^2\rd s+\bE \int_{\tau}^{T}\|\psi_1^i(s)\|^2\rd s\\
	&\le 2\tilde C_1\bE\bigg[\|g^0\|^2+|{L^g}|^2\|\rho_1(T)\|^2\bigg]
	+ 2 \tilde C_1 \bE\int_{\tau}^{T}\|f^0_s\|^2\rd s+2 \tilde{C}_2\bE\int_{\tau}^{T}\|u_1(s)\|^2\rd s\\
	&\quad\quad+2 \tilde C_1|{L_1^G}|^2(T-\tau)\bE\bigg[\int_{\tau}^{T}\bigg(\|\rho_1(s)\|^2+\|\nabla\rho_1(s)\|^2\bigg)\rd s\bigg],
	\end{align*}
 which by Gronwall's inequality indicates that
	\begin{align*}
	&\bE\bigg[\sup_{t\in [0,T]}\|u_1(t)\|^2\bigg]+\bE\int_{0}^{T}\|\nabla u_1(s)\|^2\rd s+\bE\int_{0}^{T}\|\psi_1(s)\|^2\rd s\\
	&\le \bar{C}_2\bE\bigg[\|g^0\|^2+|{L^g}|^2\|\rho_1(T)\|^2\bigg]+ \bar{C}_2\bE\bigg[\int_{0}^{T}\|G^0_s\|^2\rd s\bigg]\\
	&\quad\quad+\bar{C}_2|{L_1^G}|^2T\bE\bigg[\int_{0}^{T}\bigg(\|\rho_1(s)\|^2+\|\nabla\rho_1(s)\|^2\bigg)\rd s\bigg],
	\end{align*}
	where $\bar{C}_2=2 \tilde C_1 e^{2\tilde{C}_2T}$, with $\tilde{C}_2$ depending on $k$ and $L_2^G$. 
	
	Set $\hat{\rho}=\rho_1-\rho_2,\hat{u}=u_1-u_2,\hat{\psi}=\psi_1-\psi_2$. Analogous calculations for $\hat{u}$ yield that
	\begin{align}
	&\bE\bigg[\sup_{t\in [0,T]}\|\hat{u}(t)\|^2\bigg]+\bE\int_{0}^{T}\left(\|\nabla \hat{u}(s)\|^2+ \|\hat{\psi}(s)\|^2\right)\rd s
	\nonumber\\
	&\le 
	\bar{C}_2|{L^g}|^2\bE\bigg[\|\hat{\rho}(T)\|^2\bigg] 
	+\bar{C}_2|{L_1^G}|^2T\bE\int_{0}^{T}\bigg(\|\hat{\rho}(s)\|^2+\|\nabla\hat{\rho}(s)\|^2\bigg)\rd s
	\nonumber\\
	&\le\bar{C}_2|{L^g}|^2\bE\bigg[\sup_{s\in [0,T]}\|\hat{\rho}(s)\|^2\bigg] 
	+\bar{C}_2|{L_1^G}|^2T\bE\bigg[T\cdot\sup_{s\in [0,T]}\|\hat{\rho}(s)\|^2+\int_{0}^{T}\|\nabla\hat{\rho}(s)\|^2\bigg]\rd s
	\nonumber\\
	&\le \bar{C}_2\cdot\max{\{|{L_1^G}|^2T^2+|{L^g}|^2,|{L_1^G}|^2T\}}\bE\bigg[\sup_{s\in [0,T]}\|\hat{\rho}(s)\|^2+\int_{0}^{T}\|\nabla\hat{\rho}(s)\|^2\bigg]\rd s.
	\label{RBw2}
	\end{align}
 Substituting \eqref{Rf1} into \eqref{RBw2} we finally have
	%\begin{equation*}
	%\begin{split}
	%\bE\bigg[&\sup_{t\in [0,T]}\|\hat{u}(t)\|^2\bigg]+\bE\int_{0}^{T}\left(\|\nabla \hat{u}(s)\|^2+\sum_{i=1}^{k}\|\hat{\psi}^i(s)\|^2\right)\rd s\\
	%&\le \tilde{C}_1\tilde{C}_2.\max{\{|{L_1^G}|^2T^2+|{L^g}|^2,|{L_1^G}|^2T\}}.|{L_2^F}|^2T\bE\int_{0}^{T}\left(\|\hat{u}'(s)\|^2+\|\nabla \hat{u}'(s)\|^2+\sum_{i=1}^{k}\|{\hat{\psi'}}^i(s)\|^2\rd s\right),
	%\end{split}
	%\end{equation*}
	%that is,
	\begin{align*}
	&\bE\bigg[\sup_{t\in [0,T]}\|\hat{u}(t)\|^2\bigg]+\bE\int_{0}^{T}\left(\|\nabla \hat{u}(s)\|^2+ \|\hat{\psi}(s)\|^2\right)\rd s
	\\
	&
	\le
	 \hat{C}\bE\bigg[\sup_{s\in [0,T]}\|\hat{u}'(s)\|^2+\int_{0}^{T}\left(\|\nabla \hat{u}'(s)\|^2
	 +\|{\hat{\psi'}}(s)\|^2\rd s\right)\bigg],
	\end{align*}
	where $\hat{C}=\bar{C}e^{\bar{C}T}\cdot\max{\{|{L_1^G}|^2T^2+|{L^g}|^2,|{L_1^G}|^2T\}}\cdot T\cdot\max \{|L_2^F|^2,|L_2^F|^2T+|L_2^f|^2\}$ with $\bar{C} $ depending on $L_1^f,L^F_1$ and $L_2^G$.
	Hence we can conclude that  as long as 
	 $$\hat{C}=\bar{C}e^{\bar{C}T}\cdot\max{\{|{L_1^G}|^2T^2+|{L^g}|^2,|{L_1^G}|^2T\}}\cdot T\cdot\max \{|L_2^F|^2,|L_2^F|^2T+|L_2^f|^2\}<1,$$
	 the mapping $(u_1',{\psi_1'})\mapsto(u_1,\psi_1)$ is a contraction and FBSPDE \eqref{FSPDE}-\eqref{BSPDE} admits a unique weak solution $(\rho,u,\psi)$ which satisfies
	\begin{align}
	&\bE\bigg[\sup_{t\in [0,T]}\|\rho(t)\|^2+\sup_{t\in [0,T]}\|u(t)\|^2\bigg]+\bE\int_{0}^{T}\left(\|\nabla\rho(s)\|^2+\|\nabla u(s)\|^2+ \|\psi(s)\|^2\right)\rd s
	\nonumber\\
	&\le C\bE\left[ \|\rho_0\|^2+\|g^0\|^2+\int_{0}^{T}\|F^0_s\|^2\rd s+\int_{0}^{T}\|f^{0}_s\|^2\rd s+ \int_{0}^{T}\|G^0_s\|^2\rd s\right],
	\label{est-q-weak-solution}
	\end{align}
	with $C=C(L_1^f,L_2^f,L^g,L^F_1,L^F_2,L^G_1,L^G_2,T)$.
	
	Now, as $(\rho,u,\psi)$ is the unique weak solution to FBSPDE \eqref{FSPDE}-\eqref{BSPDE}, Assumptions \ref{ass1}-\ref{ass2} allow us to further check that $F(\rho,\nabla\rho,u,\nabla u,\psi)\in\cL^2(L^2)$ and $f(\rho)\in \cL^2(H^{1,2})$. Then by \cite[Theorem 2.5]{KaiDu2020}, $\rho$ is the strong solution of FSPDE \eqref{FSPDE}; similarly,  \cite[Theorem 3.1]{du2011strong} implies that $(u,\psi)$ is the strong solution of BSPDE \eqref{BSPDE}. Therefore, FBSPDE \eqref{FSPDE}-\eqref{BSPDE} admits a unique  strong solution  $(\rho,u,\psi)$ and as a straightforward consequence of \eqref{est-q-weak-solution} and the estimates in  \cite[Theorem 2.5]{KaiDu2020} and \cite[Theorem 3.1]{du2011strong}, it holds that
	\begin{align*}
	&\bE\bigg[\sup_{t\in [0,T]}\|\nabla\rho(t)\|^2+\sup_{t\in [0,T]}\|\nabla u(t)\|^2\bigg]+\bE\int_{0}^{T}\left(\|\Delta\rho(s)\|^2+\|\Delta u(s)\|^2+ \|\psi(s)\|_{1,2}^2\right)\rd s\\
	&\le C\bE\left[ \|\rho_0\|^2+\|\nabla\rho_0\|^2+\| g^0\|^2+\int_{0}^{T}\|F^0_s\|^2\rd s+ \int_{0}^{T}\| f^{0}_s\|^2\rd s+ \int_{0}^{T}\|G^0_s\|^2\rd s\right],
	\end{align*}
	with $C=C(L_1^f,L_2^f,\tilde{L}^f,L^g,\tilde{L}^g,L^F_1,L^F_2,L^G_1,L^G_2,T)$.
\end{proof}

\section{Rate of Convergence for Semi-Discrete Approximation}
Let $(\rho,u,\psi)$ the solution to FBSPDE \eqref{FSPDE}-\eqref{BSPDE}. With the finite element method, we approximate it by the triple $(\rho_h,u_h,\psi_h)$ satisfying the following FBSDE\footnote{As the equations \eqref{dFSPDE} and \eqref{dBSPDE} are essentially finite-dimensional, we call it an FBSDE here.}:
%
%Let $\rho_h(t),u_h(t),\psi_h(t)$ be the approximations of $\rho(t),u(t)$ and $\psi(t)$ on the finite dimensional space $V_h^0$, and they satisfy the spatially discrete version of FBSPDE \eqref{FSPDE}-\eqref{BSPDE}
%For all $t\in [0,T],$ there hold $\mathbb{P}$ a.s.
\begin{equation} \label{dFSPDE}
\left\{\begin{array}{l}
\begin{split}
\mathrm{d}\rho_h(t)=&\bigg(\Delta_h \rho_h(t)+\Pi_h F(t,x,\rho_h(t),\nabla\rho_h(t),u_h(t),\nabla u_h(t),\psi_h(t))\bigg)\,
\mathrm{d}t\\
&-\sum_{i=1}^{k} \Pi_h f^i(t,\rho_h(t),u_h(t))\mathrm{d}W^i(t),\quad t\in (0,T];\\
\rho_h(0)=&\rho_0,\\
\end{split}
\end{array}\right.
\end{equation}
and
\begin{equation} \label{dBSPDE}
\left\{\begin{array}{l}
\begin{split}
-\mathrm{d}u_h(t)&=\bigg(\Delta_h u_h(t)+ \Pi_h G\big(t,x,\rho_h(t),\nabla\rho_h(t),u_h(t),\nabla u_h(t),\psi_h(t)\big)\bigg)\,\mathrm{d}t
\\
&\quad -\sum_{i=1}^{k}\psi_h^i(t) \mathrm{d}W^i(t), \quad t\in [0,T);\\
u_h(T)&=\Pi_h g(\rho_h(T)).\\
\end{split}
\end{array}\right.
\end{equation}
In what follows, we set
$$\vrh:=\rho-\rho_h,\quad \cU:=u-u_h,\quad \zeta:=\psi-\psi_h.$$
%Using assumption \ref{ass1} and FE estimate Theorem \ref{FE}, we can have 
%\begin{equation*}
%	\begin{split}
%	&\|\Pi_h F(\omega,t,\rho_1,\nabla\rho_1,u_1,\nabla u_1,\psi_1)-\Pi_h F(\omega,t,\rho_2,\nabla\rho_2,u_2,\nabla u_2,\psi_2)\|\\
%	& \leq C_e \|F(\omega,t,\rho_1,\nabla\rho_1,u_1,\nabla u_1,\psi_1)- F(\omega,t,\rho_2,\nabla\rho_2,u_2,\nabla u_2,\psi_2)\|\\ 
%	&\leq C_eL^F_1\bigg(\|\rho_1-\rho_2\|+\|\nabla\rho_1-\nabla\rho_2\|\bigg)+C_eL^F_2\bigg(\|u_2-u_1\|+\|\nabla u_2-\nabla u_1\|+\|\psi_2-\psi_1\|\bigg).\\
%	&= L^{F_e}_1\bigg(\|\rho_1-\rho_2\|+\|\nabla\rho_1-\nabla\rho_2\|\bigg)+L^{F_e}_2\bigg(\|u_2-u_1\|+\|\nabla u_2-\nabla u_1\|+\|\psi_2-\psi_1\|\bigg).
%	\end{split}
%\end{equation*}
%with $L^{F_e}_1=C_eL^F_1,\;L^{F_e}_2=C_eL^F_2$. Similarly we can define $L^{G_e}_1,L^{G_e}_2,L^{f_e}$ and $L^{g_e}$. Now we have the following Lemma. 

%In this section, first, the wellposedness of FBSDE \eqref{dFSPDE}-
%\eqref{dBSPDE} will be addressed in Lemma \ref{lem1}. Then Lemma \ref{lem2} gives the estimate of convergence rate for the approximation $\rho_h$ of the forward process $\rho$. Finally Lemma \ref{thm3} gives the convergence rate for approximations $\rho_h,u_h$ and $\psi_h$. 

First comes the the wellposedness of FBSDE \eqref{dFSPDE}-\eqref{dBSPDE}.
\begin{lem}\label{lem1}
	Let assumptions \ref{ass1}-\ref{ass2} hold. Then for each fixed $h>0$ there exists $\;\bar{C}=\bar{C}(L^{f},L_1^{F},L^{G}_2)$ such that if 
	\begin{equation}
		\bar{C} e^{\bar{C}T}\cdot\max{\{|L_1^G|^2T^2+|L^g|^2,|L^g|^2T\}}\cdot T\cdot\max\{|L_2^F|^2,|L_2^F|^2T+|L_2^f|^2\}<1,
	\end{equation}
	FBSDE \eqref{dFSPDE}-
	\eqref{dBSPDE} admits a unique solution $(\rho_h,u_h,\psi_h)\in\bigg(\cL^2(V_h^0)\cap\cS^2(V_h^0)\bigg)\times\bigg(\cL^2(V_h^0)\cap\cS^2(V_h^0)\bigg)\times \cL^2(V_h^0)$, with 
	\begin{equation*}
	\begin{split}
	&\begin{split}
	(i)\;&\bE\bigg[\sup_{t\in [0,T]}\|\rho_h(t)\|^2+\sup_{t\in [0,T]}\|u_h(t)\|^2\bigg]+\bE\bigg[\int_{0}^{T}\bigg(\|\nabla\rho_h(s)\|^2+\|\nabla u_h(s)\|^2+ \|\psi_h(s)\|^2\bigg)\rd s\bigg]\\
	&\le C_{1}\bE\bigg[\|\rho_0\|^2+\|g^0\|^2+\int_{0}^{T}\bigg(\|f_s^{0}\|^2+\|F^0_s\|^2+\|G^0_s\|^2\bigg)\rd s\bigg], \\
	\end{split}\\
	&\begin{split}
	(ii)\;&\bE\bigg[\sup_{t\in [0,T]}\|\nabla\rho_h(t)\|^2+\sup_{t\in [0,T]}\|\nabla u_h(t)\|^2\bigg]
	\\
	&
	+\bE\bigg[\int_{0}^{T}\bigg(\|\Delta_h\rho_h(s)\|^2+\|\Delta_h u_h(s)\|^2
	+\|\psi_h(s)\|_{1,2}^2\bigg)\rd s\bigg]\\
	&\le C_{2}\bE\bigg[1+\|\rho_0\|^2+\|\nabla\rho_0\|^2+\|g^0\|^2+\int_{0}^{T}\bigg(\|f_s^{0}\|^2+\|F^0_s\|^2+\|G^0_s\|^2\bigg)\rd s\bigg],\\
	\end{split}\\
	\end{split}
	\end{equation*}
	where the constants $C_{1}=C_{1}(L_1^{f},L_2^f,L^{g},L^{F}_1,L^{F}_2,L_1^{G},L_2^{G},T,C_e)$ and\\  $C_{2}=C_{2}(L_1^{f},L_2^f,\tilde{L}^{f},L^{g},\tilde{L}^{g},L^{F}_1,L^{F}_2,L_1^{G},L_2^{G},T,C_e)$ are independent of $h$.
\end{lem}
As the computations involved in the proof of Lemma \ref{lem1} are more or less standard, its proof is postponed to the appendix. In what follows, we denote
\begin{align*}
Q_0:=\bE\left[1+\|\rho_0\|_{1,2}^2+\|g^0\|^2
+\int_0^T \left(\|f^0_s\|^2+\|F^{0}_s\|^2+\|G^0_s\|^2\right) \,ds \right].
\end{align*}

Next, we shall prove the error estimate for $\rho$ in terms of $(\cU,\zeta)$. Recalling the $L^2$-projection $\Pi_h$ and the associated error estimates in Theorem \ref{FEThm},we have for $\xi\in H^{1,2}_0$,
\begin{equation}\label{FEResult}
\begin{split}
\|\xi\|^2&=\|\Pi_h \xi\|^2+\|(id-\Pi_h)\xi\|^2\le \|\Pi_h \xi\|^2+C_e^2h^2\|\xi\|^2_{1,2},\quad\mathrm{and}\quad
\|\Pi_h\xi\|^2\le\|\xi\|^2.
\end{split}
\end{equation}
If we assume further $\xi\in H^{2,2}_0$, applying the error estimates in Theorem \ref{FEThm} and the classical Elliptic equation theory (see\cite[Lemma 9.17]{GilbargTrudingerEllipticPDE} for instance) gives that
\begin{align}\label{FEResult2}
\|\left[id-\Pi_h\right]\xi\|_{1,2}\le C_eh\|\xi\|_{2,2}\le C h\|\Delta\xi\|,
\end{align}
where the constant $C$ only depends on the domain $D$ and dimension $d$, being independent of $h$.
\begin{lem}\label{lem2}
	Given $(u,\psi)\in\big(\cL^2(H^{2,2}_0)\cap\cS^2(H^{1,2}_0)\big)\times \cL^2(H^{1,2}_0)$, let $\rho$ be the strong solution of FSPDE (\ref{FSPDE}) and given $(u_h,\psi_h)\in\big(\cL^2(V_h^0)\cap\cS^2(V_h^0)\big)\times \cL^2(V_h^0) $, $\rho_h$ be the  solution of FSDE (\ref{dFSPDE}). Then it holds that for all $\tau\in[0,T]$,
	\begin{equation*}
	\begin{split}
	&\bE\left[\sup_{t\in [0,\tau]}\|\vrh(t)\|^2\right]+\bE\int_{0}^{\tau}\|\nabla\vrh(s)\|^2\rd s\\
	&\le CQ_0h^2+ \bar{C_1}\tau\cdot\max\{|L_2^F|^2,|L_2^F|^2\tau+|L_2^f|^2\} \bE\bigg[\sup_{t\in [0,\tau]}\|\cU(t)\|^2+\int_{0}^{\tau}\bigg(\|\nabla \cU(s)\|^2+\|\zeta(s)\|^2\bigg)\rd s\bigg],
	\end{split}
	\end{equation*}
	where constants $C= C(L_1^{f},L_2^f,\tilde{L}^{f},L^{g},\tilde{L}^{g},L^{F}_1,L^{F}_2,L_1^{G},L_2^{G},T,d,D)$, 
	$\bar{C_1}=C_1e^{C_1\cdot\tau}$, and $C_1=C_1(L_1^f,L_1^F)$ are independent of $h$.
	%Moreover
	%\begin{equation*}
	%	\bE\left[\int_{0}^{T}\|\rho(t)-\rho_h(t)\|^2\rd t\right]\le Ch^4+  CE\int_{0}^{T}\bigg(\|\nabla\Delta_h^{-1}\cU(s)\|^2+\|\cU(s)\|^2+\|\nabla\Delta_h^{-1}\zeta(s)\|^2\bigg)\rd s.
	%\end{equation*}
\end{lem}
\begin{proof} Subtracting (\ref{dFSPDE}) from (\ref{FSPDE}) leads to for $0\le t\le T$
	\begin{equation} \label{erFSPDE}
	\left\{\begin{array}{l}
	\begin{split}
	&\rd\Pi_h\vrh(t)=\Delta_h\cR_h\vrh(t)\rd t-\sum_{i=1}^{k}\Pi_h\bigg[f^i(t,\rho_h(t),u_h(t))-f^i(t,\rho(t),u(t))\bigg] \mathrm{d}W^i_t\\
	&\quad+ \Pi_h\bigg[F\big(t,\rho_h(t),\nabla\rho_h(t),u_h(t),\nabla u_h(t),\psi_h(t)\big)-F\big(t,\rho(t),\nabla\rho(t),u(t),\nabla u(t),\psi(t)\big)\bigg]\,
	\mathrm{d}t,\\
	&\Pi_h\vrh(0)=0.
	\end{split}
	\end{array}\right.
	\end{equation}
	
	Fix $t\in[0,T]$. Applying It\^o's formula to equation (\ref{erFSPDE}) for $\Pi_h\vrh$ yields that $\bP$-a.s.
	%\begingroup\makeatletter\def\f@size{11}\check@mathfonts
	\begin{align*}
	&\|\Pi_h\vrh(t)\|^2
	-\int_{0}^{t}\left\|\Pi_h\bigg[f(s,\rho_h(s),u_h(s))-f(s,\rho(s),u(s))\bigg]\right\|^2\rd s
	\\
	&-2\int_{0}^{t}\bigg\langle\Pi_h\bigg[F\big(s,\rho_h(s),\nabla\rho_h(s),u_h(s),\nabla u_h(s),\psi_h(s)\big)
	\\
	&\quad\quad\quad\quad
	-F\big(s,\rho(s),\nabla\rho(s),u(s),\nabla u(s),\psi(s)\big)\bigg],\Pi_h\vrh(s)\bigg\rangle\mathrm{d}s
	\\
	&-2\sum_{i=1}^{k}\int_{0}^{t}\left\langle\Pi_h\bigg[f^i(s,\rho_h(s),u_h(s))-f^i(s,\rho(s),u(s))\bigg],\Pi_h\vrh(s)\right\rangle \mathrm{d}W^i_s
	\\
	&=2\int_{0}^{t}\langle\Delta_h\cR_h\vrh(s),\Pi_h\vrh(s)\rangle\rd s
	\\
%	&\quad+2\int_{0}^{t}\bigg\langle\Pi_h\bigg[F\big(s,\rho_h(s),\nabla\rho_h(s),u_h(s),\nabla u_h(s),\psi_h(s)\big)\\
%	&\quad-F\big(s,\rho(s),\nabla\rho(s),u(s),\nabla u(s),\psi(s)\big)\bigg],\Pi_h\vrh(s)\bigg\rangle\mathrm{d}s\\
%	&\quad -2\sum_{i=1}^{k}\int_{0}^{t}\left\langle\Pi_h\bigg[f^i(s,\rho_h(s),u_h(s))-f^i(s,\rho(s),u(s))\bigg],\Pi_h\vrh(s)\right\rangle \mathrm{d}W^i_s\\
	&=-2\int_{0}^{t}\langle\nabla\cR_h\vrh(s),\nabla\Pi_h\vrh(s)\rangle\rd s
	\\
%	+\sum_{i=1}^{k}\int_{0}^{t}\left\|\Pi_h\bigg[f^i(s,\rho_h(s),u_h(s))-f^i(s,\rho(s),u(s))\bigg]\right\|^2\rd s\\
%	&\quad+2\int_{0}^{t}\bigg\langle\Pi_h\bigg[F\big(s,\rho_h(s),\nabla\rho_h(s),u_h(s),\nabla u_h(s),\psi_h(s)\big)\\
%	&\quad-F\big(s,\rho(s),\nabla\rho(s),u(s),\nabla u(s),\psi(s)\big)\bigg],\Pi_h\vrh(s)\bigg\rangle\mathrm{d}s\\
%	&\quad -2\sum_{i=1}^{k}\int_{0}^{t}\left\langle\Pi_h\bigg[f^i(s,\rho_h(s),u_h(s))-f^i(s,\rho(s),u(s))\bigg],\Pi_h\vrh(s)\right\rangle \mathrm{d}W^i_s\\
	&=-2\int_{0}^{t}\langle\nabla\vrh(s),\nabla\Pi_h\vrh(s)\rangle\rd s,
%	+\sum_{i=1}^{k}\int_{0}^{t}\left\|\Pi_h\bigg[f^i(s,\rho_h(s),u_h(s))-f^i(s,\rho(s),u(s))\bigg]\right\|^2\rd s\\
%	&\quad+2\int_{0}^{t}\bigg\langle\Pi_h\bigg[F\big(s,\rho_h(s),\nabla\rho_h(s),u_h(s),\nabla u_h(s),\psi_h(s)\big)\\
%	&\quad-F\big(s,\rho(s),\nabla\rho(s),u(s),\nabla u(s),\psi(s)\big)\bigg],\Pi_h\vrh(s)\bigg\rangle\mathrm{d}s\\
%	&\quad -2\sum_{i=1}^{k}\int_{0}^{t}\left\langle\Pi_h\bigg[f^i(s,\rho_h(s),u_h(s))-f^i(s,\rho(s),u(s))\bigg],\Pi_h\vrh(s)\right\rangle \mathrm{d}W^i_s,
	\end{align*}
	and  this gives
	\begin{align}\label{lem2ref1}
	&\|\Pi_h\vrh(t)\|^2+2\int_{0}^{t}\|\nabla\vrh(s)\|^2\rd s\nonumber\\
	&=
	2\int_{0}^{t}\left\langle\nabla\vrh(s),\nabla\left[id-\Pi_h\right]\vrh(s)\right\rangle\rd s
	%\nonumber\\
	%&\quad\quad
	+\sum_{i=1}^{k}\int_{0}^{t}\left\|\Pi_h\bigg[f^i(s,\rho_h(s),u_h(s))-f^i(s,\rho(s),u(s))\bigg]\right\|^2\rd s\nonumber\\
	&\quad\quad+2\int_{0}^{t}\bigg\langle\Pi_h\bigg[F\big(s,\rho_h(s),\nabla\rho_h(s),u_h(s),\nabla u_h(s),\psi_h(s)\big)\nonumber\\
	&\quad\quad-F\big(s,\rho(s),\nabla\rho(s),u(s),\nabla u(s),\psi(s)\big)\bigg],\Pi_h\vrh(s)\bigg\rangle\mathrm{d}s\nonumber\\
	&\quad\quad -2\sum_{i=1}^{k}\int_{0}^{t}\left\langle\Pi_h\bigg[f^i(s,\rho_h(s),u_h(s))-f^i(s,\rho(s),u(s))\bigg],\Pi_h\vrh(s)\right\rangle \mathrm{d}W^i_s.
	\end{align}
	%\endgroup
	Using the stability of $L^2$-projection and the Lipschitz property from Assumption \ref{ass1}, we have 
	%\begingroup\makeatletter\def\f@size{11}\check@mathfonts
	\begin{align*}
%	\begin{split}
	&2\int_{0}^{t}\bigg\langle\Pi_h\bigg[F\big(s,\rho_h(s),\nabla\rho_h(s),u_h(s),\nabla u_h(s),\psi_h(s)\big)\\
	&\quad\quad-F\big(s,\rho(s),\nabla\rho(s),u(s),\nabla u(s),\psi(s)\big)\bigg],\Pi_h\vrh(s)\bigg\rangle\rd s\\
	&\le 2\int_{0}^{t} \bigg\|\Pi_h\bigg[F\big(s,\rho_h(s),\nabla\rho_h(s),u_h(s),\nabla u_h(s),\psi_h(s)\big)\\
	&\quad\quad-F\big(s,\rho(s),\nabla\rho(s),u(s),\nabla u(s),\psi(s)\big)\bigg]\bigg\|\cdot\bigg\|\Pi_h\vrh(s)\bigg\|\rd s\\
	& \le 2\int_{0}^{t}\bigg(L_1^F\big(\|\vrh(s)\|+\|\nabla\vrh(s)\|\big)+L_2^F\big(\|\cU(s)\|+\|\nabla\cU(s)\|+\|\zeta(s)\|\big)\bigg)\cdot\bigg\|\Pi_h\vrh(s)\bigg\|\rd s\\
	&\le |L_2^F|^2t\int_{0}^{t}\bigg(\frac{1}{\eps_2}\|\cU(s)\|^2+\frac{1}{\eps_3}\|\nabla \cU(s)\|^2+\frac{1}{\eps_4}\|\zeta(s)\|^2\bigg)\rd s+\eps_1\int_{0}^{t}\|\nabla\vrh(s)\|^2\rd s\\
	&\quad\quad+\bigg(\frac{|L_1^F|^2}{\eps_1}+2L_1^F\bigg)\int_{0}^{t}\|\vrh(s)\|^2\rd s+(\eps_2+\eps_3+\eps_4)\sup_{s\in [0,t]}\|\Pi_h\vrh(s)\|^2.
	%\end{split}
	\end{align*}
	%\endgroup
	%$--------------------------------------------------$
	Taking supremum over $t\in[0,\tau]$ for $\tau\in[0,T]$ in \eqref{lem2ref1}, and using deductions as above, we may arrive at
	\begin{align*}
	&(1-\eps_2-\eps_3-\eps_4)\sup_{t\in [0,\tau]}\|\Pi_h\vrh(t)\|^2+(2-\eps_5)\int_{0}^{\tau}\|\nabla\vrh(s)\|^2\rd s\\
	&\le\frac{1}{\eps_5}\int_{0}^{\tau}\|\nabla\left[id-\Pi_h\right]\vrh(s)\|^2\rd s+\eps_1\int_{0}^{\tau}\|\nabla\vrh(s)\|^2\rd s+\bigg(\frac{|L_1^F|^2}{\eps_1}+2L_1^F+2|L_1^f|^2\bigg)\int_{0}^{\tau}\|\vrh(s)\|^2\rd s \\
	&\quad\quad+|L_2^F|^2\tau\int_{0}^{\tau}\bigg(\frac{1}{\eps_2}\|\cU(s)\|^2+\frac{1}{\eps_3}\|\nabla\cU(s)\|^2+\frac{1}{\eps_4}\|\zeta(s)\|^2\bigg)\rd s+2|L_2^f|^2\int_{0}^{\tau}\|\cU(s)\|^2\rd s\\
	&\quad\quad+2\sup_{t\in [0,\tau]}\sum_{i=1}^{k}\bigg|\int_{0}^{t}\left\langle\Pi_h\bigg[f^i(s,\rho_h(s),u_h(s))-f^i(s,\rho(s),u(s))\bigg],\Pi_h\vrh(s)\right\rangle \mathrm{d}W^i_s\bigg|.
	\end{align*}
	Take $\eps_2=\eps_3=\eps_4=\frac{1}{6},\eps_5=1$, and $\eps_1=\frac{3}{4}$. It follows that
	\begin{align*}
	&\sup_{t\in [0,\tau]}\|\Pi_h\vrh(t)\|^2+\frac{1}{2}\int_{0}^{\tau}\|\nabla\vrh(s)\|^2\rd s\\
	&\le 2\int_{0}^{\tau}\|\nabla\left[id-\Pi_h\right]\vrh(s)\|^2\rd s +12|L_2^F|^2\tau\int_{0}^{\tau}\bigg(\|\cU(s)\|^2+\|\nabla \cU(s)\|^2+\|\zeta(s)\|^2\bigg)\rd s\\
	&\quad\quad+4|L_2^f|^2\int_{0}^{\tau}\|\cU(s)\|^2\rd s+C'_1\int_{0}^{\tau}\|\vrh(s)\|^2\rd s\\
	&\quad\quad+4\sup_{t\in [0,\tau]}\sum_{i=1}^{k}\bigg|\int_{0}^{t}\left\langle\Pi_h\bigg[f^i(s,\rho_h(s),u_h(s))-f^i(s,\rho(s),u(s))\bigg],\Pi_h\vrh(s)\right\rangle \mathrm{d}W^i(s)\bigg|,
	\end{align*}
	with the constant $C'_1=C'_1(L_1^f,L_1^F)$.\\
	Using \eqref{FEResult2} we can have $\|\nabla\left[id-\Pi_h\right]\vrh(s)\|\le C \|\varrho(s)\|_{1,2}\le Ch\|\Delta\varrho(s)\|$, and thus
	\begin{align}\label{lem2ref2}
	&\sup_{t\in [0,\tau]}\|\Pi_h\vrh(t)\|^2+\frac{1}{2}\int_{0}^{\tau}\|\nabla\vrh(s)\|^2\rd s\nonumber\\
	&\le 
	Ch^2 \int_0^{\tau} \|\Delta\varrho(s)\|^2\,ds
	+ 12|L_2^F|^2\tau\int_{0}^{\tau}\bigg(\|\cU(s)\|^2+\|\nabla \cU(s)\|^2+\|\zeta(s)\|^2\bigg)\rd s\nonumber\\
	&\quad\quad+4\sup_{t\in [0,\tau]}\sum_{i=1}^{k}\bigg|\int_{0}^{t}\left\langle\Pi_h\bigg[f^i(s,\rho_h(s),u_h(s))-f^i(s,\rho(s),u(s))\bigg],\Pi_h\vrh(s)\right\rangle \mathrm{d}W^i(s)\bigg|\nonumber\\
	&\quad\quad
	+4|L_2^f|^2\int_{0}^{\tau}\|\cU(s)\|^2\rd s
	+C'_1\int_{0}^{\tau}\|\vrh(s)\|^2\rd s,
	\end{align}
	where $C=C(d,D)$.\\ 
	We take expectations and use BDG inequality for the terms involving stochastic integrals to obtain
	\begin{align*}
	&4\bE\bigg[\sup_{t\in [0,\tau]}\bigg|\int_{0}^{t}\sum_{i=1}^{k}\left\langle\Pi_h\left[f^i(s,\rho_h(s),u_h(s))-f^i(s,\rho(s),u(s))\right],\Pi_h\vrh(s)\right\rangle \mathrm{d}W^i(s)\bigg|\bigg]\\
	&\le \tilde{C}\bE\bigg[\bigg(\int_{0}^{\tau}\sum_{i=1}^{k}\left\|\Pi_h\left[f^i(s,\rho_h(s),u_h(s))-f^i(s,\rho(s),u(s))\right]\right\|^2\cdot\left\|\Pi_h\vrh(s)\right\|^2 \mathrm{d}s\bigg)^\frac{1}{2}\bigg]\\
	&\le 
	\tilde{C}\bE\bigg[\bigg(2\int_{0}^{\tau}\bigg(|L_1^f|^2\left\|\vrh(s)\right\|^2+|L_2^f|^2\|\cU(s)\|^2\bigg)\cdot\big\|\Pi_h\vrh(s)\big\|^2 \mathrm{d}s\bigg)^\frac{1}{2}\bigg]\\
	&\le 
	\tilde{C}\bE\bigg[\bigg(2\sup_{t\in [0,\tau]}\big\|\Pi_h\vrh(t)\big\|^2\int_{0}^{\tau}\bigg(|L_1^f|^2\left\|\vrh(s)\right\|^2+|L_2^f|^2\|\cU(s)\|^2\bigg) \mathrm{d}s\bigg)^\frac{1}{2}\bigg]\\
	&\le
	\frac{1}{2}\bE\left[\sup_{t\in [0,\tau]}\big\|\Pi_h\vrh(t)\big\|^2\right]+ 4 |\tilde{C}L_1^f|^2\bE\int_{0}^{\tau}\big\|\vrh(s)\big\|^2 \mathrm{d}s+ 4|\tilde{C}L_2^f|^2\bE\int_{0}^{\tau}\big\|\cU(s)\big\|^2 \mathrm{d}s,
	\end{align*}
	which yields from \eqref{lem2ref2} that
	%\begingroup\makeatletter\def\f@size{10}\check@mathfonts
	\begin{equation*}
	\begin{split}
	&\bE\left[\sup_{t\in [0,\tau]}\|\Pi_h\vrh(t)\|^2\right]+\bE\int_{0}^{\tau}\|\nabla\vrh(s)\|^2\rd s\\
	&\le Ch^2\bE\int_0^{\tau} \|\Delta\varrho(s)\|^2
	+ 24|L_2^F|^2\tau \bE\int_{0}^{\tau}\bigg(\|\cU(s)\|^2+\|\nabla \cU(s)\|^2+\|\zeta(s)\|^2\bigg)\rd s+C_1\bE\int_{0}^{\tau}\|\vrh(s)\|^2\rd s\\
	&+4(\tilde{C}^2+1)|L_2^f|^2\tau \bE\bigg[\sup_{t\in [0,\tau]}\|\cU(t)\|^2\bigg],
	\end{split}
	\end{equation*}
	with $C_1=C_1(L_1^f,L_1^F)$.
	%\endgroup
	\par By \eqref{FEResult}, there holds
	\begin{equation}
	\bE\left[\sup_{t\in [0,\tau]}\|\Pi_h\vrh(t)\|^2\right]\ge \bE\left[\sup_{t\in [0,\tau]}\|\vrh(t)\|^2\right]-Ch^2\bE\left[\sup_{t\in [0,\tau]}\|\varrho(t)\|_{1,2}^2\right].
	\end{equation}
	In view of Theorem \ref{Thm1} and Lemma \ref{lem1}, we have  
	\begin{equation*}
	\begin{split}
	&\bE\left[\sup_{t\in [0,\tau]}\|\vrh(t)\|^2\right]+\bE\int_{0}^{\tau}\|\nabla\vrh(s)\|^2\rd s\\
	&\le CQ_0h^2+ 24|L_2^F|^2\tau \bE\int_{0}^{\tau}\bigg(\|\cU(s)\|^2+\|\nabla \cU(s)\|^2+\|\zeta(s)\|^2\bigg)\rd s+C_1\bE\int_{0}^{\tau}\|\vrh(s)\|^2\rd s
	\\
	&
	+4(\tilde{C}^2+1)|L_2^f|^2\tau \bE\bigg[\sup_{t\in [0,\tau]}\|\cU(t)\|^2\bigg],
	\end{split}
	\end{equation*}
	where $C=C(L_1^{f},L_2^f,\tilde{L}^{f},L^{g},\tilde{L}^{g},L^{F}_1,L^{F}_2,L_1^{G},L_2^{G},T,d,D)$.
	This together with Gronwall's inequality further implies that for all $\tau\in[0,T]$
	\begin{align*}
	%\begin{split}
	&\bE\left[\sup_{t\in [0,\tau]}\|\vrh(t)\|\right]+\bE\int_{0}^{\tau}\|\nabla\vrh(t)\|^2\rd t\\
	&\le CQ_0h^2+ \bar{C}_1|L_2^F|^2\tau \bE\int_{0}^{\tau}\bigg(\|\cU(s)\|^2+\|\nabla \cU(s)\|^2+\|\zeta(s)\|^2\bigg)\rd s+\bar{C}_1|L_2^f|^2\tau \bE\bigg[\sup_{t\in [0,\tau]}\|\cU(t)\|^2\bigg]\\
	&\le CQ_0h^2+ \bar{C}_1|L_2^F|^2\tau \bE\int_{0}^{\tau}\bigg(\|\nabla \cU(s)\|^2+\|\zeta(s)\|^2\bigg)\rd s+\bar{C}_1\tau\bigg(|L_2^F|^2\tau+|L_2^f|^2\bigg) \bE\bigg[\sup_{t\in [0,\tau]}\|\cU(t)\|^2\bigg]\\
	&\le CQ_0h^2+ \bar{C}_1\tau\cdot\max\{|L_2^F|^2,|L_2^F|^2\tau+|L_2^f|^2\} \bE\bigg[\sup_{t\in [0,\tau]}\|\cU(t)\|^2+\int_{0}^{\tau}\bigg(\|\nabla \cU(s)\|^2+\|\zeta(s)\|^2\bigg)\rd s\bigg], 
	%\end{split}
	\end{align*}
	where $\bar{C}_1=24\big(\tilde{C}^2+1\big)\cdot e^{C_1\tau}$.
\end{proof}

Now, we are ready to give the convergence rate for the semi-discrete approximation.
\begin{thm}\label{thm3}
	Let $(\rho,u,\psi)\in\bigg(\cL^2(H^{2,2}_0)\cap\cS^2(H^{1,2}_0)\bigg)\times\bigg(\cL^2(H^{2,2}_0)\cap\cS^2(H^{1,2}_0)\bigg)\times \cL^2(H^{1,2}_0)$ be the strong solution of 
	FBSPDE \eqref{FSPDE}-\eqref{BSPDE} and $(\rho_h,u_h,\psi_h)\in\bigg(\cL^2(V_h^0)\cap\cS^2(V_h^0)\bigg)\times\bigg(\cL^2(V_h^0)\cap\cS^2(V_h^0)\bigg)\times \cL^2(V_h^0) $ be the solution of FBSDE \eqref{dFSPDE}-\eqref{dBSPDE}. Then $\vrh(t):=\rho(t)-\rho_h(t),\cU(t):=u(t)-u_h(t)$ and $\zeta(t):=\psi(t)-\psi_h(t)$ satisfy
	\begin{equation*}
	\begin{split}
	\bE\bigg[\sup_{t\in [0,T]}\|\vrh(t)\|^2&+\sup_{t\in [0,T]}\|\cU(t)\|^2+\int_{0}^{T}\bigg(\|\nabla\vrh(t)\|^2+\|\nabla\cU(t)\|^2+\|\zeta(t)\|^2\bigg)\rd t\bigg]\le CQ_0h^2,
	\end{split}
	\end{equation*}
	where the constant $C=C(L_1^{f},L_2^f,\tilde{L}^{f},L^{g},\tilde{L}^{g},L^{F}_1,L^{F}_2,L_1^{G},L_2^{G},T,d,D)$.
	
	%Given $\rho\in\cL^2(H^{2,2})\cap \cS^2(H^{1,2})$, let $(u,\psi)$ be the strong solution of BSPDE (\ref{BSPDE}) and given $\rho\in\cL^2(V_h^0)\cap \cS^2(V_h^0)$, $(u_h,\psi_h)$ be the strong solution of BSPDE (\ref{dBSPDE}). Then with $\vrh(t):=\rho(t)-\rho_h(t),\cU(t):=u(t)-u_h(t)$ and $\zeta(t):=\psi(t)-\psi_h(t)$ and for $\tau\in[0,T]$
	%\begin{equation*}
	%\begin{split}
	%E\bigg[\sup_{t\in [0,\tau]}\|\cU(t)\|^2\bigg]&+E\sum_{i=1}^{k}\int_{0}^{\tau}\|\zeta^i(s)\|^2\rd s+E\int_{0}^{\tau}\|\nabla \cU(s)\|^2\rd s\le Ch^2\\&+C E\int_{0}^{\tau}\bigg(\|\cU(s)\|^2+\|\nabla \cU(s)\|^2+\|\zeta(s)\|^2\bigg)\rd s,
	%\end{split}
	%\end{equation*}
	%where $C=C(k,L^f,L^g,L_1^F,L_2^F,L_1^G,L_2^G,\tau)$.
	%	Moreover
	%		\begin{equation*}
	%	\bE\bigg[\int_{0}^{T}\|u(t)-u_h(t)\|^2\bigg]\le Ch^4+C\bE\int_{0}^{T}\bigg(\|\nabla\Delta_h^{-1}\vrh(s)\|^2+\| \vrh(s)\|^2\bigg)\rd s.
	%	\end{equation*}
\end{thm}
\begin{proof}
	Subtracting (\ref{dBSPDE}) from (\ref{BSPDE}) leads to for $0\le t\le T$
	%\begingroup\makeatletter\def\f@size{10}\check@mathfonts
	\begin{equation} \label{erBSPDE}
	\left\{\begin{array}{l}
	\begin{split}
	-\rd\Pi_h\cU(t)&=\bigg(\Delta_h\cR_h\cU(t)+ \Pi_h\bigg[G\big(t,\rho_h(t),\nabla\rho_h(t),u_h(t),\nabla u_h(t),\psi_h(t)\big)\\
	&-G\big(t,\rho(t),\nabla\rho(t),u(t),\nabla u(t),\psi(t)\big)\bigg]\bigg)\,
	\mathrm{d}t-\sum_{i=1}^{k}\Pi_h\zeta(t) \mathrm{d}W^i(t),\\
	\Pi_h\cU(0)&=\Pi_h\big[g(\rho(T))-g(\rho_h(T))\big].
	\end{split}
	\end{array}\right.
	\end{equation}
	%\endgroup
	Our computations will be divided into two parts.
	\par \textbf{Step 1}. This part is devoted to the estimate for $\zeta(t)$. Fix $t\in[0,T]$. Applying It\^o's formula to equation (\ref{erBSPDE}) for $\Pi_h\cU(t)$ yields $\bP$-a.s.
	%\begingroup\makeatletter\def\f@size{10}\check@mathfonts
	\begin{align}
	%\begin{split}
	&\|\Pi_h\cU(t)\|^2
	+2\sum_{i=1}^{k}\int_{t}^{T}\langle\Pi_h\zeta^i(s),\Pi_h\cU(s)\rangle\rd W^i_s
	+\int_{t}^{T}\|\Pi_h\zeta(s)\|^2\rd s
	-\|\Pi_h\big[g(\rho(T))-g(\rho_h(T))\big]\|^2
	\notag\\
	&
	\quad\quad-2\int_{t}^{T}\bigg\langle \Pi_h\bigg[G(t,\rho_h(t),\nabla\rho_h(t),u_h(t),\nabla u_h(t),\psi_h(t))
	\notag\\
	&\quad\quad\quad
	-G(t,\rho(t),\nabla\rho(t),u(t),\nabla u(t),\psi(t))\bigg],\Pi_h\cU(s)\bigg\rangle\rd s
	\notag\\
	&
	=2\int_{t}^{T}\langle\Delta_h\cR_h \cU(s), \Pi_h\cU(s)\rangle\rd s
	\notag\\
	&=-2\int_{t}^{T}\langle\nabla\cR_h \cU(s), \nabla\Pi_h\cU(s)\rangle\rd s
%	-2\sum_{i=1}^{k}\int_{t}^{T}\langle\Pi_h\zeta^i(s),\Pi_h\cU(s)\rangle\rd W^i_s\\
%	&
%	\quad-\sum_{i=1}^{k}\int_{t}^{T}\|\Pi_h\zeta^i(s)\|^2\rd s+2\int_{t}^{T}\bigg\langle \Pi_h\bigg[G(t,\rho_h(t),\nabla\rho_h(t),u_h(t),\nabla u_h(t),\psi_h(t))\\
%	&\quad-G(t,\rho(t),\nabla\rho(t),u(t),\nabla u(t),\psi(t))\bigg],\Pi_h\cU(s)\bigg\rangle\rd s+\|\Pi_h\big[g(\rho(T))-g(\rho_h(T))\big]\|^2\\
\notag\\
	&=-2\int_{t}^{T}\langle\nabla \cU(s),\nabla \Pi_h\cU(s)\rangle\rd s. \label{lem3ref1}
%	-2\sum_{i=1}^{k}\int_{t}^{T}\langle\Pi_h\zeta^i(s),\Pi_h\cU(s)\rangle\rd W^i_s-\sum_{i=1}^{k}\int_{t}^{T}\|\Pi_h\zeta^i(s)\|^2\rd s\\
%	&
%	\quad+2\int_{t}^{T}\bigg\langle \Pi_h\bigg[G(t,\rho_h(t),\nabla\rho_h(t),u_h(t),\nabla u_h(t),\psi_h(t))\\
%	&\quad-G(t,\rho(t),\nabla\rho(t),u(t),\nabla u(t),\psi(t))\bigg],\Pi_h\cU(s)\bigg\rangle\rd s+\|\Pi_h\big[g(\rho(T))-g(\rho_h(T))\big]\|^2,
	%\end{split}
	\end{align}
	Then using Assumption \ref{ass2}, we have
	%\begingroup\makeatletter\def\f@size{10}\check@mathfonts
	\begin{align}\label{lem3ref2}
	&2\int_{t}^{T}\bigg\langle \Pi_h\bigg[G(s,\rho_h(s),\nabla\rho_h(s),u_h(s),\nabla u_h(s),\psi_h(s))\nonumber\\
	&\quad\quad-G(s,\rho(s),\nabla\rho(s),u(s),\nabla u(s),\psi(s))\bigg],\Pi_h\cU(s)\bigg\rangle\rd s\nonumber\\
	&\le 2\int_{t}^{T}\bigg\| \Pi_h\bigg[G(s,\rho_h(s),\nabla\rho_h(s),u_h(s),\nabla u_h(s),\psi_h(s))\nonumber\\
	&\quad\quad-G(s,\rho(s),\nabla\rho(s),u(s),\nabla u(s),\psi(s))\bigg]\bigg\|\cdot\bigg\|\Pi_h\cU(s)\bigg\|\rd s \nonumber\\
	&\le 2\int_{t}^{T}  \bigg(L_1^G\big(\|\vrh(s)\|+\|\nabla\vrh(s)\|\big)+L_2^G\big(\|\cU(s)\|+\|\nabla\cU(s)\|+\|\zeta(s)\|\big)\bigg)\cdot\bigg\|\Pi_h\cU(s)\bigg\|\rd s \nonumber\\
	&\le 2|L_1^G|^2(T-t)\frac{1}{\eps_1}\int_{t}^{T}\bigg(\|\vrh(s)\|^2+\|\nabla \vrh(s)\|^2\bigg)\rd s+\int_{t}^{T}\bigg(\eps_3\|\nabla \cU(s)\|^2+\eps_4\|\zeta(s)\|^2\bigg)\rd s\nonumber\\
	&\quad\quad+\bigg(\frac{|L^G_2|^2}{\eps_3}+\frac{|L^G_2|^2}{\eps_4}+2L_2^G\bigg)\int_{t}^{T}\|\Pi_h\cU(s)\|^2\rd s+\eps_1\sup_{s\in [t,T]}\|\Pi_h\cU(s)\|^2,
	\end{align}
	%\endgroup
	Recalling \eqref{FEResult}, \eqref{FEResult2}, and the estimates in Theorem \ref{Thm1} and Lemma \ref{lem1}, we may insert above computations into \eqref{lem3ref1}  to obtain
	%stability of $L^2$-projection, that is, $\|\Pi_h\vrh(s)\|\le\|\vrh(s)\|,\|\Pi_h\nabla\vrh(s)\|\le\|\nabla\vrh(s)\|$ and\\ $\|\Pi_h\nabla\cU(s)\|\le\|\nabla\cU(s)\|$ we have
	\begin{equation*}
	\begin{split}
	&\|\cU(t)\|^2+(1-\eps_4)\sum_{i=1}^{k}\int_{t}^{T}\|\zeta^i(s)\|^2\rd s+\left(2-\eps_3-\eps_5\right)\int_{t}^{T}\|\nabla \cU(s)\|^2\rd s\\
	&\le C Q_0h^2 +|L^{g}|^2\|\vrh(T)\|^2 
	+2|L_1^G|^2(T-t)\frac{1}{\eps_1}\int_{t}^{T}\bigg(\|\vrh(s)\|^2+\|\nabla \vrh(s)\|^2\bigg)\rd s
	\\&
	\quad\quad
	+C_2\int_{t}^{T}\|\cU(s)\|^2\rd s
	+\eps_1\sup_{s\in [t,T]}\|\cU(s)\|^2-2\sum_{i=1}^{k}\int_{t}^{T}\langle\Pi_h\zeta^i(s),\Pi_h\cU(s)\rangle\rd W^i_s,
	\end{split}
	\end{equation*}
	with $C=C(L_1^{f},L_2^f,\tilde{L}^{f},L^{g},\tilde{L}^{g},L^{F}_1,L^{F}_2,L_1^{G},L_2^{G},T,d,D)$ and 
	$C_2=C_2(L_2^G)$. Taking $\eps_3=\eps_4=\frac{1}{2}$ and $\eps_5=1$, we have
	\begin{equation*}
	\begin{split}
	&\|\cU(t)\|^2+\frac{1}{2}\sum_{i=1}^{k}\int_{t}^{T}\|\zeta^i(s)\|^2\rd s+\frac{1}{2}\int_{t}^{T}\|\nabla \cU(s)\|^2\rd s\\
	&\le CQ_0h^2 +2|L_1^G|^2(T-t)\frac{1}{\eps_1}\int_{t}^{T}\bigg(\|\vrh(s)\|^2+\|\nabla \vrh(s)\|^2\bigg)\rd s+C_2\int_{t}^{T}\|\cU(s)\|^2\rd s
	\\
	&\quad\quad
	+|L^{g}|^2\|\vrh(T)\|^2
	+\eps_1\sup_{s\in [t,T]}\|\cU(s)\|^2-2\sum_{i=1}^{k}\int_{t}^{T}\langle\Pi_h\zeta^i(s),\Pi_h\cU(s)\rangle\rd W^i_s.
	\end{split}
	\end{equation*}
	Taking expectations on both sides implies particularly that	\begin{align}\label{bstep1}
	&\bE\bigg[\int_{t}^{T}\|\zeta(s)\|^2\rd s\bigg]\nonumber\\
	&\le CQ_0h^2 
	+4|L_1^G|^2(T-t)\frac{1}{\eps_1}\bE\bigg[\int_{t}^{T}\bigg(\|\vrh(s)\|^2+\|\nabla \vrh(s)\|^2\bigg)\rd s\bigg]
	\nonumber\\
	&\quad\quad
	+C_2\bE\bigg[\int_{t}^{T}\|\cU(s)\|^2\rd s\bigg]
	+2|L^{g}|^2\bE\bigg[\|\vrh(T)\|^2\bigg]+2\eps_1\bE\bigg[\sup_{s\in [t,T]}\|\cU(s)\|^2\bigg].
	\end{align}
	
	\par \textbf{Step 2}. Taking Supremum over $t\in[\tau,T]$ with $\tau\in[0,T]$ on both sides of \eqref{lem3ref1} and conducting computations as in \eqref{lem3ref2},  we have
	\begin{equation*}
	\begin{split}
	&\bE\bigg[\sup_{t\in [\tau,T]}\|\cU(t)\|^2\bigg]+(1-\eps_4)\bE\int_{\tau}^{T}\|\zeta(s)\|^2\rd s+\left(2-\eps_3-\eps_5\right)\bE\int_{\tau}^{T}\|\nabla \cU(s)\|^2\rd s\\
	&\le \frac{1}{\eps_5}\bE\int_{\tau}^{T}\big\|\nabla[id-\Pi_h]\cU(s)\big\|\rd s 
	+2|L_1^G|^2(T-\tau)\frac{1}{\eps_1}\bE\int_{\tau}^{T}\bigg(\|\vrh(s)\|^2+\|\nabla \vrh(s)\|^2\bigg)\rd s\\&
	\quad\quad+C_2\bE\int_{\tau}^{T}\|\cU(s)\|^2\rd s
	+\eps_1\bE\bigg[\sup_{s\in [\tau,T]}\|\cU(s)\|^2\bigg]+|L^{g}|^2\bE\bigg[\|\vrh(T)\|^2\bigg]\\
	&\quad\quad+2\bE\bigg[\sup_{t\in [\tau,T]}\bigg|\int_{t}^{T}\sum_{i=1}^{k}\langle\Pi_h\zeta^i(s),\Pi_h\cU(s)\rangle\rd W^i_s\bigg|\bigg],
	\end{split}
	\end{equation*}
	Take $\eps_1=1/2,\eps_3=\eps_4=3/4$  and $\eps_5=1$. It holds that
	\begin{align}\label{lem3ref3}
	&\bE\bigg[\sup_{t\in [\tau,T]}\|\cU(t)\|^2\bigg]+\frac{1}{2}\bE\int_{\tau}^{T}\|\zeta(s)\|^2\rd s+\frac{1}{2}\bE\int_{\tau}^{T}\|\nabla \cU(s)\|^2\rd s\nonumber\\
	&\le 2\bE\int_{\tau}^{T}\big\|\nabla[id-\Pi_h]\cU(s)\big\|\rd s 
	+8|L_1^G|^2(T-\tau)\bE\int_{\tau}^{T}\bigg(\|\vrh(s)\|^2+\|\nabla \vrh(s)\|^2\bigg)\rd s
	\nonumber\\
	&
	\quad\quad
	+2|L^{g}|^2\bE\bigg[\|\vrh(T)\|^2\bigg]
	+2C_2\bE\int_{\tau}^{T}\|\cU(s)\|^2\rd s
	+4\bE\bigg[\sup_{t\in [\tau,T]}\bigg|\int_{t}^{T}\sum_{i=1}^{k}\langle\Pi_h\zeta^i(s),\Pi_h\cU(s)\rangle\rd W^i_s\bigg|\bigg].
	\end{align}
	Again we use BDG inequality for the terms involving stochastic integrals to obtain
	\begin{equation*}
	\begin{split}
	&4\bE\bigg[\sup_{t\in [\tau,T]}\bigg|\int_{t}^{T}\sum_{i=1}^{k}\langle\Pi_h\zeta^i(s),\Pi_h\cU(s)\rangle\rd W^i_s\bigg|\bigg]
	\le \frac{1}{4} \bE\bigg[\sup_{t\in [\tau,T]}\| \cU(t)\|^2\bigg] +\tilde{C}^2\bE\int_{\tau}^{T}\|\zeta(s)\|^2\rd s,
	\end{split}
	\end{equation*}
	which together with \eqref{FEResult2},\eqref{bstep1}, \eqref{lem3ref3}, and estimates in Theorem \ref{Thm1} and Lemma \ref{lem1} implies that
	\begin{align*}
%	\begin{split}
	&\frac{3}{4}\bE\bigg[\sup_{t\in [\tau,T]}\|\cU(t)\|^2\bigg]+\frac{1}{2}\bE \int_{\tau}^{T}\|\zeta(s)\|^2\rd s+\frac{1}{2}\bE\int_{\tau}^{T}\|\nabla \cU(s)\|^2\rd s\\
	&\le CQ_0 h^2
	+\tilde{C}_1|L_1^G|^2(T-\tau)\bE\bigg[\int_{\tau}^{T}\bigg(\|\vrh(s)\|^2+\|\nabla \vrh(s)\|^2\bigg)\rd s\bigg]+\tilde{C}_2\bE\int_{\tau}^{T}\|\cU(s)\|^2\rd s\\
	&\quad\quad+\tilde{C}_1|L^{g}|^2\bE\bigg[\|\vrh(T)\|^2\bigg]+2\tilde{C}^2\tilde\eps_1\bE\bigg[\sup_{s\in [t,T]}\|\cU(s)\|^2\bigg].
%	\end{split}
	\end{align*}
	with constants $C=C(L^{f},\tilde{L}^{f},L^{g},\tilde{L}^{g},L^{F}_1,L^{F}_2,L_1^{G},L_2^{G},T,d,D)$, $\tilde{C}_1=\frac{4\tilde{C}^2}{\eps_1}+8$ and \linebreak $\tilde{C}_2=\tilde{C}^2C_2+C_2$. Taking $\eps_1=\frac{1}{8\tilde{C}^2}$, we have
	\begin{equation*}
	\begin{split}
	&\bE\bigg[\sup_{t\in [\tau,T]}\|\cU(t)\|^2\bigg]+\bE\sum_{i=1}^{k}\int_{\tau}^{T}\|\zeta^i(s)\|^2\rd s+\bE\int_{\tau}^{T}\|\nabla \cU(s)\|^2\rd s\\
	&\le CQ_0h^2
	+2\tilde{C}_1|L_1^G|^2(T-\tau)\bE\bigg[\int_{\tau}^{T}\bigg(\|\vrh(s)\|^2+\|\nabla \vrh(s)\|^2\bigg)\rd s\bigg]+2\tilde{C}_2\bE\int_{\tau}^{T}\|\cU(s)\|^2\rd s\\
	&\quad\quad+2\tilde{C}_1|L^{g}|^2\bE\bigg[\|\vrh(T)\|^2\bigg],
	\end{split}
	\end{equation*}
	which by Gronwall's inequality yields that
	%\begin{equation*}
	%\begin{split}
	%\bE\bigg[\sup_{t\in [\tau,T]}\|\Pi_h\cU(t)\|^2\bigg]&+\bE\sum_{i=1}^{k}\int_{\tau}^{T}\|\Pi_h\zeta^i(s)\|^2\rd s+\bE\int_{\tau}^{T}\|\nabla \cU(s)\|^2\rd s\le \tilde{C}\bE\int_{\tau}^{T}\big\|\nabla \big[id-\Pi_h\big]\cU(s)\big\|^2\rd s\\
	%&
	%+\tilde{C}|L_1^G|^2(T-\tau)\bE\int_{\tau}^{T}\bigg(\|\vrh(s)\|^2+\|\nabla \vrh(s)\|^2\bigg)\rd s+C_2\bE\int_{\tau}^{T}\|\Pi_h\cU(s)\|^2\rd s\\
	%&+\tilde{C}|L^{g}|^2\bE\bigg[\|\vrh(T)\|^2\bigg].
	%\end{split}
	%\end{equation*}
	%Now using error estimate of $L^2$-projection $\|\nabla\left[id-\Pi_h\right]\cU(s)\|\le\|\nabla\left[id-\Pi_h\right]u(s)\|\le C_eh\|\Delta u(s)\|$, we have
	%\begin{equation*}
	%\begin{split}
	%\bE\bigg[\sup_{t\in [\tau,T]}\|\Pi_h\cU(t)\|^2\bigg]&+\bE\sum_{i=1}^{k}\int_{\tau}^{T}\|\Pi_h\zeta^i(s)\|^2\rd s+\bE\int_{\tau}^{T}\|\nabla \cU(s)\|^2\rd s\le Ch^2+\tilde{C}|L^{g}|^2\bE\bigg[\|\vrh(T)\|^2\bigg]\\
	%&
	%+\tilde{C}|L_1^G|^2(T-\tau)\bE\int_{\tau}^{T}\bigg(\|\vrh(s)\|^2+\|\nabla \vrh(s)\|^2\bigg)\rd s+C_2\bE\int_{\tau}^{T}\|\Pi_h\cU(s)\|^2\rd s,\\
	%\end{split}
	%\end{equation*}
	%with $C=C(k,L^{f},\tilde{L}^{f},L^{g},\tilde{L}^{g},L^{F}_1,L^{F}_2,L_1^{G},L_2^{G},T,C_e)$. Using Gronwall's inequality we have
	\begin{equation*}
	\begin{split}
	&\bE\bigg[\sup_{t\in [0,T]}\|\cU(t)\|^2\bigg]+\bE \int_{0}^{T}\|\zeta(s)\|^2\rd s+\bE\int_{0}^{T}\|\nabla \cU(s)\|^2\rd s\\
	&\le CQ_0 h^2+\bar{C}_2|L^{g}|^2\bE\bigg[\|\vrh(T)\|^2\bigg]+\bar{C}_2|L_1^G|^2T\bE\bigg[\int_{0}^{T}\bigg(\|\vrh(s)\|^2+\|\nabla \vrh(s)\|^2\bigg)\rd s\bigg],
	\end{split}
	\end{equation*}
	where the constants $C=C(L^{f},\tilde{L}^{f},L^{g},\tilde{L}^{g},L^{F}_1,L^{F}_2,L_1^{G},L_2^{G},T,d,D)$ and $\bar{C}_2=2\tilde{C}_1\cdot e^{2\tilde{C}_2T}$ do not depend on $h$. Then we may write
	\begin{equation*}
	\begin{split}
	&\bE\bigg[\sup_{t\in [0,T]}\|\cU(t)\|^2\bigg]+\bE \int_{0}^{T}\|\zeta(s)\|^2\rd s+\bE\int_{0}^{T}\|\nabla \cU(s)\|^2\rd s\\
	&\le CQ_0h^2+\bar{C}_2|L^{g}|^2\bE\bigg[\sup_{t\in [0,T]}\|\vrh(t)\|^2\bigg]+\bar{C}_2|L_1^G|^2T  \bE\bigg[T\cdot\sup_{t\in [0,T]}\|\vrh(t)\|^2+\int_{0}^{T}\|\nabla \vrh(s)\|^2\bigg]\rd s\\
	&\le CQ_0h^2+\bar{C}_2\cdot\max{\{|L_1^G|^2T^2+|L^g|^2,|L^g|^2T\}}  \bE\bigg[\sup_{t\in [0,T]}\|\vrh(t)\|^2+\int_{0}^{T}\|\nabla \vrh(s)\|^2\bigg]\rd s.
	\end{split}
	\end{equation*}
	%\begin{equation*}
	%\begin{split}
	%\bE\bigg[\sup_{t\in [0,T]}\|\Pi_h\cU(t)\|^2\bigg]&+\bE\sum_{i=1}^{k}\int_{0}^{T}\|\Pi_h\zeta^i(s)\|^2\rd s+\bE\int_{0}^{T}\|\nabla \cU(s)\|^2\rd s\le Ch^2\\
	%&+\tilde{C_2}.\max{\{|L_1^G|^2T^2+|L^g|^2,|L^g|^2T\}}  \bE\bigg[\sup_{t\in [0,T]}\|\vrh(t)\|^2+\int_{0}^{T}\|\nabla \vrh(s)\|^2\bigg]\rd s.\\
	%\end{split}
	%\end{equation*}
	Using estimate for $\vrh(t)$ from Lemma \ref{lem2}, we have 	
	%\begingroup\makeatletter\def\f@size{10}\check@mathfonts
	\begin{equation*}
	\begin{split}
	&\bE\bigg[\sup_{t\in [0,T]}\|\cU(t)\|^2\bigg]+\bE \int_{0}^{T}\|\zeta(s)\|^2\rd s+\bE\int_{0}^{T}\|\nabla \cU(s)\|^2\rd s\\
	%&\le Ch^2+\tilde{C_2}.\max{\{|L_1^G|^2T^2+|L^g|^2,|L^g|^2T\}}\bigg(Ch^2+ \tilde{C_1}|L_2^F|^2T \bE\bigg[\int_{0}^{T}\big(\|\cU(s)\|^2+\|\nabla \cU(s)\|^2+\|\zeta(s)\|^2\big)\rd s\bigg]\bigg)\\
	%&\le Ch^2+\tilde{C_2}.\max{\{|L_1^G|^2T^2+|L^g|^2,|L^g|^2T\}}.Ch^2\\
	%&\quad+\tilde{C_1}\tilde{C_2}.\max{\{|L_1^G|^2T^2+|L^g|^2,|L^g|^2T\}}\cdot T\cdot\max\{|L_2^F|^2,|L_2^F|^2T+|L_2^f|^2\} \bE\bigg[\sup_{t\in [0,T]}\|\cU(t)\|^2+\int_{0}^{T}\bigg(\|\nabla \cU(s)\|^2+\|\zeta(s)\|^2\bigg)\rd s\bigg]\\
	&\le CQ_0h^2+\hat{C} \bE\bigg[\sup_{t\in [0,T]}\|\cU(t)\|^2+\int_{0}^{T}\bigg(\|\nabla \cU(s)\|^2+\|\zeta(s)\|^2\bigg)\rd s\bigg],
	\end{split}
	\end{equation*}
	%\endgroup
	%\begin{equation*}
	%\begin{split}
	%\bE\bigg[&\sup_{t\in [0,T]}\|\Pi_h\cU(t)\|^2\bigg]+\bE\sum_{i=1}^{k}\int_{0}^{T}\|\Pi_h\zeta^i(s)\|^2\rd s+\bE\int_{0}^{T}\|\nabla \cU(s)\|^2\rd s\le Ch^2\\
	%&+\tilde{C_2}.\max{\{|L_1^G|^2T^2+|L^g|^2,|L^g|^2T\}}.Ch^2\\
	%&+\tilde{C_1}\tilde{C_2}.\max{\{|L_1^G|^2T^2+|L^g|^2,|L^g|^2T\}}.|L_2^F|^2T \bE\bigg[T.\sup_{t\in [0,T]}\|\cU(t)\|^2+\int_{0}^{T}\bigg(\|\nabla \cU(s)\|^2+\|\zeta(s)\|^2\bigg)\rd s\bigg],\\
	%\end{split}
	%\end{equation*}
	%\begin{equation*}
	%\begin{split}
	%\bE\bigg[&\sup_{t\in [0,T]}\|\Pi_h\cU(t)\|^2\bigg]+\bE\sum_{i=1}^{k}\int_{0}^{T}\|\Pi_h\zeta^i(s)\|^2\rd s+\bE\int_{0}^{T}\|\nabla \cU(s)\|^2\rd s\le Ch^2\\
	%&+\hat{C} \bE\bigg[\sup_{t\in [0,T]}\|\cU(t)\|^2+\int_{0}^{T}\bigg(\|\nabla \cU(s)\|^2+\|\zeta(s)\|^2\bigg)\rd s\bigg],\\
	%\end{split}
	%\end{equation*}
	where $\hat{C}=\bar{C}e^{\bar{C}T}\cdot\max{\{|L_1^G|^2T^2+|L^g|^2,|L^g|^2T\}}\cdot T\cdot\max\{|L_2^F|^2,|L_2^F|^2T+|L_2^f|^2\}$ and\\ $C=C(L_1^{f},L_2^f,\tilde{L}^{f},L^{g},\tilde{L}^{g},L^{F}_1,L^{F}_2,L_1^{G},L_2^{G},T,d,D)$.

	Since we already have from Theorem \ref{Thm1}, that
	\begin{equation*}
	\hat{C}=\bar{C}e^{\bar{C}T}\cdot\max{\{|L_1^G|^2T^2+|L^g|^2,|L^g|^2T\}}\cdot T\cdot\max\{|L_2^F|^2,|L_2^F|^2T+|L_2^f|^2\}<1,
	\end{equation*} 
	we have 
	\begin{equation*}
	\begin{split}
	\bE\bigg[&\sup_{t\in [0,T]}\|\cU(t)\|^2\bigg]+\bE\sum_{i=1}^{k}\int_{0}^{T}\|\zeta^i(s)\|^2\rd s+\bE\int_{0}^{T}\|\nabla \cU(s)\|^2\rd s\le CQ_0h^2,\\
	\end{split}
	\end{equation*}
	where the constant $C=C(L_1^{f},L_2^f,\tilde{L}^{f},L^{g},\tilde{L}^{g},L^{F}_1,L^{F}_2,L_1^{G},L_2^{G},T,d,D)$ is independent of $h$. Combining this estimate with estimate of $\vrh(t)$ from Lemma \ref{lem2} we have finally,
	\begin{equation*}
	\begin{split}
	\bE\bigg[\sup_{t\in [0,T]}\|\vrh(t)\|^2&+\sup_{t\in [0,T]}\|\cU(t)\|^2+\int_{0}^{T}\bigg(\|\nabla\vrh(t)\|^2+\|\nabla\cU(t)\|^2+\|\zeta(t)\|^2\bigg)\rd t\bigg]\le CQ_0h^2.
	\end{split}
	\end{equation*}
\end{proof}

%%%%%%%%%%%%%%%%%%%%%%%%%%%%%%%%%%%%%%%%%%%%%%%%%%%%%%%%%%%%%%%%%%%%%%%%%%%%%%%%%%%%%

\section{Finite dimensional approximating FBSDEs and deep learning-based algorithms}
\subsection{Finite dimensional approximating FBSDEs}
As the approximations of the solution to FBSPDE \eqref{FSPDE}-\eqref{BSPDE}, the solution to FBSDE \eqref{dFSPDE}-\eqref{dBSPDE} is valued on the finite dimensional space $V^0_h$ and has the following form:
\begin{equation*}
\begin{split}
\rho_h(t,x)=\sum_{l=1}^{L} [\vec{\rho}_h(t)]_l \phi_h^l(x),\quad
u_h(t,x)=\sum_{l=1}^{L} [\vec{u}_h(t)]_l \phi_h^l(x),\quad \text{and}\quad
\psi_h^i(t,x)=\sum_{l=1}^{L} [\vec{\psi}_h^i(t)]_l \phi_h^l(x).
\end{split}
\end{equation*} 
   Then for each $\phi^l_h\in V^0_h, l=1,2,\dots,L$ we have
%\begin{equation*} 
%\begin{split}
%\langle \mathrm{d}\rho_h(t),\phi_h^l\rangle=&\bigg(-\langle\nabla \rho_h(t),\nabla\phi^l_h\rangle+\big\langle F(t,x,\rho_h(t),\nabla\rho_h(t),u_h(t),\nabla u_h(t),\psi_h(t)),\phi^l_h\big\rangle\bigg)\,\mathrm{d}t\\
%&-\sum_{i=1}^{k}\big\langle f^i(t, \rho_h(t),u_h(t)),\phi^l_h\big\rangle \mathrm{d}W^i(t),
%\end{split}
%\end{equation*}
%and thus,
\begin{align*} 
%\begin{split}
\langle \rho_h(t),\phi^l_h\rangle&=\langle\rho_0,\phi^l_h\rangle-\int_{0}^{t}\sum_{i=1}^{k}\big\langle f^i(s, \rho_h(s),u_h(s)),\phi^l_h\big\rangle \mathrm{d}W^i(s)\\
&+\int_{0}^{t}\bigg(-\langle\nabla \rho_h(s),\nabla\phi^l_h\rangle+\big\langle F(s,x,\rho_h(s),\nabla\rho_h(s),u_h(s),\nabla u_h(s),\psi_h(s)),\phi^l_h\big\rangle\bigg)\,\mathrm{d}s.
%\end{split}
\end{align*}
It may be written as
\begin{align*}
%\begin{split}
\mathbf{A}\vec{\rho}_h(t)=&\vec{\rho_0}_{\phi}+\int_{0}^{t}\bigg(-\mathbf{B}\vec{\rho}_h(s)+\vec{F}\big(s,\rho_h(s),\nabla\rho_h(s),u_h(s),\nabla u_h(s),\psi_h(s);\phi_h\big)\bigg)\,\mathrm{d}t\\
&-\int_{0}^{t}\sum_{i=1}^{k} \vec{f^i}(s, \rho_h(s),u_h(s);\phi_h) \mathrm{d}W^i(s),
%\end{split}
\end{align*}
or equivalently,
\begin{align*}
%\begin{split}
\vec{\rho}_h(t)=&\mathbf{A}^{-1}\vec{\rho_0}_{\phi}+\mathbf{A}^{-1}\int_{0}^{t}\bigg(-\mathbf{B}\vec{\rho}_h(s)+\vec{F}\big(s,\rho_h(s),\nabla\rho_h(s),u_h(s),\nabla u_h(s),\psi_h(s);\phi_h\big)\bigg)\,\mathrm{d}t\\
&-\mathbf{A}^{-1}\int_{0}^{t}\sum_{i=1}^{k} \vec{f^i}(s, \rho_h(s),u_h(s);\phi_h) \mathrm{d}W^i(s),
%\end{split}
\end{align*}
where $\mathbf{A}=\big(a^{ml}\big)_{1\le m,l\le L}$ and $\mathbf{B}=\big(b^{ml}\big)_{1\le m,l\le L}$ with $a^{ml}=\langle \phi_h^l,\phi_h^m\rangle$ and $b^{ml}=\langle \nabla\phi_h^l,\nabla\phi_h^m\rangle$.
The function $\vec{F}(t,\rho_h(t),\nabla\rho_h(t),u_h(t),\nabla u_h(t),\psi_h(t);\phi_h)$ is $\bR^L$-valued with $l$-th entry \linebreak $\big\langle F\big(t,x,\rho_h(t),\nabla\rho_h(t),u_h(t),\nabla u_h(t),\psi_h(t)\big),\phi_h^l\big\rangle$. Here, by $\vec{\mathfrak X}_\phi$ we denote a vector with $l$-th entry $[\vec{\mathfrak X}_\phi]_l:=\la x,\phi_h^l\ra$.

On the other hand, for each $\phi^l_h\in V^0_h, l=1,2,\dots,L$, we have
%\begin{equation*} 
%\begin{split}
%-\langle \mathrm{d}u_h(t),\phi^l_h\rangle=&\bigg(-\langle\nabla u_h(t),\nabla\phi^l_h\rangle+\big\langle G(t,x,\rho_h(t),\nabla\rho_h(t),u_h(t),\nabla u_h(t),\psi_h(t)),\phi^l_h\big\rangle\bigg)\,\mathrm{d}t\\
%&-\sum_{i=1}^{k}\big\langle \psi_h^i(t),\phi^l_h\big\rangle \mathrm{d}W^i(t),
%\end{split}
%\end{equation*}
%and thus,
\begin{equation*} 
\begin{split}
\langle u_h(t),\phi^l_h\rangle&=\langle g(\rho_h(T)),\phi^l_h\rangle-\int_{t}^{T}\sum_{i=1}^{k}\big\langle \psi_h^i(s),\phi^l_h\big\rangle \mathrm{d}W^i(s)\\
&-\int_{t}^{T}\bigg(\langle\nabla \rho_h(s),\nabla\phi^l_h\rangle-\big\langle G(s,x,\rho_h(s),\nabla\rho_h(s),u_h(s),\nabla u_h(s),\psi_h(s)),\phi^l_h\big\rangle\bigg)\,\mathrm{d}s,
\end{split}
\end{equation*}
which may be written equivalently as
%\begin{equation*}
%\begin{split}
%\mathbf{A}\vec{u}_h(t)=&\vec{g}(\rho_h(T);\phi_h)-\int_{t}^{T}\sum_{i=1}^{k} \mathbf{A}\vec{\psi}_h^i(s) \mathrm{d}W^i(s)\\
%&-\int_{t}^{T}\bigg(\mathbf{B}\vec{u}_h(s)-\vec{G}\big(s,\rho_h(s),\nabla\rho_h(s),u_h(s),\nabla u_h(s),\psi_h(s);\phi_h\big)\bigg)\,\mathrm{d}s,
%\end{split}
%\end{equation*}
%or equivalently,
\begin{align*}
%\begin{split}
\vec{u}_h(t)=&\mathbf{A}^{-1}\vec{g}(\rho_h(T);\phi_h)-\int_{t}^{T}\sum_{i=1}^{k} \vec{\psi}_h^i(s) \mathrm{d}W^i(s)\\
&-\mathbf{A}^{-1}\int_{t}^{T}\bigg(\mathbf{B}\vec{u}_h(s)-\vec{G}\big(s,\rho_h(s),\nabla\rho_h(s),u_h(s),\nabla u_h(s),\psi_h(s);\phi_h\big)\bigg)\,\mathrm{d}s.
%\end{split}
\end{align*}
%Here we denote by $[\vec{x}]_l$ the $l-$th coordinate of the vector $\vec{x}$. $\vec{F}(t,\rho_h(t),\nabla\rho_h(t),u_h(t),\nabla u_h(t),\psi_h(t);\phi_h)$ is a vector with $l-$th entry $\big\langle F\big(t,x,\rho_h(t),\nabla\rho_h(t),u_h(t),\nabla u_h(t),\psi_h(t)\big),\phi_h^l\big\rangle$, $\mathbf{A}$ is a matrix of entries $\langle \phi_h^l,\phi_h^m\rangle$,\:$\mathbf{B}$ is a matrix of entries $\langle \nabla\phi_h^l,\nabla\phi_h^m\rangle$ and \:$\mathbf{M}$ is a matrix of entries $\langle \nabla\phi_h^l,\phi_h^m\rangle $. Also by $\vec{x}_\phi$ we denote a vector with entries $[\vec{x}_\phi]_l:=\la x,\phi_h^l\ra$.
\par Finally, we have the following finite dimensional coupled FBSDE
{\small
\begin{equation}\label{fdFBSDE1}
\left\{\begin{array}{l}
\begin{split}
&\begin{split}
\vec{\rho}_h(t)=&\mathbf{A}^{-1}\vec{\rho_0}_{\phi}+\mathbf{A}^{-1}\int_{0}^{t}\bigg(-\mathbf{B}\vec{\rho}_h(s)+\vec{F}\big(s,\rho_h(s),\nabla\rho_h(s),u_h(s),\nabla u_h(s),\psi_h(s);\phi_h\big)\bigg)\,\mathrm{d}s\\
&-\mathbf{A}^{-1}\int_{0}^{t}\sum_{i=1}^{k} \vec{f^i}(s, \rho_h(s),u_h(s);\phi_h) \mathrm{d}W^i(s),
\end{split}\\
&\begin{split}
\vec{u}_h(t)=&\mathbf{A}^{-1}\vec{g}(\rho_h(T);\phi_h)-\int_{t}^{T}\sum_{i=1}^{k} \vec{\psi}_h^i(s) \mathrm{d}W^i(s)\\
&-\mathbf{A}^{-1}\int_{t}^{T}\bigg(\mathbf{B}\vec{u}_h(s)-\vec{G}\big(s,\rho_h(s),\nabla\rho_h(s),u_h(s),\nabla u_h(s),\psi_h(s);\phi_h\big)\bigg)\,\mathrm{d}s,
\end{split}
\end{split}
\end{array}\right.
\end{equation}
}
In the above coupled FBSDE, the $\bR^L$-valued random functions $\vec{\rho}_h(t), \vec{u}_h(t)$, and $\vec{\psi}_h^i(t)$, $i=1,\dots,k$, are unknown expansion coefficients in the expressions for $\rho_h$, $u_h$, and $\psi_h^i$, respectively. 
We shall use existing deep learning-based algorithms to solve the finite dimensional coupled FBSDE of the form \eqref{fdFBSDE1}.

\subsection{Deep learning algorithms for FBSDEs}\label{secDeepLearning}
Let us consider the following general form of coupled FBSDE:
\begin{equation}\label{fdFBSDE}
\left\{\begin{array}{l}
\begin{split}
X(t)=&X_0+\int_{0}^{t}\mu(s,X(s),Y(s),Z(s))\mathrm{d}s-\sum_{i=1}^k\int_{0}^{t} \sigma^i(s,X(s),Y(s))\mathrm{d}W^i(s),\\
Y(t)=&g(X(T))+\int_{t}^{T}b(s,X(s),Y(s),Z(s))\mathrm{d}s-\sum_{i=1}^k\int_{t}^{T} Z^i(s)\mathrm{d}W^i(s),
\end{split}
\end{array}\right.
\end{equation}
with the unknown processes $X(t)$, $Y(t)$, and $Z^i(t)$ being $\bR^L$-valued, for $i=1,2,\cdots,k$. When $\mu$ and $\sigma$ do not depend on $Y$ or $Z$, FBSDE \eqref{fdFBSDE} is decoupled.
Some algorithms based on deep learning techniques that  are highly capable of solving such decoupled FBSDEs \eqref{fdFBSDE} when $L$ is large have just been proposed; see \cite{weinan2017deep,han2018deep,hure2019deep}  for instance. For coupled FBSDEs, an algorithm is proposed in \cite{Han&Long2019coupleddeep} with a convergence analysis, and three more algorithms for fully coupled FBSDEs are also introduced in \cite{JiPeng2020CoupledBSDE}. Under the Markovian framework (i.e., when all the coefficients $\mu$, $\sigma$, and $g$ are deterministic), the solution $(Y(t),Z(t))$ of the BSDE in \eqref{fdFBSDE}, can be expressed as a function of $X(t)$, the solution of forward SDE in \eqref{fdFBSDE}, that is, $\big(Y(t),Z(t)\big)=\big(\cY(t,X(t)),\cZ(t,X(t))\big),t\in[0,T]$ for some deterministic functions $\cY$ and $\cZ$.
This well known result is a key-ingredient for the approximations of BSDEs in all these deep learning-based algorithms for FBSDEs. We will address these deep learning-based algorithms for finite dimensional FBSDEs as Deep BSDE. Deep BSDE methods use neural networks to approximate unknown functions and reformulates the original problem into a stochastic optimization problem. Here we shall present two existing Deep BSDE algorithms: one is for decoupled FBSDE and another for coupled FBSDE. We will introduce a third algorithm as a modified version of  the second one for coupled FBSDEs.
% and we will develop another Deep BSDE algorithm for  coupled FBSDE inspired from Generic Hybrid Algo. for stochastic optimal control problem proposed in \cite{hure2019control}.

To discuss the numerical algorithms, we first consider a partition of the time interval $[0,T]$ into grid $\pi:{t_0=0<t_1<\dots<t_J=T}$ with $|\pi|=\max_{j=0,\cdots,J-1}\Delta t_j$, $\Delta t_j:=t_{j+1}-t_j$. Now let us consider the forward representation of the BSDE in \eqref{fdFBSDE}, which is written as
\begin{equation}\label{fwrepbw}
Y(t)=Y(0)-\int_{0}^{t}b(s,X(s),Y(s),Z(s))\mathrm{d}s+\sum_{i=1}^k\int_{0}^{t} Z^i(s)\mathrm{d}W^i(s),\quad 0\le t\le T.
\end{equation}
Without any loss of generality, we take $k=1$ in the following subsections.
\subsubsection{Deep BSDE-1 for decoupled Markovian FBSDEs}
This Deep BSDE algorithm proposed in \cite{hure2019deep}  is for decoupled FBSDEs. The forward process $X$ of \eqref{fdFBSDE} is numerically approximated by  $X^{\pi}$ using Euler Scheme on time grid $\pi$. For example, forward Euler scheme can be used which is defined as
\begin{equation}\label{EulerFSDE}
X^{\pi}_{t_{j+1}}=X^{\pi}_{t_j}+\mu(t_j,X^{\pi}_{t_j})\Delta t_j-\sigma(t_j,X^{\pi}_{t_j})\Delta W_{t_j},\: j=0,\cdots,J-1,\:X^{\pi}_{t_0}=X_0,
\end{equation}
with $\Delta W_{t_{j}}:=W(t_{j+1})-W(t_j)$.\\
Here $Y(t)$ and $Z(t)$ are treated as functions of $X(t)$, that is, $Y(t)=\cY(t,X(t))$ and $Z(t)=\cZ(t,X(t))$ for some deterministic functions $\cY$ and $\cZ$. This Deep FBSDE algorithm is based on backward dynamic programming, and the discrete approximations of the functions $\cY(t,\cdot)$ and $\cZ(t,\cdot)$ on time grid $\pi$ are performed backwardly in time. These functions are approximated by deep neural networks.

The algorithm starts with an estimation $Y^{\pi}_{t_J}$ of $\cY(t_J,X^{\pi}_{t_J})$ with $Y^{\pi}_{t_J}=g(X^{\pi}_{t_J})$. Then at each time step $t_j:j=J-1,\cdots,1,0$, given an estimation $Y^{\pi}_{t_{j+1}}$ of $\cY(t_{j+1},X^{\pi}_{t_{j+1}})$, two independent deep neural networks $\mathcal{Y}^{\mathcal{N}}_j(\cdot;\theta_{1,j}) $ and $\mathcal{Z}^{\mathcal{N}}_j(\cdot;\theta_{2,j})$  approximate, respectively, $\cY(t_j,\cdot)$ and $\cZ(t_j,\cdot)$ by minimizing quadratic loss function
\begin{equation}
\widehat{\mathfrak{L}}_j(\theta_j):=\bE\bigg|Y^{\pi}_{t_{j+1}}-Y^F\big(t_{j+1}\big|X^{\pi}_{t_j},\mathcal{Y}^{\mathcal{N}}_j(X^{\pi}_{t_j};\theta_{1,j}),\mathcal{Z}^{\mathcal{N}}_j(X^{\pi}_{t_j};\theta_{2,j})\big)\bigg|^2,
\end{equation}
with respect to its parameters $\theta_j=(\theta_{1,j},\theta_{2,j})$ using gradient based method, where 
\begin{equation}
\begin{split}
&Y^F\big(t_{j+1}\big|X^{\pi}_{t_j},\mathcal{Y}^{\mathcal{N}}_j(X^{\pi}_{t_j};\theta_{1,j}),\mathcal{Z}^{\mathcal{N}}_j(X^{\pi}_{t_j};\theta_{2,j})\big)\\
&= \mathcal{Y}^{\mathcal{N}}_j(X^{\pi}_{t_j};\theta_{1,j})-b\big(t_j,X^{\pi}_{t_j},\mathcal{Y}^{\mathcal{N}}_j(X^{\pi}_{t_j};\theta_{1,j}),\mathcal{Z}^{\mathcal{N}}_j
(X^{\pi}_{t_{j}};\theta_{2,j})\big)\Delta t_j+ \mathcal{Z}^{\mathcal{N}}_j(X^{\pi}_{t_{j}};\theta_{2,j})\Delta W_{t_j},
\end{split}
\end{equation}
is computed from the forward representation \eqref{fwrepbw} of the backward equation. If
%$\theta^*_j=(\theta^*_{1,j},\theta^*_{2,j})$,
\begin{equation*}
\theta_{j}^*=(\theta^*_{1,j},\theta^*_{2,j}) \in \arg \min_{\theta_j\in\Theta} \widehat{\mathfrak{L}}_j(\theta_j),
\end{equation*}
then $\mathcal{Y}^{\mathcal{N}}_j(X^{\pi}_{t_j};\theta^*_{1,j})$ is the approximation of  $Y^{\pi}_{t_j}=\cY(t_j,X^{\pi}_{t_j})$ and $\mathcal{Z}^{\mathcal{N}}_j(X^{\pi}_{t_j};\theta^*_{2,j})$ the approximation of $Z^{\pi}_{t_j}=\cZ(t_j,X^{\pi}_{t_j})$. Finally, the backward induction on time step leads us to $\mathcal{Y}^{\mathcal{N}}_j(X^{\pi}_{t_0};\theta^*_{1,0})$, the approximation of  $Y^{\pi}_{t_0}=\cY(t_0,X^{\pi}_{t_0})$. We refer to \cite{hure2019deep}  for the convergence analysis and various numerical examples.
% **** convergence result and error estimates

\subsubsection{Deep BSDE-2 for coupled Markovian FBSDEs}

This algorithm for fully coupled FBSDE is proposed in \cite[Algorithm-2]{JiPeng2020CoupledBSDE}. Here $Y(t)$ and $Z(t)$ are also treated as functions of $X(t)$, that is, $Y(t)=\cY(t,X(t))$ and $Z(t)=\cZ(t,X(t))$ for some deterministic functions $\cY$ and $\cZ$. This method starts with estimations $\cY_0$ and $\cZ_0$ of $\cY(t_0,X_{t_0}^{\pi})=Y^{\pi}_{t_0}$ and 
$\cZ(t_0,X_{t_0}^{\pi})=Z^{\pi}_{t_0}$ respectively and then calculate $X_{t_1}^{\pi}$ and $Y_{t_1}^{\pi}$ by using Euler scheme as
\begin{equation*}
\begin{split}
X_{t_1}^{\pi}&=X_{t_0}^{\pi}+\mu(t_0,X_{t_0}^{\pi},\cY_0,\cZ_0 )\Delta t_0-\sigma(t_0,X_{t_0}^{\pi},Y_{t_0}^{\pi})\Delta W_{t_0},\\
Y_{t_1}^{\pi}&=\cY_0-b(t_0,X_{t_0}^{\pi},\cY_0,\cZ_0 )\Delta t_0+\cZ_0\Delta W_{t_0}.
\end{split}
\end{equation*}
When using two neural networks $\cY_1^{\mathcal{N}}(\cdot;\theta_{1,1})$ and $\cZ_1^{\mathcal{N}}(\cdot;\theta_{2,1})$ to approximate respectively $Y_{t_1}^{\pi}=\cY(t_1,X_{t_1}^{\pi})$ and $Z_{t_1}^{\pi}=\cZ(t_1,X_{t_1}^{\pi})$, the associated local loss function is defined as
\begin{equation*}
\mathfrak{L}_1:=\Delta t_0\cdot\bE\bigg|Y_{t_1}^{\pi}-\cY_1^{\mathcal{N}}(	X_{t_1}^{\pi};\theta_{1,1})\bigg|^2.
\end{equation*}
Then for $j=1,2,\cdots,J-2$, using Euler scheme for $X_{t_{j+1}}^{\pi}$ and $Y_{t_{j+1}}^{\pi}$ gives
\begin{equation*}
\begin{split}
X_{t_{j+1}}^{\pi}&=X_{t_j}^{\pi}+\mu\big(t_j,X_{t_j}^{\pi},\cY_j^{\mathcal{N}}(	X_{t_j}^{\pi};\theta_{1,j}),\cZ_j^{\mathcal{N}}(	X_{t_j}^{\pi};\theta_{2,j}) \big)\Delta t_j-\sigma(t_j,X_{t_j}^{\pi},Y_{t_j}^{\pi})\Delta W_{t_j},\\
Y_{t_{j+1}}^{\pi}&=Y_{t_{j}}^{\pi}-b\big(t_j,X_{t_j}^{\pi},Y_{t_{j}}^{\pi},\cZ_j^{\mathcal{N}}(	X_{t_j}^{\pi};\theta_{2,j}) \big)\Delta t_j+\cZ_j^{\mathcal{N}}(	X_{t_j}^{\pi};\theta_{2,j})\Delta W_{t_j},
\end{split}
\end{equation*}
where the two neural networks $\cY_{j+1}^{\mathcal{N}}(\cdot;\theta_{1,j+1})$ and $\cZ_{j+1}^{\mathcal{N}}(\cdot;\theta_{2,j+1})$ approximate respectively $Y_{t_{j+1}}^{\pi}=\cY(t_{j+1},X_{t_{j+1}}^{\pi})$ and $Z_{t_{j+1}}^{\pi}=\cZ(t_{j+1},X_{t_{j+1}}^{\pi})$ with associated local loss function given by
\begin{equation*}
\mathfrak{L}_{j+1}:=\Delta t_j \cdot\bE\bigg|Y_{t_{j+1}}^{\pi}-\cY_{j+1}^{\mathcal{N}}(	X_{t_{j+1}}^{\pi};\theta_{1,j+1})\bigg|^2.
\end{equation*}
Finally, using Euler scheme for $X_{t_{J}}^{\pi}$ and $Y_{t_{J}}^{\pi}$ gives
\begin{equation*}
\begin{split}
X_{t_{J}}^{\pi}&=X_{t_{J-1}}^{\pi}+\mu\big(t_{J-1},X_{t_{J-1}}^{\pi},\cY_{J-1}^{\mathcal{N}}(	X_{t_{J-1}}^{\pi};\theta_{1,{J-1}}),\cZ_{J-1}^{\mathcal{N}}(	X_{t_{J-1}}^{\pi};\theta_{2,{J-1}}) \big)\Delta t_{J-1}\\
&-\sigma(t_{J-1},X_{t_{J-1}}^{\pi},Y_{t_{J-1}}^{\pi})\Delta W_{t_{J-1}},\\
Y_{t_{J}}^{\pi}&=Y_{t_{J-1}}^{\pi}-b\big(t_{J-1},X_{t_{J-1}}^{\pi},Y_{t_{J-1}}^{\pi},\cZ_{J-1}^{\mathcal{N}}(	X_{t_{J-1}}^{\pi};\theta_{2,{J-1}}) \big)\Delta t_{J-1}+\cZ_{J-1}^{\mathcal{N}}(	X_{t_{J-1}}^{\pi};\theta_{2,{J-1}})\Delta W_{t_{J-1}},
\end{split}
\end{equation*}
and define local loss function
\begin{equation*}
\mathfrak{L}_{J}:=\bE\bigg|Y_{t_{J}}^{\pi}-g(X_{t_{J}}^{\pi})\bigg|^2.
\end{equation*}
Now the scheme is to optimize the global loss function
\begin{equation*}
\widehat{\mathfrak{L}}(\cY_0,\cZ_0,\theta_{1,1},\theta_{2,1},\cdots,\theta_{1,J},\theta_{2,J}):=\sum_{j=1}^{J}\mathfrak{L}_{j},
\end{equation*}
over all $\theta=(\cY_0,\cZ_0,\theta_{1,1},\theta_{2,1},\cdots,\theta_{1,J},\theta_{2,J})$, and for some   $\theta^*=(\cY^*_0,\cZ^*_0,\theta^*_{1,1},\theta^*_{2,1},\cdots,\theta^*_{1,J},\theta^*_{2,J})$ if
\begin{equation*}
\theta^*\in \arg \min_{\theta\in\Theta} \widehat{\mathfrak{L}}(\theta),
\end{equation*}
then $\cY_0^*$ is the desired approximation of $Y^{\pi}_{t_0}$ by Deep BSDE-2. We refer to \cite{JiPeng2020CoupledBSDE} for various numerical examples, while the reader may refer to \cite{Han&Long2019coupleddeep} for an alternative algorithm for a class of coupled Markovian FBSDEs with both convergence analysis and numerical examples.

\subsubsection{Deep BSDE-3 for coupled Markovian FBSDEs}
This method is just a modified version of Deep BSDE-2. This method starts with estimations $\cY_0$ and $\cZ_0$ of $\cY(t_0,X_{t_0}^{\pi})=Y^{\pi}_{t_0}$ and 
$\cZ(t_0,X_{t_0}^{\pi})=Z^{\pi}_{t_0}$ respectively and then calculate $X_{t_1}^{\pi}$ and $Y_{t_1}^{\pi}$ by using Euler scheme as
\begin{equation*}
\begin{split}
X_{t_1}^{\pi}&=X_{t_0}^{\pi}+\mu(t_0,X_{t_0}^{\pi},\cY_0,\cZ_0 )\Delta t_0-\sigma(t_0,X_{t_0}^{\pi},Y_{t_0}^{\pi})\Delta W_{t_0},\\
Y_{t_1}^{\pi}&=\cY_0-b(t_0,X_{t_0}^{\pi},\cY_0,\cZ_0 )\Delta t_0+\cZ_0\Delta W_{t_0},
\end{split}
\end{equation*}
and use two neural networks $\cY_1^{\mathcal{N}}(\cdot;\theta_{1,1})$ and $\cZ_1^{\mathcal{N}}(\cdot;\theta_{2,1})$ to approximate respectively $Y_{t_1}^{\pi}=\cY(t_1,X_{t_1}^{\pi})$ and $Z_{t_1}^{\pi}=\cZ(t_1,X_{t_1}^{\pi})$ with associated local loss function
\begin{equation*}
\mathfrak{L}_1:=\bE\bigg|Y_{t_1}^{\pi}-\cY_1^{\mathcal{N}}(	X_{t_1}^{\pi};\theta_{1,1})\bigg|^2.
\end{equation*}
Then for $j=1,2,\cdots,J-2$, using Euler scheme to calculate $X_{t_{j+1}}^{\pi}$ and $Y_{t_{j+1}}^{\pi}$ gives
\begin{align*}
%\begin{split}
X_{t_{j+1}}^{\pi}&=X_{t_j}^{\pi}+\mu\big(t_j,X_{t_j}^{\pi},\cY_j^{\mathcal{N}}(	X_{t_j}^{\pi};\theta_{1,j}),\cZ_j^{\mathcal{N}}(	X_{t_j}^{\pi};\theta_{2,j}) \big)\Delta t_j-\sigma(t_j,X_{t_j}^{\pi},Y_{t_j}^{\pi})\Delta W_{t_j},\\
Y_{t_{j+1}}^{\pi}&=\cY_j^{\mathcal{N}}(	X_{t_j}^{\pi};\theta_{1,j})-b\big(t_j,X_{t_j}^{\pi},\cY_j^{\mathcal{N}}(	X_{t_j}^{\pi};\theta_{1,j}),\cZ_j^{\mathcal{N}}(	X_{t_j}^{\pi};\theta_{2,j}) \big)\Delta t_j+\cZ_j^{\mathcal{N}}(	X_{t_j}^{\pi};\theta_{2,j})\Delta W_{t_j}.
%\end{split}
\end{align*}
Here, the difference from Deep BSDE-2 is lying in that the computation of $Y_{t_{j+1}}^{\pi}$ is based on the neural networks $\cY_j^{\mathcal{N}}(	X_{t_j}^{\pi};\theta_{1,j})$ and $\cZ_j^{\mathcal{N}}(	X_{t_j}^{\pi};\theta_{2,j})$, rather than $Y_{t_{j}}^{\pi}$ and $Z_{t_{j}}^{\pi}$. Further, using two neural networks $\cY_{j+1}^{\mathcal{N}}(\cdot;\theta_{1,j+1})$ and $\cZ_{j+1}^{\mathcal{N}}(\cdot;\theta_{2,j+1})$ to approximate respectively $Y_{t_{j+1}}^{\pi}=\cY(t_{j+1},X_{t_{j+1}}^{\pi})$ and $Z_{t_{j+1}}^{\pi}=\cZ(t_{j+1},X_{t_{j+1}}^{\pi})$ with associated local loss function
\begin{equation*}
\mathfrak{L}_{j+1}:=\bE\bigg|Y_{t_{j+1}}^{\pi}-\cY_{j+1}^{\mathcal{N}}(	X_{t_{j+1}}^{\pi};\theta_{1,j+1})\bigg|^2.
\end{equation*}
Finally, using Euler scheme to calculate $X_{t_{J}}^{\pi}$ and $Y_{t_{J}}^{\pi}$ gives
\begin{align*}
%\begin{split}
X_{t_{J}}^{\pi}&=X_{t_{J-1}}^{\pi}+\mu\big(t_{J-1},X_{t_{J-1}}^{\pi},\cY_{J-1}^{\mathcal{N}}(	X_{t_{J-1}}^{\pi};\theta_{1,{J-1}}),\cZ_{J-1}^{\mathcal{N}}(	X_{t_{J-1}}^{\pi};\theta_{2,{J-1}}) \big)\Delta t_{J-1}\\
&-\sigma(t_{J-1},X_{t_{J-1}}^{\pi},Y_{t_{J-1}}^{\pi})\Delta W_{t_{J-1}},\\
Y_{t_{J}}^{\pi}&=\cY_{J-1}^{\mathcal{N}}(	X_{t_{J-1}}^{\pi};\theta_{1,{J-1}})-b\big(t_{J-1},X_{t_{J-1}}^{\pi},\cY_{J-1}^{\mathcal{N}}(	X_{t_{J-1}}^{\pi};\theta_{1,{J-1}}),\cZ_{J-1}^{\mathcal{N}}(	X_{t_{J-1}}^{\pi};\theta_{2,{J-1}}) \big)\Delta t_{J-1}\\
&+\cZ_{J-1}^{\mathcal{N}}(	X_{t_{J-1}}^{\pi};\theta_{2,{J-1}})\Delta W_{t_{J-1}},
%\end{split}
\end{align*}
and define local loss function 
\begin{equation*}
\mathfrak{L}_{J}:=\bE\bigg|Y_{t_{J}}^{\pi}-g(X_{t_{J}}^{\pi})\bigg|^2.
\end{equation*}
The scheme is to minimize the global loss function
\begin{equation*}
\widehat{\mathfrak{L}}(\cY_0,\cZ_0,\theta_{1,1},\theta_{2,1},\cdots,\theta_{1,J},\theta_{2,J}):=\sum_{j=1}^{J}\mathfrak{L}_{j}.
\end{equation*}
over all $\theta=(\cY_0,\cZ_0,\theta_{1,1},\theta_{2,1},\cdots,\theta_{1,J},\theta_{2,J})$. If  
\begin{equation*}
\theta^* =(\cY^*_0,\cZ^*_0,\theta^*_{1,1},\theta^*_{2,1},\cdots,\theta^*_{1,J},\theta^*_{2,J})\in \arg \min_{\theta\in\Theta} \widehat{\mathfrak{L}}(\theta),
\end{equation*}
then $\cY_0^*$ is the desired approximation of $Y^{\pi}_{t_0}$ by Deep BSDE-3.

  In contrast with Deep BSDE-2, $Y_{t_{j}}^{\pi}$ and $ Z_{t_{j}}^{\pi}$ in Deep BSDE-3   are replaced by $\cY_j^{\mathcal{N}}(	X_{t_j}^{\pi};\theta_{1,j})$ and $\cZ_j^{\mathcal{N}}(	X_{t_j}^{\pi};\theta_{2,j})$ in Euler scheme to calculate $Y_{t_{j+1}}^{\pi}$ for $j=1,2,\cdots,J-1$.

\section{Numerical examples}
In this section, first, we will discuss the finite-dimensional framework for FBSPDEs on the domain $D=(0,1)$ with homogeneous Dirichlet boundary conditions. Then, we will solve two examples of FBSPDEs with the finite element method and deep learning schemes.
\subsection{Framework for Homogeneous Dirichlet Boundary}\label{frameDC}
Let $D\subset \mathbb{R}$ with $D=(0,1)$ and $\mathcal{T}_h:0=x_0< x_1<\cdots< x_L< x_{L+1}=1$ be a partition of the domain $D$ into $L+1$ subintervals $I_j=(x_{j-1},x_j)$ with $h_j=|I_j|=x_j-x_{j-1},j=1,2,\cdots,L,L+1$. Define $h=\max\{h_j:j=1,2,\cdots,L,L+1\}$. Then let $\{\phi_h^1,\phi_h^2,\cdots,\phi_h^L\}$ be the set of nodal basis functions corresponding to the internal nodes  $\{x_1,x_2,\cdots,x_L\}$ which span the finite dimensional function space $V_h^0$. Nodal basis functions $\phi_h^i, i=1,2,\cdots,L$ are hat functions and given by 
{\small
\begin{equation}
\phi_h^i(x)=\left\{\begin{array}{l}
\begin{split}
\frac{x-x_{i-1}}{h_i},&\quad\text{ for}\quad x_{i-1}\le x\le x_i\\
\frac{x_{i+1}-x}{h_{i+1}},&\quad\text{ for}\quad x_{i}\le x\le x_{i+1}\\
0,\quad\quad&\quad \text{ else.}
\end{split}
\end{array}\right.
\end{equation}
}
Set $\phi_h^{i,-1}(x)=\frac{x-x_{i-1}}{h_i}$ and $\phi_h^{i,+1}(x)=\frac{x_{i+1}-x}{h_{i+1}}$.
%We have 
%\begin{equation}
%\nabla\phi_h^i(x)=\left\{\begin{array}{l}
%\begin{split}
%\frac{1}{h_i},&\quad\text{ for}\quad x_{i-1}\le x\le x_i\\
%-\frac{1}{h_{i+1}},&\quad\text{ for}\quad x_{i}\le x\le x_{i+1}\\
%0,\quad\quad&\quad \text{ else.}
%\end{split}
%\end{array}\right.
%\end{equation}
%
%Obvisouly, the basis functions satisfy $\phi_h^i(x_j)=\delta_{ij}$, $i,j=1,2,\cdots,L$, where $\delta_{ij}$ is the Kronecker delta and for $|i-j|>1$, $\phi_h^i$ and $\phi_h^j$ have disjoint non-overlapping supports. 
%Now let's compute the matrices $\mathbf{A},\mathbf{B}$ and $\mathbf{M}$ in finite dimesional FBSDE (\ref{fdFSDE}) and (\ref{fdBSDE}).

%\subsubsection{Mass Matrix $\mathbf{A}$}
The mass matrix $\mathbf{A}=[a_{ij}]_{i,j=1}^L$ is endowed with entries $a_{ij}=\langle \phi_h^i,\phi_h^j \rangle$. 
%Since basis functions have disjoint non-overlapping support in case of $|i-j|>1$, it is obvious that
%\begin{equation*}
%a_{ij}=\langle\phi_h^i,\phi_h^j\rangle=\int_{0}^{1}\phi_h^i(x)\cdot\phi_h^j(x)dx=0.
%\end{equation*}
%Then for $i=j$, we have
%\begin{equation*}
%\begin{split}
%a_{ii}=\langle\phi_h^i,\phi_h^i\rangle&=\int_{0}^{1}\phi_h^i(x)\cdot\phi_h^i(x)dx\\
%&=\int_{x_{i-1}}^{x_i}\left(\frac{x-x_{i-1}}{h_i}\right)^2dx+\int_{x_i}^{x_{i+1}}\left(\frac{x_{i+1}-x}{h_{i+1}}\right)^2dx\\
%&=\frac{1}{h_i^2}\left[\frac{(x-x_{i-1})^3}{3}\right]_{x_{i-1}}^{x_i}+\frac{1}{h_{i+1}^2}\left[\frac{-(x_{i+1}-x)^3}{3}\right]_{x_i}^{x_{i+1}}\\
%&=\frac{h_i}{3}+\frac{h_{i+1}}{3}.
%\end{split}
%\end{equation*}
%Now the off-diagonal elements
%\begin{equation*}
%\begin{split}
%a_{i,i+1}=\langle\phi_h^i,\phi_h^{i+1}\rangle&=\int_{0}^{1}\phi_h^i(x)\cdot\phi_h^{i+1}(x)dx=\int_{x_i}^{x_{i+1}}\frac{(x_{i+1}-x)}{h_{i+1}}\frac{(x-x_i)}{h_{i+1}}dx\\
%&=\frac{1}{h_{i+1}^2}\int_{x_i}^{x_{i+1}}(x_{i+1}x-x_{i+1}x_i-x^2+xx_i)dx\\
%&=\frac{1}{h_{i+1}^2}\left[\frac{x_{i+1}x^2}{2}-x_{i+1}x_ix-\frac{x^3}{3}+\frac{x^2x_i}{2}\right]_{x_i}^{x_{i+1}}\\
%&=\frac{1}{6h_{i+1}^2}(x_{i+1}^3-3x_{i+1}^2x_i+3x_{i+1}x_i^2+x_i^3)\\
%&=\frac{1}{6h_{i+1}^2}(x_{i+1}-x_i)^3=\frac{h_{i+1}}{6}
%\end{split}
%\end{equation*}
%and similarly
%\begin{equation*}
%\begin{split}
%a_{i+1,i}=\langle\phi_h^{i+1},\phi_h^{i}\rangle&=\int_{0}^{1}\phi_h^{i+1}(x)\cdot\phi_h^{i}(x)dx=\cdots=\frac{h_{i+1}}{6}.\\
%\end{split}
%\end{equation*}
Basic calculations imply that the symmetric Mass matrix $\mathbf{A}$ is given by
{\small $$
\mathbf{A}=\begin{pmatrix}
\frac{h_1}{3}+\frac{h_{2}}{3}&\frac{h_2}{6}&0&\cdots&0&0\\
\frac{h_2}{6}&\frac{h_2}{3}+\frac{h_3}{3}&\frac{h_3}{6}&\cdots&0&0\\
0&\cdots&\cdots&\cdots&\cdots&0\\
\vdots&\vdots&\vdots&\vdots&\vdots&\vdots\\
0&0&\cdots&\cdots&\frac{h_{L-1}}{3}+\frac{h_{L}}{3}&\frac{h_L}{6}\\
0&0&0&\cdots&\frac{h_{L}}{6}&\frac{h_L}{3}+\frac{h_{L+1}}{3}
\end{pmatrix}_{L\times L}.
$$
}
%In case of uniform mesh, that is, when $h_i=h,i=1,2,\cdots,L+1$, we have
%$$
%\mathbf{A}_{unif}=h\begin{pmatrix}
%\frac{2}{3}&\frac{1}{6}&0&\cdots&0&0\\
%\frac{1}{6}&\frac{2}{3}&\frac{1}{6}&\cdots&0&0\\
%0&\cdots&\cdots&\cdots&\cdots&0\\
%\vdots&\vdots&\vdots&\vdots&\vdots&\vdots\\
%0&0&\cdots&\cdots&\frac{2}{3}&\frac{1}{6}\\
%0&0&0&\cdots&\frac{1}{6}&\frac{2}{3}
%\end{pmatrix}_{L\times L}=\frac{h}{6}\begin{pmatrix}
%4&1&0&\cdots&0&0\\
%1&4&1&\cdots&0&0\\
%0&\cdots&\cdots&\cdots&\cdots&0\\
%\vdots&\vdots&\vdots&\vdots&\vdots&\vdots\\
%0&0&\cdots&\cdots&4&1\\
%0&0&0&\cdots&1&4
%\end{pmatrix}_{L\times L}.
%$$

%\subsubsection{Stiffness Matrix $\mathbf{B}$}
In a similar way, the stiffness matrix $\mathbf{B}=[b_{ij}]_{i,j=1}^L$ is equipped with entries $b_{ij}=\langle \nabla\phi_h^i,\nabla\phi_h^j \rangle$. 
%Since $\nabla\phi_h^i$ and $\nabla\phi_h^j$ also have disjoint non-overlapping support when $|i-j|>1$,  we have 
%\begin{equation*}
%b_{ij}=\langle\nabla\phi_h^i, \nabla\phi_h^j\rangle=\int_{0}^{1}\nabla\phi_h^i(x)\cdot\nabla\phi_h^j(x)dx=0.
%\end{equation*}
%When $i=j$, we have
%\begin{equation*}\begin{split}
%b_{ii}=\int_0^1(\nabla\phi_h^i(x))^2dx=&\int_{x_{i-1}}^{x_i}\left(\frac{1}{h_i}\right)^2dx+\int_{x_{i}}^{x_{i+1}}\left(-\frac{1}{h_{i+1}}\right)^2dx\\
%=&\frac{x_i-x_{i-1}}{h_i^2}+\frac{x_{i+1}-x_i}{h_{i+1}^2}\\
%=&\frac{1}{h_i}+\frac{1}{h_{i+1}}
%\end{split}
%\end{equation*}
%Then the off diagonal elements
%\begin{equation*}\begin{split}
%b_{i,i+1}=\int_{0}^{1}\nabla\phi_h^i(x)\cdot\nabla\phi_h^{i+1}(x)dx=&\int_{x_i}^{x_{i+1}}\left(-\frac{1}{h_{i+1}}\right)\cdot\frac{1}{h_{i+1}}dx\\
%=&-\frac{x_{i+1}-x_i}{h_{i+1}^2}=-\frac{1}{h_{i+1}}.
%\end{split}
%\end{equation*}
%Obviously $b_{i+1,i}=b_{i,i+1}=-\frac{1}{h_{i+1}}$. 
Straightforward computations yield the Stiffness matrix:
{\small 
$$
\mathbf{B}=\begin{pmatrix}
\frac{1}{h_1}+\frac{1}{h_{2}}&-\frac{1}{h_2}&\cdots&0&0&0\\
-\frac{1}{h_2}&\frac{1}{h_2}+\frac{1}{h_3}&-\frac{1}{h_3}&\cdots&0&0\\
0&\cdots&\cdots&\cdots&\cdots&0\\
\vdots&\vdots&\vdots&\vdots&\vdots&\vdots\\
0&0&\cdots&\cdots&\frac{1}{h_{L-1}}+\frac{1}{h_{L}}&-\frac{1}{h_L}\\
0&0&0&\cdots&-\frac{1}{h_{L}}&\frac{1}{h_L}+\frac{1}{h_{L+1}}
\end{pmatrix}_{L\times L}.
$$
}
In addition, the $l-$th component of the vector $\vec{g}(\rho_h(T);\phi_h)$ is defined as $\langle g(\rho_h(T)),\phi_h^l\rangle$. The involved integrals may be evaluated via conventional numerical approximations. 
Here, to evaluate the integrals one can use numerical integration. 
%%%%%%%%%%%%%%%%%%%%%%%%%%%%%%%%%%%%%%%%%%%%%%%%%%%
%%%%%%%%%%%%%%%%%%%%%%%%%%%%%%%%%%%%%%%%%%%%%%%%%%%

\subsection{Example 1}
 Consider the following decoupled FBSPDE with homogeneous Dirichlet boundary conditions:
\begin{equation}\label{Ex1FSPDE}
\left\{\begin{array}{l}
\begin{split}
%&\begin{split}
\rd\rho(t,x)&=\delta\Delta \rho(t,x)\,\rd t-\rho(t,x)\sum_{i=1}^{k}\gamma_t^i \,\rd W_{t}^{i}
\\
%\end{split}\\
\rho(0,x)&=\rho_0(x), \\
\rho(t,x)\bigg|_{\partial D}&=0,
\\
%\end{split}
%\end{array}\right.
%\end{equation}
%and
%\begin{equation}\label{Ex1BSPDE}
%\left\{\begin{array}{l}
%\begin{split}
%&\begin{split}
-\rd u(t,x)&=\bigg(\delta\Delta u(t,x)+\sum_{i=1}^{k} \gamma_{t}^{i} \psi^i(t,x)+f(t,x,\rho(t,x))\bigg)\,\rd t-\sum_{i=1}^{k}\psi^i (t,x)\,\rd W_{t}^{i}
\\
%\end{split}\\
u(T,x)&=g(\rho(T,x)), \\
u(t,x)\bigg|_{\partial D}&=0.
\end{split}
\end{array}\right.
\end{equation}
Here, $ k=1, g(x)=1-e^{-x},\rho_0>0, \gamma_t\equiv \gamma$ is aconstant and $f(t,x,\rho(t,x))$ is given by
\begin{align*}
%	\begin{split}
	f(t,x,\rho(t,x))=&\delta\left(\bE[\nabla \rho_0(x+\sqrt{2\delta}B_t)]e^{-\gamma W_t-\frac{1}{2}\gamma^2t}\right)^2e^{-\rho(t,x)}+\frac{1}{2}\gamma^2\rho^2(t,x)e^{-\rho(t,x)}\\
	&+\gamma^2\rho(t,x)e^{-\rho(t,x)}-2\delta\bE[\Delta \rho_0(x+\sqrt{2\delta}B_t)]e^{-\gamma W_t-\frac{1}{2}\gamma^2t}e^{-\rho(t,x)}.
%	\end{split}
\end{align*}
The analytic solution of above FBSPDE gives
\begin{equation*}
	\rho(t,x)=\bE[\rho_0(x+\sqrt{2\delta} B_t)]e^{-\gamma W_t-\frac{1}{2}\gamma^2t},
\end{equation*}
and
\begin{equation*}
	u(t,x)=1-e^{-\rho(t,x)}, 
\end{equation*}
where $B_t$ is an (auxilliary) standard Brownian motion independent of $W_t$.
\par

%Finite dimensional FBSDE corresponding to above FBSPDE under finite dimensional framework described in Section \ref{frameDC} becomes
The approximating finite dimensional FBSDEs are of the following form:
\begin{equation}\label{Ex1FDforw}
\left\{\begin{array}{l}
\begin{split}
\mathrm{d}\vec{\rho}_h(t)&=\mu(t,\vec{\rho}_h(t))\mathrm{d}t- \sigma(t,\vec{\rho}_h(t))\mathrm{d}W(t),\quad
\vec{\rho}_h(0)=\mathbf{A}^{-1}\vec{\rho_0}_\phi,\\
-\mathrm{d}\vec{u}_h(t)&=b\big(t,\vec{\rho}_h(t),\vec{u}_h(t),\vec{\psi}_h(t)\big)\mathrm{d}t- \vec{\psi}_h(t)\mathrm{d}W(t),\quad
\vec{u}_h(T)=\mathbf{A}^{-1}\vec{g}(\rho_h(T);\phi_h),
\end{split}
\end{array}\right.
\end{equation}
with 
\begin{equation*}
\begin{split}
\mu(t,\vec{\rho}_h):=&-\delta\mathbf{A}^{-1}\mathbf{B}\vec{\rho}_h,\\
\sigma(t,\vec{\rho}_h):=&\gamma_t\vec{\rho}_h,\\
b(t,\vec{\rho}_h(t),\vec{u}_h(t),\vec{\psi}_h(t)):=&-\delta\mathbf{A}^{-1}\mathbf{B}\vec{u}_h+\gamma_t\vec{\psi}_h+\mathbf{A}^{-1}\vec{f}(t,\rho_h(t);\phi_h).
\end{split}
\end{equation*}

We choose $T=0.5, \delta=0.20, \gamma=1.0,\rho_0(x)=\sin(\pi x)$ and the solution of above finite dimensional FBSDE is approximated by using Deep BSDE-1 algorithm. We adopt uniform mesh with $L$ internal nodes and uniform time grid. We use fully connected neural network comprising 2 hidden layers with $L+10$ neurons in each layer. Hyperbolic tangent is used as activation function for hidden layers and Adam optimizer adopted for training. We set batch size 512 for training purpose. The forward process is numerically approximated at the time grid by backward Euler scheme:\\
\begin{equation*}
\vec{\rho}_h(t_{j+1})=(I+\delta\mathbf{A}^{-1}\mathbf{B}\Delta t)^{-1}\big(\vec{\rho}_h(t_{j})-\gamma\vec{\rho}_h(t_j)\Delta W_{t_j}\big),\; j=0,1,\dots,J-1,\; \vec{\rho}_h(0)=\mathbf{A}^{-1}\vec{\rho_0}_\phi.
\end{equation*}

%The $l$-th component of vector $\vec{\rho_0}_\phi$ is defined as $\langle \rho_0,\phi_h^l\rangle$, where $\phi_h^l$ is the nodal basis corresponding to the internal node $x_l,l=1,2,\cdots,L$ of the finite dimensional function space $V_h^0$. We can have for $1\le l\le L$
%
%\begin{equation*}
%	\begin{split}
%	\langle \rho_0,\phi_h^l\rangle=&\int_{x_{l-1}}^{x_{l+1}}\rho_0(x)\phi_h^l(x)dx=\int_{x_{l-1}}^{x_{l+1}}\sin(\pi x)\phi_h^l(x)dx\\
%	=&\int_{x_{l-1}}^{x_{l}}\sin(\pi x)\frac{x-x_{l-1}}{h} dx+\int_{x_{l}}^{x_{l+1}}\sin(\pi x)\frac{x_{l+1}-x}{h} dx\\
%	=&\frac{1}{h}\left[\int_{x_{l-1}}^{x_{l}}\{x \sin(\pi x)-x_{l-1} \sin(\pi x)\} dx+\int_{x_{l}}^{x_{l+1}}\{x_{l+1}\sin(\pi x)-x\sin(\pi x)\} dx\right]\\
%	=&\frac{1}{h}\bigg[\frac{-x\cos(\pi x)}{\pi}+\frac{\sin(\pi x)}{\pi^2}+x_{l-1}\frac{\cos(\pi x)}{\pi}\bigg]_{x_{l-1}}^{x_l}\\
%	&+\frac{1}{h}\bigg[-x_{l+1}\frac{\cos(\pi x)}{\pi}+\frac{x\cos(\pi x)}{\pi}-\frac{\sin(\pi x)}{\pi^2}\bigg]_{x_{l}}^{x_{l+1}}\\
%	=&\frac{1}{h}\bigg[\frac{-x_l\cos(\pi x_l)}{\pi}+\frac{\sin(\pi x_l)}{\pi^2}+x_{l-1}\frac{\cos(\pi x_l)}{\pi}+x_{l-1}\frac{\cos(\pi x_{l-1})}{\pi}-\frac{\sin(\pi x_{l-1})}{\pi^2}\\&-x_{l-1}\frac{\cos(\pi x_{l-1})}{\pi}-x_{l+1}\frac{\cos(\pi x_{l+1})}{\pi}+\frac{x_{l+1}\cos(\pi x_{l+1})}{\pi}-\frac{\sin(\pi x_{l+1})}{\pi^2}+x_{l+1}\frac{\cos(\pi x_l)}{\pi}\\
%	&-\frac{x_l\cos(\pi x_l)}{\pi}+\frac{\sin(\pi x_l)}{\pi^2}\bigg]\\
%	=&\frac{1}{\pi^2 h}\bigg[-\sin(\pi x_{l-1})+2\sin(\pi x_l)-\sin(\pi x_{l+1})\bigg].
%	\end{split}
%\end{equation*}

 With Deep BSDE-1, we simulate the approximate solution for $L=5,15,20$ with $\Delta t=0.05$, $L=25$ with $\Delta t=0.025$, $L=35$ with $\Delta t=0.167$ and $L=50$ with $\Delta t=0.001$. Figure \ref{fig_Ex_1} shows the mean value of $u(0,x)$ from 10 simulations and Table \ref{Rel_err_DBSDE1} shows the relative errors.  Letting $\hat{u}(0,x)$ be the approximation of $u(0,x)$, we investigate the relative error:
 \begin{equation*}
 	R_E= \frac{\int_D \left|u(0,x) -\frac{1}{10}\sum_{i=1}^{10}\hat{u}^i(0,x)\right|^2 \rd x}{\int_D |u(0,x)|^2 \rd x} .
 \end{equation*}
 In Example-1, we can see that Deep BSDE-1 improves the accuracy of the approximations as mesh sizes for both space and time domains decrease.

% ; possibly, this may be due to the cumulated errors from the neural network approximations on each time step, or maybe, a refined division of time interval might be needed.

\begin{figure}
\begin{floatrow}
\ffigbox{%
    \includegraphics[width=8cm, height=5.6cm]{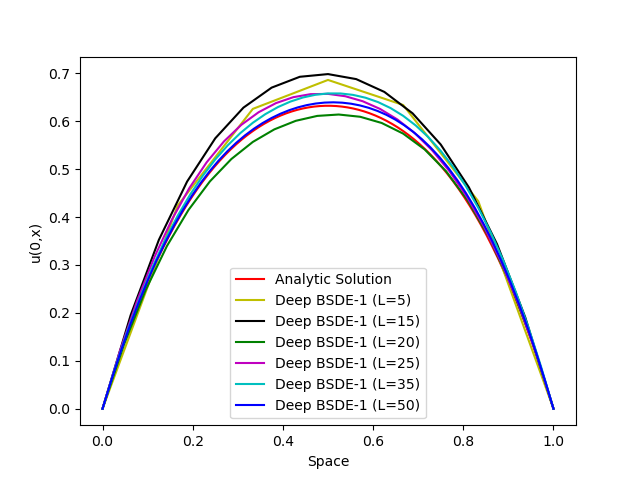}
}{%
  \caption{Example 1 ($\delta=0.2$, deep BSDE-1)}%
  \label{fig_Ex_1}
}
\capbtabbox{%
	\begin{tabular}{|c | c |c|} 
		\hline
		$L$& Relative Error & $\Delta t$\\ [0.5ex] 
		\hline 
		$5$ & 0.004395 & .05 \\ [0.5ex] 
		\hline
		$15$ & 0.009749  & .05 \\ [0.5ex] 
		\hline
		$20$ & 0.000893  & .05\\ [0.5ex] 
		\hline
		$25$ & 0.001699  & .025\\ [0.5ex] 
		\hline 
		$35$ & 0.001905  & .0167 \\ [0.5ex] 
		\hline
		$50$ & 0.000294  & .001 \\ [0.5ex] 
		\hline
	\end{tabular} 
}{%
\bigskip\medskip

  \caption{Relative Error}
	\label{Rel_err_DBSDE1}
}
\end{floatrow}
\end{figure}

%%%%%%%%%%%%%%%%%%%%%%%%%%%%%%%%%%%%%%%%%%%%%%%%%%%%%%%%%%%%%%%%%%%%%%%%%%%%%%%%%%%%%%%%%%%%%%%%%%%%%%%%%%%%%%%%%%%%%%%%%%%%%%%%%%%%%%%%%%%%%%%%%%%%%%%%%%%%%%%%%%%%%%%%%%%%%%%%%%

\subsection{Example 2}
 Consider the following coupled and \emph{nonlocal} FBSPDE:
\begin{equation}\label{Ex3FSPDE}
\left\{\begin{array}{l}
\begin{split}
\rd\rho(t,x)&=\bigg(\delta\Delta \rho(t,x)+f_1\big(t,x,\rho(t,x),u(t,x)\big)\bigg)\,\rd t-f_3(t,x)\rd W_t,
\\
\rho(0,x)&=\rho_0(x)=\frac{\pi}{2}\sin(\pi x)+\frac{1}{2}\sin(2\pi x), \\
\rho(t,x)\bigg|_{\partial D}&=0,\\
%\end{split}
%\end{array}\right.
%\end{equation}
%\begin{equation}\label{Ex3BSPDE}
%\left\{\begin{array}{l}
%\begin{split}
%&\begin{split}
-\rd u(t,x)&=\bigg(\delta\Delta u(t,x)+f_2\big(t,x,\rho(t,x),u(t,x)\big)\bigg)\,\rd t-\psi(t,x)\rd W_t
\\
%\end{split}\\
u(T,x)&=g(\rho(T,x))=\arctan(\rho(T,x)), \\
u(t,x)\bigg|_{\partial D}&=0,
\end{split}
\end{array}\right.
\end{equation}
with
\begin{align*}
%\begin{split}
f_1\big(t,x,\rho(t),u(t)\big)=&\alpha\cdot\cos(u(t))-\frac{\alpha}{\sqrt{1+\rho(t)^2}}+\delta\pi^2\rho(t)+\delta\cdot\frac{2+\cos(W_t)}{2}\cdot\pi^2\sin(2\pi x)\\
&-\frac{\cos(W_t)}{12}\cdot\sin(2\pi x),\\
f_2\big(t,x,\rho(t),u(t)\big)=&\frac{2\delta\rho(t)}{\big(1+\rho(t)^2\big)^2}\cdot\bigg|\frac{\pi^2}{2}\cos(\pi x)+\frac{2+\cos(W_t)}{3}\cdot\pi\cos(2\pi x)\bigg|^2\\
&+\frac{2\delta}{1+\rho(t)^2}\cdot\bigg(\frac{\pi^3}{2}\sin(\pi x)+\frac{2+\cos(W_t)}{3}\cdot 2\pi^2\sin(2\pi x)\bigg)\\
&+\alpha\cdot u(t)-\alpha\cdot\arctan(\rho(t))+\frac{\rho(t)}{\big(1+\rho(t)^2\big)^2}\cdot\frac{\sin^2(W_t)\cdot\sin^2(2\pi x)}{36}\\
&-\frac{1}{1+\rho(t)^2}\bigg(\delta\pi^2\rho(t)+\delta\cdot\frac{2+\cos(W_t)}{2}\cdot\pi^2\sin(2\pi x)-\frac{\cos(W_t)}{12}\cdot\sin(2\pi x)\bigg)\\
&+\gamma\cdot\bigg(\int_{0}^{1}\sin(2\pi x)\rho(t)\rd x-\frac{2+\cos(W_t)}{12}\bigg),\\
f_3(t,x)=&\frac{\sin(W_t)}{6}\cdot\sin(2\pi x).
%\end{split}
\end{align*}
The analytic solution gives
\begin{align*}
%\begin{split}
\rho(t,x)=&\frac{\pi}{2}\sin(\pi x)+\frac{2+\cos(W_t)}{6}\sin(2\pi x),\\
u(t,x)=&\arctan(\rho(t,x)).
%\end{split}
\end{align*} 
The approximating finite dimensional FBSDE is given by
\begin{equation}\label{Ex3FDforw}
\left\{\begin{array}{l}
\begin{split}
\mathrm{d}\vec{\rho}_h(t)&=\mu(t,\vec{\rho}_h(t),\vec{u}_h(t))\mathrm{d}t-\sigma(t)\rd W_t,\quad
\vec{\rho}_h(0)=\mathbf{A}^{-1}\vec{\rho_0}_\phi,\\
-\mathrm{d}\vec{u}_h(t)&=b\big(t,\vec{\rho}_h(t),\vec{u}_h(t)\big)\mathrm{d}t- \vec{\psi}_h(t)\mathrm{d}W(t),\quad
\vec{u}_h(T)=\mathbf{A}^{-1}\vec{g}(\rho_h(T);\phi_h),
\end{split}
\end{array}\right.
\end{equation}
with 
\begin{align*}
%\begin{split}
\mu(t,\vec{\rho}_h(t),\vec{u}_h(t))=&-\delta\mathbf{A}^{-1}\mathbf{B}\vec{\rho}_h(t)+\mathbf{A}^{-1}\vec{f}_1\big(t,\rho_h(t),u_h(t);\phi_h\big),\\
b(t,\vec{\rho}_h(t),\vec{u}_h(t))=&-\delta\mathbf{A}^{-1}\mathbf{B}\vec{u}_h(t)+\mathbf{A}^{-1}\vec{f}_2\big(t,\vec{\rho}_h(t),u_h(t);\phi_h\big),\\
\sigma(t)=&\mathbf{A}^{-1}\vec{f_3}(t;\phi_h).
%\end{split}
\end{align*}

When $\alpha=0.2=\gamma$ and $\delta=.001$, the resulting finite dimensional FBSDE is coupled and the solution is approximated via both Deep BSDE-2 and Deep BSDE-3 algorithms. In both schemes, we consider $T=0.5$, the uniform time grid with $\Delta t=0.05$ and  the uniform mesh with $L$ internal nodes. We use fully connected neural network comprising 2 hidden layers with $L+10$ neurons in each layer. Hyperbolic tangent is used as activation function for hidden layers and Adam optimizer is used for training. We use batch size 512 for training purpose. %For training measure $\nu_j$ at $t_j$ we consider $\mathcal{N}\big(0,0.5\big)$ for all $j$.\\
For time discretization of forward process, the backward Euler scheme is employed, which is written as:
\begin{align*}
%\begin{split}
&\vec{\rho}_h(t_{j+1})=\big(I+\delta\mathbf{A}^{-1}\mathbf{B}\Delta t\big)^{-1}\bigg(\vec{\rho}_h(t_{j})+\mathbf{A}^{-1}\vec{f}_1\big(t_j,\rho_h(t_j),u_h(t_j);\phi_h\big)\Delta t-\mathbf{A}^{-1}\vec{f_3}(t_j;\phi_h)\Delta W_{t_j}\bigg),\\ &\quad\quad\quad \quad \quad j=0,1,\dots,J-1,\\& \vec{\rho}_h(0)=\mathbf{A}^{-1}\vec{\rho_0}_\phi.
%\end{split}
\end{align*}

With Deep BSDE-2, we compute the approximate solution for $L=5$ and $L=15$. Figure \ref{fig_2_Ex_2} shows the mean values of $u(0,x)$ from 10 runs and Table \ref{Rel_err_DBSDE2} shows the relative errors. The results from Deep BSDE-2 are not stable for $L\geq 20$ and the associated numerical results are not presented here. 

It is worth noting that the approximating performance of Deep BSDE-2 may be improved when $\Delta t$ is smaller, while herein we fix $\Delta t=0.05$ to compare these two methods: Deep BSDE-2 and Deep BSDE-3. Indeed,
with Deep BSDE-3, we simulate the approximate solution for $L=5,15,30$ and $L=50$. Figure \ref{fig_5_Ex_2} shows the mean value of $u(0,x)$ from 10 runs and Table \ref{Rel_err_DBSDE3} shows the relative errors. Clearly, Deep BSDE-3 provides stable solutions for even bigger values of $L$ compared to Deep BSDE-2 and the accuracy of Deep BSDE-3 is also higher than Deep BSDE-2.

\begin{figure}
\begin{floatrow}
\ffigbox{%
\includegraphics[width=8cm, height=8cm]{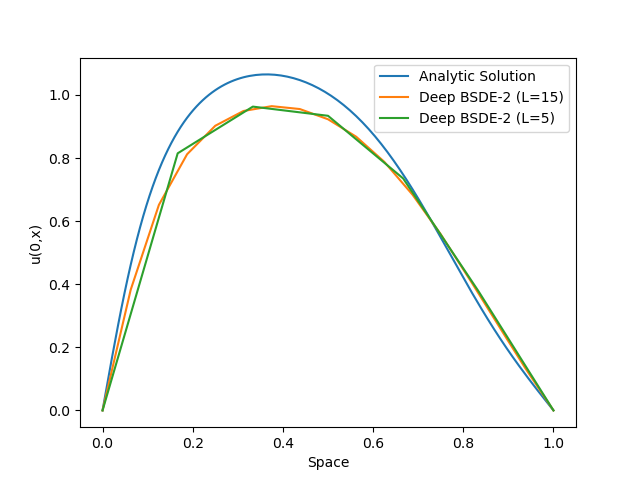}
}{%
	\caption{Example 2 (Deep BSDE-2) }
	\label{fig_2_Ex_2}
}
\capbtabbox{%
	\begin{tabular}{|c | c|}
		\hline
		 $L$ & Relative Error  \\ [0.5ex] 
		\hline
		5 & 0.012639 \\ [0.5ex] 
		\hline
			15 & 0.010200 \\ [0.5ex] 
		\hline
	\end{tabular}
}{%
\bigskip
\bigskip

\caption{Relative Error}
\label{Rel_err_DBSDE2}
}
\end{floatrow}
\end{figure}

%\begin{figure}
%	\centering
%	\includegraphics[width=8cm, height=6cm]{Ex_2Deep2}
%	\caption{Example 2 (Deep BSDE-2) }
%	\label{fig_2_Ex_2}
%\end{figure}
%
%
%
%
%\begin{table}
%	\centering
%	\begin{tabular}{|c | c|}
%		\hline
%		 $L$ & Relative Error  \\ [0.5ex] 
%		\hline
%		5 & 0.104965 \\ 
%		\hline
%			15 & 0.095891 \\ 
%		\hline
%	\end{tabular}
%\caption{Relative Error from Deep BSDE-2}
%\label{Rel_err_DBSDE2}
%\end{table}

\begin{figure}
\begin{floatrow}
\ffigbox{%
\includegraphics[width=8cm, height=8cm]{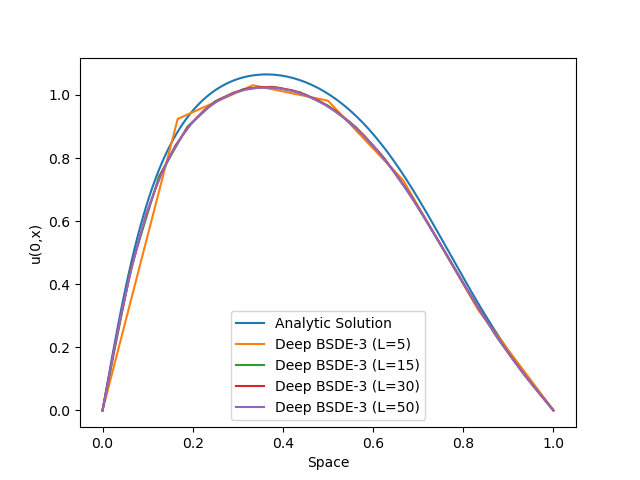}
}{%
	\caption{Example 2 (Deep BSDE-3)}
	\label{fig_5_Ex_2}
}
\capbtabbox{%
	\begin{tabular}{|c | c |} 
		\hline
		 $L$& Relative Error\\ [0.5ex] 
		 \hline
		 $5$ & 0.003284 \\ [0.5ex] 
		\hline
		 $15$ & 0.001618 \\ [0.5ex] 
		\hline
		$30$ & 0.001551 \\ [0.5ex] 
		\hline
		$50$ & 0.001671 \\ [0.5ex] 
		\hline
	\end{tabular}
}{%
\bigskip
\bigskip

\caption{Relative Error from Deep BSDE-3}
\label{Rel_err_DBSDE3}
}
\end{floatrow}
\end{figure}

\appendix
\section{Appendix}
\subsection{Proof of Lemma \ref{lem1}}

\begin{proof}[Proof of Lemma \ref{lem1}]
	The proof of the existence and uniqueness of solution to FBSDE \eqref{dFSPDE}-\eqref{dBSPDE} and $Assertion\;(i)$ is the same as that of Theorem \ref{Thm1}. We only need to prove $Assertion\;(ii)$. Computations involved in this proof will be divided into two parts.
	%\end{proof}
	%\begin{proof}$Assertion\;(ii)$.
	\par \textbf{Step 1}. This part is devoted to some estimates associated to the forward equation \eqref{dFSPDE}. Fix $t\in[0,T]$. Applying It\^o's formula to equation (\ref{dFSPDE}) for $\nabla\rho_h(t)$ yields $\bP$-a.s.
	\begin{align}
	%\begin{split}
	\|&\nabla\rho_h(t)\|^2
	=\|\nabla\rho_0\|^2+2\int_{0}^{t}\big\langle\nabla\Delta_h\rho_h(s),\nabla\rho_h(s)\big\rangle\rd s+ \int_{0}^{t}\|\nabla \Pi_h f(s,\rho_h(s),u_h(s))\|^2\rd s
	\notag\\
	&
	\quad+2\int_{0}^{t}\big\langle \nabla \Pi_h F\big(s,\rho_h(s),\nabla\rho_h(s),u_h(s),\nabla u_h(s),{\psi_h}(s)\big),\nabla\rho_h(s)\big\rangle\rd s
	\notag\\
	&\quad
	-2\sum_{i=1}^{k}\int_{0}^{t}\big\langle\nabla \Pi_hf^i(s,\rho_h(s),u_h(s)),\nabla\rho_h(s)\big\rangle\rd W^i_s
	\notag\\
	&=\|\nabla\rho_0\|^2-2\int_{0}^{t}\| \Delta_h\rho_h(s) \|^2 \rd s+\int_{0}^{t}\|\nabla \Pi_h f(s,\rho_h(s),u_h(s))\|^2\rd s
	\notag\\
	&
	\quad+2\int_{0}^{t}\big\langle \nabla \Pi_h F\big(s,\rho_h(s),\nabla\rho_h(s),u_h(s),\nabla u_h(s),{\psi_h}(s)\big),\nabla\rho_h(s)\big\rangle\rd s
	\notag\\
	&\quad-2\sum_{i=1}^{k}\int_{0}^{t}\big\langle\nabla \Pi_hf^i(s,\rho_h(s),u_h(s)),\nabla\rho_h(s)\big\rangle\rd W^i_s. \label{lem1ref1}
	%\end{split}
	\end{align}
	%\begin{equation*}
	%\begin{split}
	%\|&\nabla\rho_h(t)\|^2=\|\nabla\rho_0\|^2-2\int_{0}^{t}\big\langle\Delta_h\rho_h(s),\Delta_h\rho_h(s)\big\rangle\rd s-2\sum_{i=1}^{k}\int_{0}^{t}\big\langle\nabla \Pi_hf^i(s,\rho_h(s)),\nabla\rho_h(s)\big\rangle\rd W^i_s\\
	%&
	%+\sum_{i=1}^{k}\int_{0}^{t}\|\nabla \Pi_h f^i(s,\rho_h(s))\|^2\rd s+2\int_{0}^{t}\big\langle \nabla \Pi_h F\big(s,\rho_h(s),\nabla\rho_h(s),u_h(s),\nabla u_h(s),{\psi_h}(s)\big),\nabla\rho_h(s)\big\rangle\rd s,
	%\end{split}
	%\end{equation*}
%	that is,
%	\begin{align}\label{lem1ref1}
%	\|&\nabla\rho_h(t)\|^2+2\int_{0}^{t}\big\|\Delta_h\rho_h(s)\big\|^2 \rd s=\|\nabla\rho_0\|^2+\sum_{i=1}^{k}\int_{0}^{t}\|\nabla \Pi_h f^i(s,\rho_h(s),u_h(s))\|^2\rd s\nonumber\\
%	&
%	\quad+2\int_{0}^{t}\big\langle \nabla \Pi_hF\big(s,\rho_h(s),\nabla\rho_h(s),u_h(s),\nabla u_h(s),{\psi_h}(s)\big),\nabla\rho_h(s)\big\rangle\rd s\nonumber\\
%	&\quad-2\sum_{i=1}^{k}\int_{0}^{t}\big\langle\nabla \Pi_h f^i(s,\rho_h(s),u_h(s)),\nabla\rho_h(s)\big\rangle\rd W^i_s.
%	\end{align}
	Using Lipschitz property from Assumption \ref{ass1}, we have
	\begin{align*}
	&2\int_{0}^{t}\big\langle \nabla \Pi_h F(s,\rho_h(s),\nabla\rho_h(s),u_h(s),\nabla u_h(s),{\psi_h}(s)),\nabla\rho_h(s)\big\rangle\rd s\\
	&\le 2\int_{0}^{t}\big|\big\langle  \Pi_hF(s,\rho_h(s),\nabla\rho_h(s),u_h(s),\nabla u_h(s),{\psi_h}(s)),\Delta_h\rho_h(s)\big\rangle\big|\rd s\\
	&\le 2\int_{0}^{t}\| \Pi_h F\big(s,\rho_h(s),\nabla\rho_h(s),u_h(s),\nabla u_h(s),{\psi_h}(s)\big)\|\cdot\|\Delta_h\rho_h(s)\|\rd s\\
	&\le 2\int_{0}^{t}\bigg( L^{F}_1\bigg(\|\rho_h(s)\|+\|\nabla\rho_h(s)\|\bigg)+L^{F}_2\bigg(\|u_h(s)\|+\|\nabla u_h(s)\|+\|{\psi_h}(s)\|\bigg)\bigg)\cdot\|\Delta_h\rho_h(s)\|\rd s\\
	&\quad\quad + 2\int_{0}^{t}\|F^0_s|\cdot\|\Delta_h\rho_h(s)\|\rd s\\
	&\le 2\int_{0}^{t}\|F^0_s\|\cdot\|\Delta_h\rho_h(s)\|\rd s+2L^{F}_1\int_{0}^{t}\|\rho_h(s)\|\cdot\|\Delta_h\rho_h(s)\|\rd s\\
	&\quad\quad+2L^{F}_1\int_{0}^{t}\|\nabla\rho_h(s)\|\cdot\|\Delta_h\rho_h(s)\|\rd s+2L^{F}_2\int_{0}^{t}\|u_h(s)\|\cdot\|\Delta_h\rho_h(s)\|\rd s\\
	&\quad\quad+2L^{F}_2\int_{0}^{t}\|\nabla u_h(s)\|\cdot\|\Delta_h\rho_h(s)\|\rd s+2L^{F}_2\int_{0}^{t}\|{\psi_h}(s)\|\cdot\|\Delta_h\rho_h(s)\|\rd s\\
	&\le \int_{0}^{t}\bigg(\frac{1}{\eps_1}\|F^0_s\|^2+\frac{|L^{F}_1|^2}{\eps_2}\|\rho_h(s)\|^2
	+\frac{|L^{F}_2|^2}{\eps_4}\|u_h(s)\|^2+\frac{|L^{F}_2|^2}{\eps_5}\|\nabla u_h(s)\|^2+\frac{|L^{F}_2|^2}{\eps_6}\|\psi_h(s)\|^2\bigg)\rd s\\
	&\quad\quad+\bigg(\eps_1+\eps_2+\eps_3+\eps_4+\eps_5+\eps_6\bigg)\int_{0}^{t}\|\Delta_h\rho_h(s)\|^2\rd s+\frac{|L^{F}_1|^2}{\eps_3}\int_{0}^{t}\|\nabla\rho_h(s)\|^2\rd s.
	\end{align*}
	Taking $\eps_1=\eps_2=\eps_3=\eps_4=\eps_5=\eps_6=\frac{1}{4}$, we  conclude from \eqref{lem1ref1} that
	\begin{equation*}
	\begin{split}
	&\|\nabla\rho_h(t)\|^2+\frac{1}{2}\int_{0}^{t}\|\Delta_h\rho_h(s)\|^2\rd s\\
	&\le
	\|\nabla\rho_0\|^2
	+\int_{0}^{t}\|\nabla \Pi_h f(s,\rho_h(s),u_h(s))\|^2\rd s-2\sum_{i=1}^{k}\int_{0}^{t}\langle\nabla \Pi_h f^i(s,\rho_h(s),u_h(s)),\nabla\rho_h(s)\rangle\rd W^i_s\\
	&
	\quad\quad+4\int_{0}^{t}\bigg(\|F^0_s\|^2+|L^{F}_1|^2\|\rho_h(s)\|^2+|L^{F}_2|^2\|u_h(s)\|^2+|L^{F}_2|^2\|\nabla u_h(s)\|^2+|L^{F}_2|^2\|\psi_h(s)\|^2\bigg)\rd s\\
	&\quad\quad+4|L^{F}_1|^2\int_{0}^{t}\|\nabla\rho_h(s)\|^2\rd s.
	\end{split}
	\end{equation*}
	Noticing $\|\nabla\Pi_hf(s,\rho_h(s),u_h(s))\|\le C_e\| f(s,\rho_h(s),u_h(s))\|_{1,2}\le C_e\tilde{L}^f\big(1+\|\rho_h(s)\|_{1,2}+\|u_h(s)\|_{1,2}\big)$, we have further
	\begin{equation*}
	\begin{split}
	&\|\nabla\rho_h(t)\|^2+\frac{1}{2}\int_{0}^{t}\|\Delta_h\rho_h(s)\|^2\rd s\\
	&\le\|\nabla\rho_0\|^2+3|C_e\tilde{L}^f|^2t+4\int_{0}^{t}\|F^0_s\|^2\rd s-2\sum_{i=1}^{k}\int_{0}^{t}\langle\nabla\Pi_h f^i(s,\rho_h(s),u_h(s)),\nabla\rho_h(s)\rangle\rd W^i_s\\
	&
	\quad\quad+4\int_{0}^{t}\bigg(|L^{F}_1|^2\|\rho_h(s)\|^2+|L^{F}_2|^2\|u_h(s)\|^2+|L^{F}_2|^2\|\nabla u_h(s)\|^2+|L^{F}_2|^2\|\psi_h(s)\|^2\bigg)\rd s\\
	&\quad\quad+\bigg(4|L^{F}_1|^2+3|C_e\tilde{L}^{f}|^2\bigg)\int_{0}^{t}\|\rho_h(s)\|_{1,2}^2\rd s+3|C_e\tilde{L}^{f}|^2\bE\int_{0}^{t}\|u_h(s)\|_{1,2}^2\rd s.
	\end{split}
	\end{equation*}
Taking supremum over $t\in[0,\tau]$ for $\tau\in[0,T]$ and then expectations on both sides we have
	\begin{align}\label{lem1ref2}
	&\bE\bigg[\sup_{t\in [0,\tau]}\|\nabla\rho_h(t)\|^2\bigg]+\frac{1}{2}\bE\int_{0}^{\tau}\|\Delta_h\rho_h(s)\|^2\rd s\nonumber\\
	&\le \bE\bigg[\|\nabla\rho_0\|^2\bigg]+4\bE\int_{0}^{\tau}\|F^0_s\|^2\rd s+2\bE\bigg[\sup_{t\in [0,\tau]}\bigg|\int_{0}^{t}\sum_{i=1}^{k}\langle\nabla\Pi_h f^i(s,\rho_h(s),u_h(s)),\nabla\rho_h(s)\rangle\rd W^i_s\bigg|\bigg]\nonumber\\
	&+\bigg(4|L^{F}_1|^2+3|C_e\tilde{L}^{f}|^2\bigg)\bE\int_{0}^{\tau}\|\rho_h(s)\|_{1,2}^2\rd s+3|C_e\tilde{L}^{f}|^2\bE\int_{0}^{\tau}\|u_h(s)\|_{1,2}^2\rd s+3|C_e\tilde{L}^f|^2\tau
		\nonumber\\
		&
	+4\bE\int_{0}^{\tau}\bigg(|L^{F}_1|^2\|\rho_h(s)\|^2+|L^{F}_2|^2\|u_h(s)\|^2+|L^{F}_2|^2\|\nabla u_h(s)\|^2+|L^{F}_2|^2\|\psi_h(s)\|^2\bigg)\rd s.
	\end{align}
	For the terms involving stochastic integrals, we use BDG inequality to obtain 
	\begin{align*}
	&2\bE\bigg[\sup_{t\in [0,\tau]}\bigg|\int_{0}^{t}\sum_{i=1}^{k}\langle\nabla\Pi_h f^i(s,\rho_h(s),u_h(s)),\nabla\rho_h(s)\rangle\rd W^i_s\bigg|\bigg]\\
	& \le \tilde{C}\bE\bigg[\bigg(\int_{0}^{\tau}\sum_{i=1}^{k}\bigg|\langle\nabla\Pi_h f^i(s,\rho_h(s),u_h(s)),\nabla\rho_h(s)\rangle\bigg|^2\rd s\bigg)^\frac{1}{2}\bigg]\\
	&\le \tilde{C}\bE \bigg[\bigg(\int_{0}^{\tau}|C_e\tilde{L}^f|^2\bigg(1+\|\rho_h(s)\|_{1,2}+\|u_h(s)\|_{1,2}\bigg)^2\cdot\big\|\nabla\rho_h(s)\big\|^2\rd s\bigg)^\frac{1}{2}\bigg]\\
	&\le \tilde{C}\bE \bigg[\bigg(\sup_{t\in [0,\tau]}\big\|\nabla\rho_h(t)\big\|^2\cdot|C_e\tilde{L}^f|^2\int_{0}^{\tau}
	\bigg(1+\|\rho_h(s)\|_{1,2}+\|u_h(s)\|_{1,2}\bigg)^2\rd s\bigg)^\frac{1}{2}\bigg]\\
	&\le \bE \bigg[\eps_7\sup_{t\in [0,\tau]}\big\|\nabla\rho_h(t)\big\|^2+\frac{|\tilde{C}C_e\tilde{L}^f|^2}{\eps_7}\int_{0}^{\tau}\bigg(1+\|\rho_h(s)\|_{1,2}+\|u_h(s)\|_{1,2}\bigg)^2\rd s\bigg]\\
	&\le \eps_7\bE\bigg[\sup_{t\in [0,\tau]}\|\nabla\rho_h(t)\|^2\bigg]+\frac{3|\tilde{C}C_e\tilde{L}^{f}|^2}{\eps_7}\bE\int_{0}^{\tau}\|\rho_h(s)\|_{1,2}^2\rd s+\frac{3|\tilde{C}C_e\tilde{L}^{f}|^2}{\eps_7}\bE\int_{0}^{\tau}\|u_h(s)\|_{1,2}^2\rd s\\
	&\quad+\frac{3|\tilde{C}C_e\tilde{L}^{f}|^2}{\eps_7}\tau,
	\end{align*}
	 with $\eps_7=\frac{1}{2}$. This together with \eqref{lem1ref2} implies that
	\begin{align*}
	%\begin{split}
	&\bE\bigg[\sup_{t\in [0,\tau]}\|\nabla\rho_h(t)\|^2\bigg]+\bE\int_{0}^{\tau}\|\Delta_h\rho_h(s)\|^2\rd s\\
	&\le 2\bE\bigg[\|\nabla\rho_0\|^2\bigg]+8\bE\int_{0}^{\tau}\|F^0_s\|^2\rd s+\bigg(8|L^{F}_1|^2+4(3|\tilde{C}|^2+1)|C_e\tilde{L}^{f}|^2\bigg)\bE\int_{0}^{\tau}\|\rho_h(s)\|_{1,2}^2\rd s\\
	&
	\quad\quad+8\bE\int_{0}^{\tau}\bigg(|L^{F}_1|^2\|\rho_h(s)\|^2+|L^{F}_2|^2\|u_h(s)\|^2+|L^{F}_2|^2\|\nabla u_h(s)\|^2+|L^{F}_2|^2\|\psi_h(s)\|^2\bigg)\rd s\\
	&\quad\quad+4(3|\tilde{C}|^2+1)|C_e\tilde{L}^{f}|^2\tau+|C_e\tilde{L}^f|^2(3+12|\tilde{C}|^2)\bE\int_{0}^{\tau}\|u_h(s)\|_{1,2}^2\rd s.
	%\end{split}
	\end{align*}
	Now using Gronwall's inequality and estimate from $Assertion\;(i)$ we have 
	\begin{align*}
	%\begin{split}
	&\bE\bigg[\sup_{t\in [0,T]}\|\nabla\rho_h(t)\|^2\bigg]+\bE\int_{0}^{T}\|\Delta_h\rho_h(s)\|^2\rd s\\
	&\le C\bE\bigg[1+\|\rho_0\|^2+\|\nabla\rho_0\|^2+\|g^0\|^2\bigg]
	+C\bE\int_{0}^{T}\bigg(\sum_{i=1}^{k}\|f_s^{i,0}\|^2+\|F^0_s\|^2+\|G^0_s\|^2\bigg)\rd s,
	%\end{split}
	\end{align*}
	with $C=C(L_1^{f},L_2^f,\tilde{L}^{f},L^{g},L^{F}_1,L^{F}_2,L_1^{G},L_2^{G},T,C_e)$.
	
	\par \textbf{Step 2}. Then we conduct the computations for the backward equation \eqref{dBSPDE}. Fix $t\in[0,T]$. Applying It\^o's formula to equation (\ref{dBSPDE}) for $\nabla u_h(t)$ yields $\bP$-a.s.
	\begin{equation*}
	\begin{split}
	&\|\nabla u_h(t)\|^2=\|\nabla\Pi_h g(\rho_h(T)\|^2+2\int_{t}^{T}\langle\nabla\Delta_h u_h(s),\nabla u_h(s)\rangle\rd s-2\sum_{i=1}^{k}\int_{t}^{T}\langle\nabla\psi_h^i(s),\nabla u_h(s)\rangle\rd W^i_s\\
	&
	\quad\quad- \int_{t}^{T}\|\nabla\psi_h(s)\|^2\rd s+2\int_{t}^{T}\big\langle \nabla \Pi_h G(s,\rho_h(s),\nabla\rho_h(s),u_h(s),\nabla u_h(s),{\psi_h}(s)),\nabla u_h(s)\big\rangle\rd s\\
	&=\|\nabla\Pi_h g(\rho_h(T)\|^2-2\int_{t}^{T}\langle\Delta_h u_h(s),\Delta_h u_h(s)\rangle\rd s-2\sum_{i=1}^{k}\int_{t}^{T}\langle\nabla\psi_h^i(s),\nabla u_h(s)\rangle\rd W^i_s\\
	&
	\quad\quad- \int_{t}^{T}\|\nabla\psi_h(s)\|^2\rd s+2\int_{t}^{T}\big\langle \nabla \Pi_h G(s,\rho_h(s),\nabla\rho_h(s),u_h(s),\nabla u_h(s),{\psi_h}(s)),\nabla u_h(s)\big\rangle\rd s.
	\end{split}
	\end{equation*}
	%\begin{equation*}
	%\begin{split}
	%\|\nabla u_h(t)\|^2&=\|\nabla\Pi_h g(\rho_h(T)\|^2-2\int_{t}^{T}\langle\Delta_h u_h(s),\Delta_h u_h(s)\rangle\rd s-2\sum_{i=1}^{k}\int_{t}^{T}\langle\nabla\psi_h^i(s),\nabla u_h(s)\rangle\rd W^i_s\\
	%&
	%-\sum_{i=1}^{k}\int_{t}^{T}\|\nabla\psi_h^i(s)\|^2\rd s+2\int_{t}^{T}\big\langle \nabla \Pi_h G(s,\rho_h(s),\nabla\rho_h(s),u_h(s),\nabla u_h(s),{\psi_h}(s)),\nabla u_h(s)\big\rangle\rd s,
	%\end{split}
	%\end{equation*}
	Using $\|\nabla\Pi_h g(\rho_h(T)\|\le C_e\| g(\rho_h(T)\|_{1,2}\le C_e\tilde{L}^g\bigg(1+\|\rho_h(T)\|_{1,2}\bigg)$, we have
	\begin{align}\label{lem1ref3}
	&\|\nabla u_h(t)\|^2+2\int_{t}^{T}\|\Delta_h u_h(s)\|^2\rd s+ \int_{t}^{T}\|\nabla\psi_h(s)\|^2\rd s\nonumber\\
	&\le 2|C_e\tilde{L}^g|^2+2|C_e{L^{g}}|^2\|\rho_h(T)\|_{1,2}^2-2\sum_{i=1}^{k}\int_{t}^{T}\langle\nabla\psi_h^i(s),\nabla u_h(s)\rangle\rd W^i_s\nonumber\\
	&\quad\quad+2\int_{t}^{T}\bigg|\big\langle \Pi_h G(s,\rho_h(s),\nabla\rho_h(s),u_h(s),\nabla u_h(s),{\psi_h}(s)),\Delta_h u_h(s)\big\rangle\bigg|\rd s.
	\end{align}
	In view of the Lipschitz property from Assumption \ref{ass2}, we have
	\begin{align*}
	%\begin{split}
	&2\int_{t}^{T}\big|\big\langle  \Pi_h G(s,\rho_h(s),\nabla\rho_h(s),u_h(s),\nabla u_h(s),{\psi_h}(s)),\Delta_h u_h(s)\big\rangle\big|\rd s\\
	&\le 2\int_{t}^{T}\| \Pi_h G\big(s,\rho_h(s),\nabla\rho_h(s),u_h(s),\nabla u_h(s),{\psi_h}(s)\big)\|\cdot\|\Delta_h u_h(s)\|\rd s\\
	&\le 2\int_{t}^{T}\bigg( L^{G}_1\big(\|\rho_h(s)\|+\|\nabla\rho_h(s)\|\big)+L^{G}_2\big(\|u_h(s)\|+\|\nabla u_h(s)\|+\|{\psi_h}(s)\|\big)\bigg)\cdot\|\Delta_h u_h(s)\|\rd s\\
	&\quad\quad +2\int_{t}^{T}\|G^0_s\|\cdot\|\Delta_h u_h(s)\|\rd s\\
	&\le 2\int_{t}^{T}\|G^0_s\|\cdot\|\Delta_h u_h(s)\|\rd s+2L^{G}_1\int_{t}^{T}\|\rho_h(s)\|\cdot\|\Delta_h u_h(s)\|\rd s+2L^{G}_1\int_{t}^{T}\|\nabla\rho_h(s)\|\cdot\|\Delta_h u_h(s)\|\rd s\\
	&\quad\quad+2L^{G}_2\int_{t}^{T}\|u_h(s)\|\cdot\|\Delta_h u_h(s)\|\rd s+2L^{G}_2\int_{t}^{T}\|\nabla u_h(s)\|\cdot\|\Delta_h u_h(s)\|\rd s\\&\quad\quad+2L^{G}_2\int_{t}^{T}\|{\psi_h}(s)\|\cdot\|\Delta_h u_h(s)\|\rd s\\
	&\le \int_{t}^{T}\bigg(\frac{1}{\eps_1}\|G^0_s\|^2+\frac{|L^{G}_1|^2}{\eps_2}\|\rho_h(s)\|^2+\frac{|L^{G}_1|^2}{\eps_3}\|\nabla\rho_h(s)\|^2+\frac{|L^{G}_2|^2}{\eps_4}\|u_h(s)\|^2+\frac{|L^{G}_2|^2}{\eps_6}\|\psi_h(s)\|^2\bigg)\rd s\\
	&\quad\quad+\bigg(\eps_1+\eps_2+\eps_3+\eps_4+\eps_5+\eps_6\bigg)\int_{t}^{T}\|\Delta_h u_h(s)\|^2\rd s+\frac{|L^{G}_2|^2}{\eps_5}\int_{t}^{T}\|\nabla u_h(s)\|^2\rd s.
	%\end{split}
	\end{align*}
	Taking $\eps_1=\eps_2=\eps_3=\eps_4=\eps_5=\eps_6=\frac{1}{4}$ and combining the above result with \eqref{lem1ref3} yield that
	\begin{align}\label{step2sup}
	&\|\nabla u_h(t)\|^2+\frac{1}{2}\int_{t}^{T}\|\Delta_h u_h(s)\|^2\rd s
	+ \int_{t}^{T}\|\nabla\psi_h(s)\|^2\rd s\nonumber\\
	&\le 2 |C_e\tilde{L}^g|^2+ 2|C_e{L^{g}}|^2\|\rho_h(T)\|_{1,2}^2+4\int_{t}^{T}\|G^0_s\|^2\rd s+4|L^{G}_2|^2\int_{t}^{T}\|\nabla u_h(s)\|^2\rd s\nonumber\\
	&
	\quad\quad+4\int_{t}^{T}\bigg(|L^{G}_1|^2\|\rho_h(s)\|^2+|L^{G}_1|^2\|\nabla\rho_h(s)\|^2+|L^{G}_2|^2\|u_h(s)\|^2+|L^{G}_2|^2\|\psi_h(s)\|^2\bigg)\rd s\nonumber\\
	&\quad\quad-2\sum_{i=1}^{k}\int_{t}^{T}\langle\nabla\psi_h^i(s),\nabla u_h(s)\rangle\rd W^i_s,\quad\text{a.s.}.
	\end{align}
	Then taking expectations on both sides we have in particular
	\begin{align}\label{estimatePsi_h}
	&\bE\int_{t}^{T}\|\nabla\psi_h(s)\|^2\rd s\nonumber\\
	&\le 2 \bE\bigg[|C_e\tilde{L}^g|^2+|C_e{L^{g}}|^2\|\rho_h(T)\|_{1,2}^2\bigg]+4\bE\int_{t}^{T}\|G^0_s\|^2\rd s+4|L^{G}_2|^2\bE\int_{t}^{T}\|\nabla u_h(s)\|^2\rd s\nonumber\\
	&
	\quad\quad+4\bE\int_{t}^{T}\bigg(|L^{G}_1|^2\|\rho_h(s)\|^2+|L^{G}_1|^2\|\nabla\rho_h(s)\|^2+|L^{G}_2|^2\|u_h(s)\|^2+|L^{G}_2|^2\|\psi_h(s)\|^2\bigg)\rd s.
	\end{align}
	On the other hand, taking supremum over $t\in[\tau,T]$ for $\tau\in[0,T]$ and taking expectations in \eqref{step2sup},  we have
	\begin{align}\label{lem1ref4}
	&\bE\bigg[\sup_{t\in [\tau,T]}\|\nabla u_h(t)\|^2\bigg]+\frac{1}{2}\bE\int_{\tau}^{T}\|\Delta_h u_h(s)\|^2\rd s+\bE\int_{\tau}^{T}\|\nabla\psi_h(s)\|^2\rd s\nonumber\\
	&\le 2 |C_e\tilde{L}^g|^2+2 |C_e{L^{g}}|^2\bE\bigg[\|\rho_h(T)\|_{1,2}^2\bigg]+4\bE\int_{\tau}^{T}\|G^0_s\|^2\rd s+4|L^{G}_2|^2\bE\int_{\tau}^{T}\|\nabla u_h(s)\|^2\rd s\nonumber\\
	&
	\quad\quad+4\bE\int_{\tau}^{T}\bigg(|L^{G}_1|^2\|\rho_h(s)\|^2+|L^{G}_1|^2\|\nabla\rho_h(s)\|^2+|L^{G}_2|^2\|u_h(s)\|^2+|L^{G}_2|^2\|\psi_h(s)\|^2\bigg)\rd s\nonumber\\
	&\quad\quad+2\bE\bigg|\sup_{t\in [0,\tau]}\int_{\tau}^{T}\sum_{i=1}^{k}\langle\nabla\psi_h^i(s),\nabla u_h(s)\rangle\rd W^i_s\bigg|.
	\end{align}
	Now we use BDG inequality for the terms involving stochastic integrals and obtain
	\begin{align*}
	%\begin{split}
	&2\bE\bigg[\sup_{t\in [\tau,T]}\bigg|\int_{t}^{T}\sum_{i=1}^{k}\langle\nabla\psi_h^i(s),\nabla u_h(s)\rangle\rd W^i_s\bigg|\bigg]
	%\\
	%&
	\le  \frac{1}{2} \bE\bigg[\sup_{t\in [\tau,T]}\|\nabla u_h(t)\|^2\bigg] +2\tilde{C}\bE\int_{\tau}^{T}\|\nabla\psi_1(s)\|^2\rd s,
	%\end{split}
	\end{align*}
	which together with \eqref{lem1ref4} implies that
	\begin{align*}
	%\begin{split}
	&\frac{1}{2}\bE\bigg[\sup_{t\in [\tau,T]}\|\nabla u_h(t)\|^2\bigg]+\frac{1}{2}\bE\int_{\tau}^{T}\|\Delta_h u_h(s)\|^2\rd s+\bE\int_{\tau}^{T}\|\nabla\psi_h(s)\|^2\rd s\\
	&\le 2 |C_e\tilde{L}^g|^2+4\bE\int_{\tau}^{T}\|G^0_s\|^2\rd s+2\tilde{C}\bE\int_{\tau}^{T}\sum_{i=1}^{k}\|\nabla\psi_h^i(s)\|^2\rd s+4|L^{G}_2|^2\bE\int_{\tau}^{T}\|\nabla u_h(s)\|^2\rd s\\
	&
	\quad\quad+4\bE\int_{\tau}^{T}\bigg(|L^{G}_1|^2\|\rho_h(s)\|^2+|L^{G}_1|^2\|\nabla\rho_h(s)\|^2+|L^{G}_2|^2\|u_h(s)\|^2+|L^{G}_2|^2\|\psi_h(s)\|^2\bigg)\rd s\\
	&\quad\quad+2|C_e{L^{g}}|^2\bE\bigg[\|\rho_h(T)\|_{1,2}^2\bigg].
	%\end{split}
	\end{align*}
	Then using estimates from $Assertion\;(i)$ and from \eqref{estimatePsi_h}, and applying Gronwall's inequality, we have
	\begin{align*}
	%\begin{split}
	&\frac{1}{2}\bE\bigg[\sup_{t\in [0,T]}\|\nabla u_h(t)\|^2\bigg]+\frac{1}{2}\bE\int_{0}^{T}\|\Delta_h u_h(s)\|^2\rd s+ \bE\int_{0}^{T}\|\nabla\psi_h(s)\|^2\rd s \\
	&\le C+C\bE\bigg[\|\rho_0\|_{1,2}^2+\|g^0\|^2\bigg]+C\bE\int_{0}^{T}\bigg(\|f_s^{0}\|^2+\|F^0_s\|^2+\|G^0_s\|^2\bigg)\rd s,
	%\end{split}
	\end{align*}
	with $C=C(L_1^{f},L_2^f,\tilde{L}^{f},L^{g},\tilde{L}^{g},L^{F}_1,L^{F}_2,L_1^{G},L_2^{G},T,C_e)$. Combining this with the estimate from $Step\;1$ finally gives $Assertion\;(ii)$.	
\end{proof}

%\end{appendix}

\bigskip
 
 \noindent
 \textbf{ {\large Acknowledgements}}
 
 The study on numerical analysis of FBSPDEs was kindly suggested and commented by different colleagues when the second author was working in Berlin and Michigan, and their valuable comments and suggestions are greatly appreciated. J. Qiu would also like to thank Professors Erhan Bayraktar, Kai Du, Jing Zhang, and Chao Zhou for the helpful discussions on numerical analysis or deep BSDE methods.   
 
 \bibliographystyle{siam}
%\bibliography{Project1bibfile} 

\end{document}